\documentclass[a4paper,11pt]{article} 
\title{On Thermal Transpiration and Thermomolecular Pressure Difference} 
\author{
	I-Kun Chen\\
	\textit{Institute of Applied Mathematical Sciences, National Taiwan University}\\
	\quad\\
	Kai-Li Wang\\
	\textit{Institute of Mathematics, National Taiwan University}
}\date{\today} 
\usepackage[a4paper,margin=0.8in]{geometry} 
\usepackage{framed} 
\usepackage{graphicx} 
\usepackage{enumerate} 
\usepackage[utf8]{inputenc} 
\usepackage{CJKutf8} 
\usepackage{amsmath} 
\usepackage{amsfonts} 
\usepackage{pgf,tikz} 
\usetikzlibrary{arrows} 
\usepackage{float} 
\usepackage{mathtools} 
\usepackage{mathrsfs} 
\usepackage{amsthm} 
\usepackage{bm} 
\usepackage{multicol} 
\usepackage{verbatim} 
\usepackage{flushend,cuted} 
\usepackage{amssymb} 
\usepackage{makeidx} 
\usepackage{hyperref} 
\usepackage{titlesec}
\hypersetup
{
	colorlinks=true,
	linkcolor=blue,
	filecolor=magenta,      
	urlcolor=cyan,
	pdftitle={Overleaf Example},
	pdfpagemode=FullScreen,
}
\usepackage{type1cm} 
\usepackage{tabularx} 
\usepackage{array} 
\usepackage{catchfilebetweentags} 
\newcommand{\di}{\displaystyle} 
\newtheorem{theorem}{Theorem}[section]
\newtheorem{definition}{Definition}[section]
\newtheorem{proposition}{Proposition}[section]

\newtheorem{lemma}{Lemma}[section]
\newtheorem{remark}{Remark}[section]
\newtheorem{Assumption}{Assumption}[section]

\titleformat{\section}[block]
{\large\bfseries\filcenter} 
{Section\,\thesection.}{0.3em}{}
\titlespacing*{\section}{0pt}{0pt}{10pt}

\titleformat{\subsection}[block]
{\normalsize\bfseries} 
{\thesubsection.}{0.3em}{}

\begin{document}
	\maketitle
	\tableofcontents
	\begin{center}
		\textbf{Acknowledgments}
	\end{center}
	
	The authors would like to thank Professor Kazuo Aoki and Professor Tai-Ping Liu
	for their interest and insightful suggestions. The authors would also like to thank
	Professor Chun-Hsiung Hsia and Dr. Jhe-Kuan Su for their fruitful comments.
	The research of the first author is supported in part by NSTC Grant
	112-2115-M-002-009-MY3.
	
	\begin{abstract}
		In this article, we demonstrate the phenomenon of thermal transpiration in a bounded convex domain. We employ the stationary Boltzmann equation with a cutoff potential. For boundary condition, we partition the boundary into diffuse reflection and incoming regions. We establish the existence of solution in a weighted $L^\infty$ space. Furthermore, we consider a convex domain with diffuse reflection boundary condition in the middle and incoming boundary condition at the two ends.
		We first consider Maxwellians with the same pressure but different temperatures at the two ends, $T_1=1$ and $T_2>1$. We prove that the total flux $U(x)$ is directed toward the hot end. Furthermore, we derive an estimate for the total flux:
		\begin{align}
			U(x)\geq C\left(1-\frac{1}{\sqrt{T_2}}\right).
		\end{align}
		In addition, we show that when the pressures $P_1,P_2$  and temperatures $T_1,T_2$ on the two ends satisfy the relation
		\begin{align}
			\frac{P_1}{P_2}=\sqrt{\frac{T_1}{T_2}}
		\end{align}
		the total flux of the solution is of order $\mathcal{O}(\frac{1}{\kappa})$. This result is consistent with Knudsen's finding of thermomolecular pressure difference in 1909.
	\end{abstract}
\section{Introduction}\label{section 1}
	
	\subsection{Introduction}\label{subsection 1.1}
	
	\qquad The thermal transpiration is a phenomenon that temperature gradient  on the boundary induces a flow in a pipe \cite{Sone Y}. This phenomenon has many engineering applications, including the Crookes radiometer \cite{Crookes W 1874}, which can be used in energy harvesting \cite{Strongrich A Pikus A Sebastiao I B  Alexeenko A}; the Knudsen compressor, which pumps gas without any moving mechanical part \cite{Vargo S E Muntz E P Shiflett G R Tang W C 1999}; microelectromechanical systems \cite{Passian A Warmack R J Ferrell T L  Thundat T 2003}, and MEMS-based multistage Knudsen compressor, where
	the gas moves from a cold chamber to a hot chamber creating a pressure difference across the compressor element \cite{Han Y L Phillip Muntz E Alexeenko A Young M 2007}. Further applications and details can be found in \cite{Kandlikar S 2006}, \cite{Sone Y}.\par
	
	\qquad The thermal transpiration was first observed by Reynolds in 1879 \cite{Reynolds O 1879}. He observed that gas moves along a solid surface due to temperature differences. In the same year, Maxwell developed a theory to explain this effect \cite{Maxwell J C 1879}. Moreover, when considering two chambers connected by a tube, Knudsen showed that in the limit of infinitely small tube, the pressure between the two chambers satisfies
	\begin{align}\label{1.1}
		\frac{P_1}{\sqrt{T_1}}=\frac{P_2}{\sqrt{T_2}}
	\end{align}
	in equilibrium. This phenomenon is called the thermomolecular pressure difference. This relation is discussed in \cite{Knudsen M 1909},  \cite{Knudsen M 1910}. In an experiment, S. Chu Liang provided a more accurate expression of \eqref{1.1} in \cite{Liang S C 1953}, with further applications given in \cite{Bennett M J Tompkins F C 1957}, \cite{Takaishi T Sensui Y 1963}.\par
	\qquad In kinetic theory, the stationary Boltzmann equation serves as a fundamental  mathematical model for describing the behavior of rarefied gas. It reads
	\begin{align}\label{eqeq1.2}
		v\cdot\nabla_x F(x,v)=\frac{1}{\kappa}Q(F,F)(x,v), \quad\text{in}\quad \Omega\times\mathbb{R}^3. 
	\end{align}
	Here, $\kappa$ is the Knudsen number and $F$ represents the velocity distribution function of the gas molecules. The velocity distribution function is used to describe the number of molecules $\mathrm{d}N$ in a six-dimensional phase-space element 
	$\mathrm{d}\mathbf{X}\mathrm{d}\boldsymbol{\xi}$. That is, 
	\[
	\mathrm{d}N =F(\mathbf{X}, \boldsymbol{\xi})\, \mathrm{d}\mathbf{X}\mathrm{d}\boldsymbol{\xi}.
	\]
	The collision operator $Q$ that appears in \eqref{eqeq1.2} is defined as follows.
	\begin{align*}
		&Q(F,G)(x,v)\\
		&=\int_{\mathbb{R}^3}\int_{S^2}\left[F(v')G(u')+F(u')G(v')-F(v)G(u)-F(u)G(v)\right]B(|v-u|,\omega)d\omega du,
	\end{align*}
	where 
	\begin{align}
		v'=v+((u-v)\cdot\omega)\omega \quad\text{and}\quad u'=u-((u-v)\cdot\omega)\omega.
	\end{align}
	Note that $v',u'$ can also be expressed in the equivalent form
	\begin{align}
		v'=\frac{v+u}{2}+\frac{|v-u|}{2}\sigma\quad \text{and}\quad u'=\frac{v+u}{2}-\frac{|v-u|}{2}\sigma,
	\end{align}
	where $\sigma\in\mathbb{S}^2$.\par
	\qquad Consequently, stationary Boltzmann equation has been extensively applied to describe the gas distribution in the domains involving thermal transpiration. However, in engineering applications, obtaining numerical solutions requires substantial computational efforts, either
	deterministically \cite{Su W Zhu L Wang P Zhang Y W L}, \cite{Tcheremissine F 2005 May}, \cite{Wu L Reese J M Zhang Y 2014} or stochastically \cite{Bird G A 1994}, \cite{Radtke G A Hadjiconstantinou N G  Wagner W 2011}. Some numerical analysis of thermal transpiration employ the lattice Boltzmann model \cite{Tang G H Zhang Y H Gu X J Barber R W  Emerson D R 2009}, the BGK model \cite{Guo Y 2003}, \cite{Loyalka S K 1971}, \cite{Niimi H 1971} or the linearized Boltzmann Equation model \cite{Takata S Funagane H 2011}. In particular, for the linearized Boltzmann equation model, Kazuo Aoki, Yoshio Sone. have made many significant contributions \cite{A rarefied gas flow induced by a temperature field}, \cite{Ohwada T Sone Y Aoki K 1989}. Moreover, Sone provided an explicit form of the solution when the free transport equation is used to describe the thermal transpiration and thermomolecular pressure difference \cite{Sone Y}.\par
	
	\qquad Although there are many results on the numerical analysis of thermal transpiration using the linearized Boltzmann equation, there are only a few mathematical results. 
	In 2007, Chiun-Chuan Chen, I-Kun Chen, Tai-Ping Liu and Yoshio Sone solved this problem for gas between two plates or within an infinite long pipe under diffuse reflection boundary conditions. They showed that when the Knudsen number $\kappa$ is sufficiently large, the local flow between two plates was of the order log $\kappa$ and the flow in a pipe has finite positive velocity. Moreover, they proved that in both case, the local flux is proportional to the wall temperature gradient \cite{Chen C C Chen I K Liu T P Sone Y 2007}.
	This result is used to study 
	the boundary logarithmic singularity \cite{Chen I K Liu T P Takata S 2014},\cite{Takata S Funagane H 2011}.\par
	\qquad In this paper, we address the  problem in the case where the pipe has finite length, and we impose the incoming boundary condition at both ends with Maxwellians and diffusion boundary condition in the middle.\par
	\qquad In this article, we formulate the mixed boundary problem for the Boltzmann equation as follows:
	\begin{align}\label{eq 1.1}
		\left\{
		\begin{aligned}
			v\cdot\nabla_x F&=\frac{1}{\kappa}Q(F,F),&&\quad\text{in}\quad \Omega\times\mathbb{R}^3, \\
			F(x,v)&=F_{in}(x,v), &&\quad\text{on}\quad \partial\Omega^{-}_1,\\
			F(x,v)&=m_{T(x)}\int_{u\cdot n(x)>0}F(x,u)u\cdot n(x)du, &&\quad\text{on}\quad \partial\Omega^{-}_2,
		\end{aligned}
		\right.
	\end{align}
	where $\Omega$ is a bounded convex domain in $\mathbb{R}^3$. Here, $\partial\Omega_1$ and $\partial\Omega_2$ are subsets of $\partial\Omega$ and satisfy the following assumption
	
	\begin{Assumption}\label{assumption of boundary}
		\qquad
		\begin{enumerate}
			\item $\partial\Omega_1\cup\partial\Omega_2=\partial\Omega$.
			\item $\partial\Omega_1$ and $\partial\Omega_2$ are of positive measure.
			\item $\partial\Omega_1$ and $\partial\Omega_2$ are unions of connected components of  $\partial\Omega$.
			\item $\partial\Omega_1$ is relative open to $\partial\Omega$.
			\item $\overline{\partial\Omega_1}\cap\overline{\partial\Omega_2}$ consists of finitely many $C^2$ curves in $\mathbb{R}^3$.
		\end{enumerate}
	\end{Assumption}
	
	Here $m_T=\frac{1}{2\pi T^2}e^{-\frac{|v|^2}{2T}}$ and  $m_{T(x)}=\frac{1}{2\pi T^2(x)}e^{-\frac{|v|^2}{2T(x)}}$. Also, we denote $m=m_1$.
	Let $\partial E$ is a subset of  $\partial\Omega$. We define $\partial E^{-}, \partial E^{+}$ and $\partial E^{0}$ as follows. 
	\begin{align*}
		\partial E^{-}&=\{(x,v)\in\partial E\times\mathbb{R}^3 :\, n(x)\cdot v<0\},\\
		\partial E^{+}&=\{(x,v)\in\partial E\times\mathbb{R}^3 :\, n(x)\cdot v>0\},\\
		\partial E^{0}&=\{(x,v)\in\partial E\times\mathbb{R}^3 :\, n(x)\cdot v=0\},
	\end{align*}
	where $n(x)$ denotes the outward unit normal vector as $x\in\partial\Omega$. In this article, we always assume that $T(x)$ is positive and belongs to $ L^\infty(\partial\Omega)$. In this article, we consider that $\mathbf{B}(|v-u|,\omega)$ satisfies
	\begin{align}\label{eq 1.2}
		\mathbf{B}(|v-u|,\omega)=\mathfrak{B}|v-u|^\gamma\left(\frac{u-v}{|u-v|}\cdot \omega\right),\quad -3<\gamma\leq 1,
	\end{align}
	for some constant $\mathfrak{B}>0$. When we discuss the thermal transpiration, we define two parameters as
	\begin{align*}
		A&\coloneq\inf\{x_3\in\mathbb{R}^3 | (x_1,x_2,x_3)\in\Omega\},\\
		B&\coloneq\sup\{x_3\in\mathbb{R}^3 | (x_1,x_2,x_3)\in\Omega\}.
	\end{align*}
	Let $A<a<b<B$, and define
	\begin{align}
		\partial\Omega_1=\Big(\partial\Omega\cap\{z<a\}\Big)\cup\Big(\partial\Omega\cap\{z>b\}\Big).
	\end{align}
	Finally, we set the distribution of gas $F(x,v)$ in $\Omega$ which satisfies the following boundary value problem:
	\begin{align}\label{eq 1.3}
		\left\{
		\begin{aligned}
			v\cdot\nabla_x F&=\frac{1}{\kappa}Q(F,F),&&\text{in}\quad \Omega\times\mathbb{R}^3, \\
			F(x,v)&=\frac{1}{(2\pi)^{3/2}}e^{-\frac{|v|^2}{2}},&& \text{on}\quad \partial \Omega^-_{1,a}=\partial\Omega^-_1\cap\{x_3<a\},\\
			F(x,v)&=\frac{1}{(2\pi)^{3/2}T^{5/2}_2}e^{-\frac{|v|^2}{2T_2}},&&\text{on}\quad \partial \Omega^-_{1,b}=\partial\Omega^-_1\cap\{x_3>b\},\\
			F(x,v)&=m_T(x,v)\int_{u\cdot n(x)>0}F(x,u)u\cdot n(x)du,&&\text{on}\quad \partial \Omega^-_2=\partial \Omega^-\setminus\left(\partial \Omega^-_{1,a}\cup \partial \Omega^-_{1,b}\right).
		\end{aligned}
		\right.
	\end{align}
	Furthermore, we impose that $T(x_1,x_2,x_3)=1$ when $x_3\leq a$, and that $T(x_1,x_2,x_3)=T_2$ when $x_3\geq b$. We write the solution as a perturbation near the standard Maxwellian $M(v)=\frac{1}{\sqrt{2\pi}}m(v)$. Let $F=M+m^{1/2}f$ and $F_{in}=M+m^{1/2}f_{in}$. Then, the function $f$ solves the following boundary value problem:
	\begin{align}\label{eq 1.1.1}
		\left\{
		\begin{aligned}
			v\cdot\nabla_x f&=-\frac{1}{\kappa}L(f)+\frac{1}{\kappa}\Gamma(f,f),&&\quad\text{in}\quad \Omega\times\mathbb{R}^3, \\
			f(x,v)&=f_{in}(x,v), &&\quad\text{on}\quad \partial\Omega^{-}_1,\\
			f(x,v)&=m^{-1}m_{T(x)}P(f)+\frac{1}{\sqrt{2\pi}}m^{-1/2}(m_{T(x)}-m), &&\quad\text{on}\quad \partial\Omega^{-}_2,
		\end{aligned}
		\right.
	\end{align}
	where 
	\begin{align*}
		L(f)&=m^{-1/2}\Big(Q(M,m^{1/2}f)+Q(m^{1/2}f,M)\Big),\\
		\Gamma&=m^{-1/2}Q(m^{1/2}f,m^{1/2}f),\\
		P(f)(x,v)&=m^{1/2}\int_{u\cdot n(x,v)>0}f(x,u)m^{1/2}u\cdot n(x)du.
	\end{align*}
	The operator $L$ has two components: the multiplicative operator $\nu$ and integral operator $K$, that is , $L(f)=\nu(v)f-K(f)$. Their definition  can be found in \textbf{section \ref{section 2}} or \cite{Glassey Robert T 1996}.
	In what follows, we discuss the well-posedness of the boundary value problem \eqref{eq 1.1.1}. The integral form of $\eqref{eq 1.1.1}$ is
	\begin{align}
		f=\frac{1}{\kappa}S^\kappa_{\Omega}Kf+J^{\kappa}\left(f_{in}\textbf{1}_{\partial\Omega^{-}_1}+\left(m^{-1}m_{T(x)}P(f)+\frac{1}{\sqrt{2\pi}}m^{-1/2}(m_{T(x)}-m)\right)\textbf{1}_{\partial\Omega^{-}_2} \right)+\frac{1}{\kappa}S^\kappa_{\Omega}\Gamma(f,f).
	\end{align}
	Here, we define
	\begin{align}
		S_\Omega^\kappa f(x,v)&\coloneq\int_{0}^{\tau_{-}(x,v)}e^{-\frac{1}{\kappa}\nu(v)s}f(x-sv,v)ds,\\
		J^\kappa f&\coloneq e^{-\frac{1}{\kappa}\nu(v)\tau_{-}(x,v)}f(q(x,v),v),
	\end{align}
	where
	\begin{align}
		\tau_{-}(x,v)=\sup\{t\, | \, x-tv\in\overline{\Omega}\} \quad\text{and}\quad q(x,v)=x-\tau_{-}(x,v)v.
	\end{align}
	For convenience, we define
	\begin{align}
		\langle f, g\rangle&=\int_{\Omega\times\mathbb{R}^3}fgdxdv,\quad \langle f, g\rangle_v=\int_{\mathbb{R}^3}fgdv,\\
		\langle f, g\rangle_{\partial^-_2}&=\int_{\partial\Omega^-_2}fg |v\cdot n(x)|dxdv.
	\end{align}
	To state our main theorem, we introduce some function spaces. For $\alpha\geq 0$, $\beta\in\mathbb{R}$ and $p\geq 1$, we say $f\in L^p_{\alpha,\beta}(\Omega\times\mathbb{R}^3)$ if
	\begin{align*}
		\begin{aligned}
			\lVert f\rVert_{L^p_{\alpha,\beta}}=\lVert f\rVert_{L^p_{\alpha,\beta}(\Omega\times\mathbb{R}^3)}
			&\coloneq\left(\int_{\Omega}\int_{\mathbb{R}^3}\left|f(x,v)e^{\alpha|v|^2}(1+|v|)^{\beta}\right|^pdvdx\right)^{1/p}<\infty \quad &&\text{for }p<\infty,\\
			\lVert f\rVert_{L^\infty_{\alpha,\beta}}=\lVert f\rVert_{L^\infty_{\alpha,\beta}(\Omega\times\mathbb{R}^3)}&\coloneq \underset{{(x,v)\in\Omega\times\mathbb{R}}}{\text{ess sup}}\left|f(x,v)e^{\alpha|v|^2}(1+|v|)^{\beta}\right|<\infty &&\text{for } p=\infty.
		\end{aligned}
	\end{align*}
	Moreover, we say $f\in L^p_{\alpha,\beta}(\partial E^-)$  for some $\partial E\subset\partial\Omega$, if
	\begin{align*}
		\begin{aligned}
			\lVert f\rVert_{L^p_{\alpha,\beta}(\partial E^-)}
			&\coloneq\left(\int_{\partial E^-}\left|f(x,v)e^{\alpha|v|^2}(1+|v|)^{\beta}\right|^p|v\cdot n(x)|dSdv\right)^{1/p}<\infty \quad &&\text{for }p<\infty,\\
			\lVert f\rVert_{L^\infty_{\alpha,\beta}(\partial E^-)}&\coloneq \underset{{(x,v)\in\partial E^-}}{\text{ess sup}}\left|f(x,v)e^{\alpha|v|^2}(1+|v|)^{\beta}\right|<\infty &&\text{for } p=\infty,
		\end{aligned}
	\end{align*}
	where $dS$ denotes the surface measure on $\partial\Omega$. For convenience, we write $L^p(A)=L^p_{0,0}(A)$ when $A=\Omega\times\mathbb{R}^3\, \text{or }\partial E^-\subseteq\partial\Omega\times\mathbb{R}^3$. Throughout this article, we denote $f(x)=\mathcal{O}(g(x))$ when $\lim_{x\to\infty}\frac{f(x)}{g(x)}$ is a constant.\par
	\qquad We now introduce the concept of an unsealed set $\Omega$ in $\mathbb{R}^3$. The term \emph{unseal} describes the visibility of $\partial\Omega_1$ from points on $\partial\Omega$. If every point $x$ can directly see the $\partial\Omega_1$, then we call $\Omega$ a \textbf{completely unsealed set}. If, for every $x\in\partial\Omega$, the $\partial\Omega_1$ becomes visible after tracing back $m$ times for some fixed integer $m$, then we call $\Omega$ an \textbf{unsealed set}. The precise definition is in \textbf{section \ref{section 4}}.
	\begin{theorem}\label{theorem 1.1}
		Let $\Omega\subset\mathbb{R}^3$ be a bounded convex domain with $C^2$ boundary. Suppose that $\Omega$ is unsealed with respect to $\partial\Omega_1$ and satisfies the \textbf{assumption \ref{assumption of boundary}}. Assume that $-3<\gamma\leq 1$ and that \eqref{eq 1.2} holds. Then, for any $0< \alpha <1/4$ and $\kappa_0>0$, there exists $\delta>0$ which depends on $\alpha$, $\Omega$, $\partial\Omega_1$, and $\kappa_0$ such that if 
		\begin{align}
			\lVert m^{-1/2}\left(F_{in}-M\right)\rVert_{L^\infty_{a,0}(\partial\Omega^-_1)}\leq \delta
		\end{align}
		and 
		\begin{align}
			|T(x)-1|\leq \delta,\quad \forall x\in \Omega_2
		\end{align}
		then the equation \eqref{eq 1.1} has a unique mild solution for $\kappa\geq \kappa_0$. Furthermore,
		\begin{align}\label{eq 1.9}
			&\lVert m^{-1/2}\left(F-M\right)\rVert_{L^\infty_{\alpha,0}}+\lVert m^{-1/2}\left(F-M\right)\rVert_{L^\infty_{\alpha,0}(\partial\Omega^+_2)}\\\nonumber
			&\leq 2C_{\zeta}\Big(\left\lVert m^{-1/2}\left(F_{in}-M\right)\right\rVert_{L^\infty_{\alpha,0}(\partial\Omega^-_1)}+\left\lVert m^{-1/2}(m_{T(x)}-m)\right\rVert_{L^\infty_{\alpha,0}(\partial\Omega^-_2)}\Big).
		\end{align}
	\end{theorem}	
	\qquad Notice that, when $\kappa>\kappa_0$, $\delta$ and \eqref{eq 1.9} is independent of $\kappa$. This provides a foundation to discuss thermal transpiration and thermomolecular pressure difference.

	\qquad For $x=(x_1,x_2,z)\in\Omega$, the local flux $u(x)$ and the total flux $U(z)$ in the $x_3$-direction are defined as follows:
	\begin{align}\label{eq 1.15}
		u(x)\coloneq\int_{\mathbb{R}^3}v_3F(x,v)dv,\quad U(z)\coloneq \int_{\{(x_1,x_2,z)\in\Omega\}}u(x) dS.
	\end{align}
	By conservation law, the total flux is a constant with respect to $z$ when $z\in [a,b]$.
	\begin{theorem}\label{theorem 1.2}
		Let $\Omega\subset\mathbb{R}^3$ be a bounded domain with $C^1$ boundary such that $\Omega\cap\{a\leq x_3\leq b\}$ is convex. Suppose for every $\kappa>\kappa_0>0$, there exists a constant $\mathcal{C}$ such that \eqref{eq 1.3} admits a mild solution with the following estimate:
		\begin{align}
			\lVert F\rVert_{L^\infty_{0,\beta}(\Omega\times\mathbb{R}^3)}\leq \mathcal{C},
		\end{align}
		where $\beta>\max\{\gamma,0\}+5$. Then, there exists $\kappa_1>\kappa_0$ such that the total flux $U(x_3)$ of $F$ is positive whenever $\kappa\geq\kappa_1$. Moreover,
		\begin{align}
			U(x_3)\geq C\left(1-\frac{1}{\sqrt{T_2}}\right).
		\end{align}  
	\end{theorem}
	\qquad In Knudsen's experiment \cite{Knudsen M 1909}, he connects two chambers with different temperature $T_1,T_2$ with a narrow tube. He measures the pressures in two tanks $P_1,P_2$ at equilibrium when there is no flux in the tube. He observes in the limit of infinitely small tubing $\frac{P_1}{\sqrt{T_1}}=\frac{P_2}{\sqrt{T_2}}$.
	For the free transport equation case,  
	the relation is exact \cite{Sone Y}. In our case, we obtain the same result by considering the following boundary value problem:
	\begin{align}\label{eq 1.4}
		\left\{
		\begin{aligned}
			v\cdot\nabla_x F&=\frac{1}{\kappa}Q(F,F),&&\text{in}\quad \Omega\times\mathbb{R}^3, \\
			F(x,v)&=\frac{1}{(2\pi)^{3/2}}e^{-\frac{|v|^2}{2}},&& \text{on}\quad \partial \Omega^-_{1,a}=\partial\Omega^-_1\cap\{x_3<a\},\\
			F(x,v)&=\frac{1}{(2\pi)^{3/2}T^2_2}e^{-\frac{|v|^2}{2T_2}},&&\text{on}\quad \partial \Omega^-_{1,b}=\partial\Omega^-_1\cap\{x_3>b\},\\
			F(x,v)&=m_T(x,v)\int_{v\cdot n(x)>0}F(x,v)v\cdot n(x)dv,&&\text{on}\quad \partial \Omega^-_2=\partial \Omega^-\setminus\left(\partial \Omega^-_{1,a}\cup \partial \Omega^-_{1,b}\right).
		\end{aligned}
		\right.
	\end{align}
	\begin{theorem}\label{corollary 1.1}
		Let $\Omega\subset\mathbb{R}^3$ be a bounded domain with $C^1$ boundary such that $\Omega\cap\{a\leq x_3\leq b\}$ is convex. Suppose that for every $\kappa>\kappa_0$, equation \eqref{eq 1.4} admits a mild solution bounded in $L^\infty_{0,\beta}$ by a constant independent of $\kappa$, where $\beta>\max\{\gamma,0\}+5$. Then, the total flux $U(x_3)$ of $F$ satisfies the following estimate:
		\begin{align}
			|U(x_3)|=\mathcal{O}(\frac{1}{\kappa}).
		\end{align}
	\end{theorem}
	\begin{remark}
		We can extend \textbf{Theorem \ref{corollary 1.1}} to the case when the incoming data is not localized at two places, i.e.,
		\begin{align}\label{eq 1.5}
			\left\{
			\begin{aligned}
				v\cdot\nabla_x F&=\frac{1}{\kappa}Q(F,F),&&\text{in}\quad \Omega\times\mathbb{R}^3, \\
				F(x,v)&=\frac{1}{(2\pi)^{3/2}T^2_i}e^{-\frac{v^2}{2}},&& \text{on}\quad \partial B^-_i, i=1,\cdots,n,\\
				F(x,v)&=m_T(x,v)\int_{v\cdot n(x)>0}F(x,v)v\cdot n(x)dv,&&\text{on}\quad \partial \Omega^-_2=\partial \Omega^-\setminus\left(\cup_{i=1}^n \partial B^-_i\right),
			\end{aligned}
			\right.
		\end{align}
		where $\partial B_i$ is an open subset of $\partial\Omega$ and we assume that $T(x)=T_i$ if $x_i\in\partial B_i$.
	\end{remark}
	\begin{remark}
		Note that the condition $\Omega\cap\{a\leq x_3\leq b\}$ is convex ensures that $\Omega$ is an unsealed set with respect to 
		\begin{align*}
			\Big(\partial\Omega\cap\{z<a\}\Big)\cup\Big(\partial\Omega\cap\{z>b\}\Big).
		\end{align*}
		The proof of \textbf{Theorem \ref{theorem 1.2}} and \textbf{Theorem \ref{corollary 1.1}} is valid for a domain in which $\Omega\cap\{a\leq x_3\leq b\}$ is not convex but is unsealed with respect to $\Big(\partial\Omega\cap\{z<a\}\Big)\cup\Big(\partial\Omega\cap\{z>b\}\Big)$.  
	\end{remark}
	\begin{remark}
		Note that when $\Omega$ is a bounded convex domain with $C^2$ boundary. The \textbf{Theorem \ref{theorem 1.1}} ensures the condition of \textbf{Theorem \ref{theorem 1.2}} and \textbf{Theorem \ref{corollary 1.1}}.
	\end{remark}
	\subsection{Organization of the article}\label{subsection 1.2}
	\qquad We briefly outline the proof in this article. We use the traditional $L^2\text{--}L^\infty$ estimate together with the bootstrap method to construct a solution of \eqref{eq 1.1}.\par 
	\qquad In \textbf{section \ref{section 2}}, we review some elementary properties of the linear operator $L$. In \textbf{section \ref{section 3}}, inspired by \cite{Duan R 2019}, \cite{Esposito R 2013}, and \cite{Kawagoe Daisuke 2025}, we introduce a  simpler and more efficient approach to obtaining the $L^2$ estimate 
	This approach is based on the operator $S_\Omega^\kappa K$, which is compact in $L^2_{0,\gamma/2}$.\par
	\qquad In \textbf{section \ref{section 4}}, we use the unsealed condition to obtain the $L^\infty$ estimate. Then, we prove \textbf{Theorem \ref{theorem 1.1}} in this section. The concept of the unseal set provides several advantages, such as reducing the weight requirement in $L^\infty$. Moreover, this is a key tool in proving \textbf{Theorem \ref{theorem 1.2}} and \textbf{Theorem \ref{corollary 1.1}}. \par
	\qquad In \textbf{section \ref{section 5}}, we use original Boltzmann equation to address the problem by modeling the distribution as in \eqref{eq 1.3}, with incoming data at $\Big(\partial\Omega\cap\{z<a\}\Big)$ and $\Big(\partial\Omega\cap\{z>b\}\Big)$. The benefit of linearity heavily used in \cite{Chen C C Chen I K Liu T P Sone Y 2007} does not apply to our case. However, we observe that the crucial mechanism for the appearance of flux lies in the different between the incoming data and the solution of the free transport model, which satisfies \eqref{eq 1.2}. By combining this observation with the concept of the unsealed set, we complete the proof of \textbf{Theorem \ref{theorem 1.2}} and \textbf{Theorem \ref{corollary 1.1}}.\par
	\qquad In \textbf{Section \ref{section 6}}, we discuss the $H^2$ estimate for the elliptic problem with Robin boundary conditions, which is used in the proof of \textbf{Proposition \ref{proposition 2.7}}.\par
	\qquad In \textbf{Section \ref{section 7}}, we provide a sufficient condition for a set to be unsealed.

\section{Preliminaries}\label{section 2}
	\qquad In this section, we review some properties of the linear operator $L$ and its components: the multiplicative operator $\nu$ and the integral operator, $K$. We also present several propositions about the operator $S^\kappa_\Omega$ and $J^\kappa$, which are useful in \textbf{section \ref{section 3}} and \textbf{section \ref{section 4}}. In addition, we recall some basic properties of $Q_+$ and $Q_{-}$.\par
	\qquad We emphasize that we linearize around Maxwellian $M=\frac{1}{(2\pi)^{3/2}}e^{-\frac{|v|^2}{2}}$. There is another popular convention to linearize around $m=\frac{1}{\pi^{3/2}}e^{-|v|^2}$, for example \cite{Kawagoe Daisuke 2025}, \cite{Sone Y}.
	Thus, the exponents differ by a factor of 2 and the constants are  different.
	\subsection{Estimates for $\nu$ and $K$}
	\qquad Recall that $L(f)=\nu(v)f-K(f)$ when the assumption \eqref{eq 1.2} holds, and that $\nu(v)$ and $K(f)$ have explicit forms as follows
	\begin{align*}
		\nu(v)&=\iint_{\mathbb{R}^3\times\mathbb{S}^2}B(|v-u|,\sigma)m(u)dud\sigma,\\
		K(f)(x,v)&=\int_{\mathbb{R}^3}k(v,u)f(x,u)du=\int_{\mathbb{R}^3}\big(k_1(v,u)-k_2(v,u)\big)f(x,u)du.
	\end{align*}
	The explicit formulas for $k_1$ and $k_2$ are:
	\begin{align*}
		k_1(v,u)&=B_0\pi|v-u|^{\gamma}e^{-\frac{1}{4}(|v|^2+|u|^2)},\\
		k_2(v,u)&=\frac{2\mathfrak{B}}{\pi^{3/2}}\frac{1}{|v-u|}e^{-\frac{1}{8}|v-u|^2-|V_1(v,u)|^2}\times\int_{W_{v-u}}e^{-|w+V_2(v,u)|^2}\left(|v-u|^2+|w|^2\right)^\frac{\gamma}{2}dw,
	\end{align*} 
	where
	\begin{align*}
		V_1(v,u)&\coloneq \frac{1}{2\sqrt{2}}\frac{|v|^2-|u|^2}{|v-u|^2}(v-u),\\
		V_2(v,u)&\coloneq\frac{(v-u)\times(v\times u)}{\sqrt{2}|v-u|^2},\\
		W_{v-u}&\coloneq\{w\in\mathbb{R}^3| w\perp(v-u)\}.
	\end{align*}
	Notice that we have the following identity
	\begin{align*}
		|V_1(v,u)|^2+|V_2(v,u)|^2&=\frac{1}{8}|v+u|^2,\\
		V_2(v,u)\cdot(v-u)&=0,\\
		\frac{1}{8}|v-u|^2+|V_1(v,u)|^2+|V_2(v,u)|^2&=\frac{1}{4}(|v|^2+|u|^2).
	\end{align*}
	For their derivation, see \cite{Duan R 2017}. We first give some estimates for the collision frequency $\nu(v)$.
	\begin{proposition}\label{proposition 2.1}
		Suppose that \eqref{eq 1.2} holds and $-3<\gamma\leq 1$. For the collision frequency $\nu(v)$, there exist positive constants $\nu_0$ and $\nu_1$ such that
		\begin{align*}
			\nu_0(1+|v|)^{\gamma/2}\leq \nu(v) \leq \nu_1(1+|v|)^{\gamma/2}
		\end{align*}
		for all $v\in\mathbb{R}^3$.
	\end{proposition}
	For the proof of \textbf{Proposition \ref{proposition 2.1}}, see \cite{Duan R 2017}, \cite{Glassey Robert T 1996}.

	\begin{proposition}\label{proposition 2.2}
		Let $-3<\gamma\leq 1$ and let $\Omega$ be a bounded domain in $\mathbb{R}^3$. Suppose \eqref{eq 1.2} holds. Then, we have
		\begin{align*}
			\int_{0}^{\frac{L}{|v|}}e^{-\frac{1}{\kappa}\nu(|v|)t}dt\leq C_{\gamma,L} \frac{\kappa}{1+|v|}
		\end{align*}
		for any fix $L>0$ and $(x,v)\in\Omega\times\mathbb{R}^3$.
	\end{proposition}
	The proof of \textbf{Proposition \ref{proposition 2.2}} can be found in \cite{Kuan-Hsiang Wang 2024}. We next introduce some estimate for integral operator $K(f)$.
	\begin{proposition}\label{proposition 2.3}
		Let $-3<\gamma\leq 1$ and $0<\delta<1$. Suppose that \eqref{eq 1.2} holds. Then, we have
		\begin{align}\label{proposition eq 2.1}
			|k(v,u)|\leq C_{\gamma,\delta}\frac{w_\gamma(|v-u|)}{(1+|v|+|u|)^{1-\gamma}}E_\delta(v,u)
		\end{align}
		for all $v,u\in\mathbb{R}^3$, where 
		\begin{align*}
			E_\delta(v,u)\coloneq e^{-\frac{1-\delta}{8}\left(|v-u|^2+4|V_1(v,u)|^2\right)}=e^{-\frac{1-\delta}{8}\left(|v-u|^2+\left(\frac{|v|^2-|u|^2}{|v-u|}\right)^2\right)}
		\end{align*}
		and 
		\begin{align}
			w_\gamma(|v-u|)\coloneq\left\{
			\begin{aligned}
				&\frac{1}{|v-u|}, &&-1<\gamma\leq 1, \\
				&\frac{|\ln|v-u||+1}{|v-u|}, &&\gamma=-1,\\
				&\frac{1}{|v-u|^{|\gamma|}}, &&-3<\gamma < -1.
			\end{aligned}
			\right.
		\end{align}
	\end{proposition}
	The detail of proof can be found in \cite{Duan R 2017} and the appendix of \cite{Kawagoe Daisuke 2025}.\par
	\qquad In the \textbf{section \ref{section 4}}, we estimate $L^\infty_{\alpha,0}$ norm of the solution of Boltzmann equation. Hence, we need to discuss
	\begin{align*}
		K_\alpha(f)\coloneq w(v)K(w^{-1}f)=\int_{\mathbb{R}^3}w(v)k(v,u)w^{-1}(u)f(x,u)du\coloneq\int_{\mathbb{R}^3}k_\alpha(v,u)f(x,u)du
	\end{align*} 
	for $w(v)=e^{\alpha|v|^2}$, $\alpha>0$. It is known in \cite{Chen I-Kun 2024} that
	\begin{align*}
		E_\delta(v,u)&=e^{a|v|^2}e^{-\alpha_{1,a,\delta}|v-u|^2}e^{-(1-\delta)\left(\frac{(v-u)\cdot v}{|v-u|}-\alpha_{2,a,\delta}|v-u|\right)^2}e^{-a|u|^2}\\
		&=e^{-a|v|^2}e^{-\alpha_{1,a,\delta}|v-u|^2}e^{-(1-\delta)\left(\frac{(v-u)\cdot v}{|v-u|}+\alpha_{2,a,\delta}|v-u|\right)^2}e^{a|u|^2}
	\end{align*} 
	for $-1/4 <a <1/4$, $\delta>0$ and $v,u\in\mathbb{R}^3$, where
	\begin{align*}
		\alpha_{1,a,\delta}&\coloneq \frac{(1-\delta)^2-4a^2}{8(1-\delta)},\\
		\alpha_{2,a,\delta}&\coloneq \frac{1-\delta-2a}{2\sqrt{2}(1-\delta)}.
	\end{align*}
	Plugging these identities into that estimate \eqref{proposition eq 2.1} with $a=\alpha$ and $\delta=1/2-\alpha$, we obtain
	\begin{align}\label{eq 2.3}
		|k_\alpha(v,u)|\leq C_{\gamma,\alpha} \frac{w(|v-u|)}{(1+|v|+|u|)^{1-\gamma}}e^{-\alpha_1|v-u|^2}e^{-\alpha_3\left(\frac{(v-u)\cdot v}{|v-u|}-\alpha_2|v-u|\right)^2}
	\end{align}
	for some $\alpha_1,\alpha_2,\alpha_3>0$ depending only on $\alpha$. Based on the above calculations, we obtain the following estimates, which are related to the kernel of the operator $K_\alpha$.
	\begin{proposition}\cite{Caflisch ; Russel E.}\label{proposition 2.4}
		Let $\alpha_1,\alpha_3>0$ and $\alpha_2\in\mathbb{R}^3$. Then, we have
		\begin{align}
			\int_{\mathbb{R}^3}\frac{1}{|v-u|^\mu}e^{-\alpha_1|v-u|^2}e^{-\alpha_3\left(\frac{(v-u)\cdot v}{|v-u|}-\alpha_2|v-u|\right)^2}\leq \frac{C_{\alpha,\mu}}{1+|v|}
		\end{align}
		for any $\mu<3$ and $v\in\mathbb{R}^3$.
	\end{proposition}
	\begin{proof}
		Let $w=v-u$. By using the sphere coordinate and setting $\omega=|v|\cos\theta-\alpha_2 r$, we have
		\begin{align*}
			\int_{\mathbb{R}^3}\frac{1}{|v-u|^\mu}e^{-\alpha_1|v-u|^2}e^{-\alpha_3\left(\frac{(v-u)\cdot v}{|v-u|}-\alpha_2|v-u|\right)^2}du&=\int_{0}^\infty\int_{0}^{2\pi}\int_{0}^{\pi}\frac{1}{r^{\mu-2}}e^{-\alpha_1r^2}e^{-\alpha_3\left(|v|\cos\theta-\alpha_2r\right)^2}\sin\theta d\theta d \phi dr\\
			&=\frac{2\pi}{|v|}\int_{0}^\infty\int_{-|v|-\alpha_2 r}^{|v|-\alpha_2r}\frac{1}{r^{\mu-2}}e^{-\alpha_1r^2}e^{-\alpha_3 w^2}d\omega \phi dr\\
			&\leq C_{\alpha_1,\alpha_3}\frac{1}{|v|}.
		\end{align*}
		Moreover, we can deduce that
		\begin{align*}
			\int_{\mathbb{R}^3}\frac{1}{|v-u|^\mu}e^{-\alpha_1|v-u|^2}e^{-\alpha_3\left(\frac{(v-u)\cdot v}{|v-u|}-\alpha_2|v-u|\right)^2}\leq C_\alpha.
		\end{align*}
		Here, the constant $C_\alpha$ is independent to $v$. Combining these results, the proof is complete.
	\end{proof}
	
	\begin{proposition}\cite{Kawagoe Daisuke 2025}\label{proposition 2.5}
		Let $0\leq \alpha <1/4, \beta\in\mathbb{R}^3, -3<\gamma\leq 1$ and $1\leq q <\min\{3,\frac{3}{|\gamma|}\}$. Suppose that \eqref{eq 1.2} holds. Then we have
		\begin{align}
			\di\int_{\mathbb{R}^3}|k_\alpha(v,u)|^q(1+|u|)^\beta d u\leq C_{\alpha,\gamma}(1+|v|)^{\beta+q(\gamma-1)-1}
		\end{align}
		for all $v\in\mathbb{R}^3.$
	\end{proposition}
	\begin{remark}
		We set $\min\{3,\frac{3}{|\gamma|}\}=3$ when $\gamma=0$
	\end{remark}
	The proof can be found in \cite{Kawagoe Daisuke 2025}. For convenience, we present the proof here.
	\begin{proof}
		\textbf{Case 1 $\beta\geq 0$}:\\
		Notice that $(1 + |u|)^{\beta} \leq C_{\beta} \left((1 + |v|)^{\beta} + |v - u|^{\beta}\right) $ for all $v,u\in\mathbb{R}^3$. By \eqref{eq 2.3}, we have
		\begin{align*}
			|k_\alpha(v,u)|^q(1+|u|)^\beta&\leq C(1+|v|)^{\beta+q(\gamma-1)}w_\gamma(|v-u|)^qe^{-q\alpha_1|v-u|}e^{-q\alpha_3\left(\frac{(v-u)\cdot v}{|v-u|}-\alpha_2|v-u|\right)^2}\\\nonumber
			&\quad+C(1+|v|)^{q(\gamma-1)}w_\gamma(|v-u|)^q|v-u|^\beta e^{-q\alpha_1|v-u|}e^{-q\alpha_3\left(\frac{(v-u)\cdot v}{|v-u|}-\alpha_2|v-u|\right)^2}. 
		\end{align*}
		Observe that there exists $\mu_{\gamma,q}<3$ such that $w_\gamma(|v-u|)^q\leq |v-u|^{-\mu_{\gamma,q}}$. The constant $\mu_{\gamma,q}$ only depend on $\gamma$ and $q$. Also, we note  
		\begin{align}
			w_{-1}(|v-u|)e^{-\alpha_1|v-u|^2}\leq C_\epsilon|v-u|^{-(1+\epsilon)}
		\end{align}
		for some small positive constant $\epsilon$. Thus, we can treat the case $\gamma=-1$ in the same way as for $-3<\gamma<-1$. In any case, by \textbf{proposition \ref{proposition 2.4}} we have
		\begin{align*}
			\int_{\mathbb{R}^3}|k_\alpha(v,u)|^q(1+|u|)^\beta d u&\leq C_{\alpha,\gamma}(1+|v|)^{q(\gamma-1)-1}+C_{\alpha,\gamma}(1+|v|)^{\beta+q(\gamma-1)-1}\\\nonumber
			&\leq C_{\alpha,\gamma}(1+|v|)^{q(\gamma-1)-1}. 
		\end{align*}
		\textbf{Case 2 $\beta < 0$}:\\
		Observe that 
		\begin{align*}
			(1+|u|)^\beta&=(1+|v|)^\beta\frac{(1+|v|)^{|\beta|}}{(1+|u|)^{|\beta|}}\\\nonumber
			&\leq (1+|v|)^\beta\frac{ (1 + |u|)^{|\beta|} + |v - u|^{|\beta|}}{(1+|u|)^{|\beta|}}\\\nonumber
			&\leq (1+|v|)^\beta\left(1+|v-u|^{|\beta|}\right).
		\end{align*}
		We apply the above argument to obtain the same estimate. Thus, the proof is complete.
	\end{proof}
	\subsection{Proposition of $L$}
	Next, we review some basic properties of $L$, including the spectral gap and the projection of $L$. Recall that the normal orthogonal basis of the null space Ker $L$ is given by 
	\begin{align*}
		\Phi_{0} &= (2\pi)^{-\tfrac{3}{4}} \, m^{\tfrac{1}{2}}(v), \\
		\Phi_{i} &= (2\pi)^{-\tfrac{3}{4}} v_{i}m^{\tfrac{1}{2}}(v), \quad i=1,2,3, \\
		\Phi_{4} &= \frac{(2\pi)^{-\tfrac{3}{4}}}{\sqrt{6}} \, (|v|^{2}-3)m^{\tfrac{1}{2}}(v).
	\end{align*}
	For each $f=f(v)\in L^2(\mathbb{R}^3_v)$, we denote the macroscopic part $P_L(f)$ as the projection of $f$ onto Ker $L$, that is, 
	\begin{align*}
		P_L(f)=\sum_{i=0}^{4}\langle f,\Phi_{i} \rangle_v\Phi_{i}, 
	\end{align*}
	and further denote $(I-P_L)(f)=f-P_L(f)$ to be the microscopic part of $f$. It is well-known that there exists a constant $C_L$ such that
	\begin{align}
		\langle L(f),f \rangle_v\geq C_L\lVert (I-P_L)(f)\rVert^2_{L^2_{0,\gamma/2}(dv)}.
	\end{align}
	Here, $L^2_{0,\gamma/2}(dv)$ denotes integration only with respect to the variable $v$. The proof can be found in \cite{Glassey Robert T 1996}, \cite{Guo Y 2003} and \cite{Mouhot Clément 2007}.\par
	 \qquad We present a key estimate concerning $P_L$ and $(I-P_L)$. This kind of estimate first appeared in the pure diffuse reflection boundary value problem. In that setting, Esposito, Guo, Kim, and Marra proved that the macroscopic part $P_L(f)$
	can be controlled by the microscopic part, together with the source term and the boundary term $(I-P)f$ \cite{Esposito R 2013}. The difference is that, in our problem, the boundary condition is more complicated, so we need to adapt this idea to our setting.
	\begin{proposition}\label{proposition 2.7}
		Let $\Omega$ be a bounded connected domain in $\mathbb{R}^3$ with $C^2$ boundary and $-3<\gamma\leq 1$. Suppose that $f\in L^2_{0,\gamma/2}$ is a solution to 	
		\begin{align}\label{Linear Boltz}
			\left\{
			\begin{aligned}
				v\cdot\nabla_x f+\frac{1}{\kappa}L(f)&=\phi_1,&&\quad\text{in}\quad \Omega\times\mathbb{R}^3, \\
				f&=f_{in}, &&\quad\text{on}\quad \partial\Omega^{-}_1,\\
				f&=P(f)+\phi_2(x,v), &&\quad\text{on} \quad\partial\Omega^{-}_2
			\end{aligned}
			\right.
		\end{align}
		in distribution sense. Then,
		\begin{align}\label{eq 2.9}
			\lVert P_L(f)\rVert^2_{L^2_{0,\gamma/2}}&\leq C
			\begin{pmatrix*}[l]
				&\left[1+\frac{1}{\kappa}\right]\lVert (I-P_L)(f)\rVert^2_{L^2_{0,\gamma/2}}+\lVert \phi_1\rVert^2_{L^2_{0,-\gamma/2}}+\lVert (I-P)f\rVert^2_{L^2(\partial\Omega^{+}_2)}\\
				&+\lVert f\rVert^2_{L^2(\partial\Omega^{+}_1)}+\lVert f_{in}\rVert^2_{L^2(\partial\Omega^{-}_1)}
				+\lVert\phi_2\rVert^2_{L^2(\partial\Omega^{-}_2)}
			\end{pmatrix*}. 
		\end{align}
	\end{proposition}
	
	\begin{proof}
		Recall $P_L(f)=\sum_{i=0}^{4}\langle f,\Phi_{i} \rangle_v\Phi_{i}$. Thus, it suffices to compute $\langle f,\Phi_{i} \rangle_v$, for $i=0,\cdots,4$. The estimate of $\langle f,\Phi_{i} \rangle_v$ are the same as those given in \textbf{lemma 3.3} in \cite{Esposito R 2013} for $i=1,\cdots,4$. The standard elliptic estimates appearing in \cite{Esposito R 2013} can also be found in the section 2 of \cite{Grisvard Pierre.} or section 6 of \cite{Evans Lawrence C 2022}.
		For $i=0$, due to the difference between the mixed  and diffuse boundary conditions we need to find a different test function to estimate $\langle f,\Phi_0 \rangle$. The Green's identity provides the following weak version of \eqref{Linear Boltz}:
		\begin{equation}\label{weak version of solution}
			\int_{\partial\Omega\times\mathbb{R}^3} \psi f \, dS dv
			- \iint_{\Omega \times \mathbb{R}^3} v \cdot \nabla_x \psi \, f dxdv
			= - \iint_{\Omega \times \mathbb{R}^3} \psi \, \frac{1}{\kappa}L(\mathrm{I}-\mathbf{P}) fdxdv
			+ \iint_{\Omega \times \mathbb{R}^3} \psi \, \phi_1dxdv.
		\end{equation}
		We choose the test function
		\begin{align}
			\psi=(|v|^2-\beta_{\langle f,\Phi_0 \rangle_v})v\cdot\nabla_x \phi_{\langle f,\Phi_0 \rangle_v}\Phi_0,
		\end{align}
		where $\phi_{\langle f,\Phi_0 \rangle_v}$ is the solution of following elliptic problem
		\begin{align}\label{ellip}
			\left\{
			\begin{aligned}
				-\Delta \phi_{\langle f,\Phi_0 \rangle_v} &= \langle f,\Phi_0 \rangle_v, && \text{in }\Omega, \\
				\chi(x)\phi_{\langle f,\Phi_0 \rangle_v}+\partial_n\phi_{\langle f,\Phi_0 \rangle_v}&= 0, && \text{on }\partial\Omega.
			\end{aligned}
			\right.
		\end{align}
		Here $\chi(x)$ is a non-negative $C^2$ function. We require that $\chi(x)=0$ on $\partial\Omega_{2}$ but $\chi> 0$ on a positive measure subset of $\partial\Omega$. Notice that we have following estimate
		\begin{align}\label{H2 estimate}
			\lVert \phi_{\langle f,\Phi_0 \rangle_v}\rVert_{H^2(\Omega)}\leq C\lVert \langle f,\Phi_0 \rangle_v\rVert_{L^2(\Omega)}.
		\end{align}
		This estimate can be found in \cite{elliptic problem} and \cite{Renardy}. We then bound the right hand side of \eqref{weak version of solution} by 
		\begin{align*}
			\frac{1}{\kappa}\|\phi_{\langle f,\Phi_0 \rangle_v}\|_{H^2(dx)}\|(\mathrm{I}-\mathbf{P}_L)f\|_{L^2_{0,\gamma/2}} + \|\phi_{\langle f,\Phi_0 \rangle_v}\|_{H^2(dx)}\|\phi_1\|_{L^2_{0,\gamma/2}}.
		\end{align*}
		We choose $\beta_{\langle f,\Phi_0 \rangle_v}$ such that 
		\begin{align}\label{vanish}
			\int_{\mathbb{R}^3} \left( |v|^2 - \beta_{\langle f,\Phi_0 \rangle_v} \right)
			\left( \frac{|v|^2}{2} - \frac{3}{2} \right)
			(v_i)^2 \,\mu(v)\, dv = 0.
		\end{align}
		Moreover, we decompose that
		\begin{align}\label{decomposition}
			\begin{aligned}
				f
				&= \mathbf{1}_{\partial\Omega^+_1}f+\mathbf{1}_{\partial\Omega^-_1}f_{in}+\mathbf{1}_{\partial\Omega_2}Pf + \mathbf{1}_{\partial\Omega^+_2}(I - P) f
				+ \mathbf{1}_{\partial\Omega^-_2}\phi_2,
				&&\qquad \text{on } \partial\Omega\times\mathbb{R}^3,\\
				f
				&= \sum_{i=0}^{4}\langle f,\Phi_{i} \rangle_v\Phi_{i}
				+ (\mathrm{I}-\mathbf{P}_L) f,
				&&\qquad \text{on } \Omega \times \mathbb{R}^3.
			\end{aligned}
		\end{align}
		By \eqref{decomposition}, since the contribution of $\langle f,\Phi_{4} \rangle_v$ vanishes in \eqref{weak version of solution} due to
		our choice of $\beta_{\langle f,\Phi_0 \rangle_v}$ and the contributions of $\langle f,\Phi_1 \rangle_v$, $\langle f,\Phi_2 \rangle_v$, and $\langle f,\Phi_3 \rangle_v$  vanish in \eqref{weak version of solution} due to the oddness, the left-hand side of \eqref{weak version of solution} takes the form of
		\begin{align}\label{detail 2}
			\int_{\partial\Omega\times\mathbb{R}^3}
			n \cdot v\left(|v|^{2} - \beta_{\langle f,\Phi_0 \rangle_v}\right)v\cdot\nabla_x\phi_{\langle f,\Phi_0 \rangle_v}\Phi_0
			\bigl[\mathbf{1}_{\partial\Omega^+_1}f+\mathbf{1}_{\partial\Omega^-_1}f_{in}+\mathbf{1}_{\partial\Omega_2}Pf + \mathbf{1}_{\partial\Omega^+_2}(I - P) f
			+ \mathbf{1}_{\partial\Omega^-_2}\phi_2\bigr]
		\end{align}
		\begin{align}\label{detail 3}
			-\sum_{i,k=1}^{d} \iint_{\Omega \times \mathbb{R}^3}
			(|v|^{2} - \beta_{\langle f,\Phi_0 \rangle_v}) v_i v_k
			\, \partial_{ik} \phi_{\langle f,\Phi_0 \rangle_v}\, \langle f,\Phi_0 \rangle_v\Phi^2_0 dxdv
		\end{align}
		\begin{align}\label{detail 4}
			- \sum_{i,k=1}^{d} \iint_{\Omega \times \mathbb{R}^3}
			(|v|^{2} - \beta_{\langle f,\Phi_0 \rangle_v}) v_i v_k
			\, \partial_{ik} \phi_{\langle f,\Phi_0 \rangle_v}\Phi_0\, (\mathrm{I}-\mathbf{P}_L) f .
		\end{align}
		
		We make an orthogonal decomposition at the boundary $\partial\Omega_2$,
		\[
		v= (v \cdot n) n + v-(v \cdot n) n = v_n n+ v_\perp.
		\]
		The contribution of $Pf = z(x)m^{\frac{1}{2}}=z(x)\Phi_{0}$ is
		\begin{align*}
			&\int_{\partial\Omega_2\times\mathbb{R}^3}
			v_n(|v|^{2} - \beta_{\langle f,\Phi_0 \rangle_v}) v \cdot \nabla_x \phi_{\langle f,\Phi_0 \rangle_v}(x)z(x)\Phi^2_{0}dSdv\\
			&=
			\int_{\partial\Omega_2\times\mathbb{R}^3}
			v_n(|v|^{2} - \beta_{\langle f,\Phi_0 \rangle_v})\,
			v_n \, z(x)\partial_n\phi_{\langle f,\Phi_0 \rangle_v}\Phi_{0}^2dSdv+\int_{\partial\Omega_2\times\mathbb{R}^3}
			v_n (|v|^{2} - \beta_a)\,
			v_\perp\cdot \nabla_x \phi_{\langle f,\Phi_0 \rangle_v}\,
			z(x)\Phi_{0}^2dSdv.
		\end{align*}
		Notice that the boundary condition \eqref{ellip} reduces to the Neumann boundary condition on $\partial\Omega_{2}$. Therefore, the first term vanish. While, the second term also vanishes due to
		the oddness of $ v_nv_\perp$.
		Therefore, \eqref{detail 2} and \eqref{detail 4} are bounded by
		$$
		\|\phi_{\langle f,\Phi_0 \rangle_v}\|_{H^2(dx)}
		\lVert (I-P_L)(f)\rVert_{L^2_{0,\gamma/2}}
		$$
		and 
		$$
		\|\phi_{\langle f,\Phi_0 \rangle_v}\|_{H^2(dx)}\left(\lVert (I-P)f\rVert_{L^2(\partial\Omega^{+}_2)}
		+\lVert f\rVert_{L^2(\partial\Omega^{+}_1)}
		+\lVert\phi_2\rVert_{L^2(\partial\Omega^{-}_2)}
		+\lVert f_{in}\rVert_{L^2(\partial\Omega^{-}_1)}\right).
		$$
		
		For \eqref{detail 3}, if $k \neq i$, it vanishes due to the oddness.
		Hence, the only contribution comes from $k=i$.
		\[
		-\sum_{i=1}^{d}
		\iint_{\Omega \times \mathbb{R}^3}
		(|v|^{2} - \beta_a)\,(v_i)^2\,
		\mu\, \partial_{ii} \phi_a\, a .
		\]
		Using $-\Delta_x \phi_{\langle f,\Phi_0 \rangle_v} = {\langle f,\Phi_0 \rangle_v}$ and 
		$\lVert \phi_{\langle f,\Phi_0 \rangle_v}\rVert_{H^2(dx)}\leq C\lVert \langle f,\Phi_0 \rangle_v\rVert_{L^2(dx)}$, we obtain
		\begin{align*}
			\|{\langle f,\Phi_0 \rangle_v}\|_{L^2(dx)}
			\lesssim
			\begin{pmatrix*}[l]
				&
				\left[1+\frac{1}{\kappa}\right]\lVert (I-P_L)(f)\rVert_{L^2_{0,\gamma/2}}
				+\lVert \phi_1\rVert_{L^2_{0,-\gamma/2}}
				+\lVert (I-P)f\rVert_{L^2(\partial\Omega^+_2)}\\
				&+\lVert f\rVert_{L^2(\partial\Omega^{+}_1)}
				+\lVert f_{in}\rVert_{L^2(\partial\Omega^{-}_1)}
				+\lVert\phi_2\rVert_{L^2(\partial\Omega^{-}_2)}
			\end{pmatrix*}.
		\end{align*}
		We therefore deduce the proposition.
	\end{proof}

	\subsection{Proposition of $S_\Omega$ and $J$}
	Here we introduce a proposition from \cite{Kawagoe Daisuke 2025}.
	\begin{proposition}\label{proposition 1.3}
		Let $\Omega$ be a bounded convex domain with $C^1$ boundary, $\alpha\geq 0, \beta\in\mathbb{R}^3, -3<\gamma\leq 1$ and $2\leq p\leq\infty$. Suppose that \eqref{eq 1.2} holds. Then, the operator norm of $S^\kappa_\Omega : L^p_{\alpha,\beta}\to L^p_{\alpha,\beta+1}$ is $\mathcal{O}(\kappa)$. 
	\end{proposition}
	We also have the following estimate.
	\begin{proposition}\label{proposition 2.8}
		Let $\Omega$ be a bounded convex domain with $C^1$ boundary and $-3<\gamma\leq 1$, Suppose $f\in L^2(\partial\Omega^{+}_2)$. Then, 
		\begin{align}
			\lVert J^\kappa\large(f_{in}\textbf{1}_{\partial\Omega^{-}_1}+(P(f)+\phi_2)\textbf{1}_{\partial\Omega^{-}_2})\rVert_{L^2_{0,\gamma/2}}\leq C\kappa\Big(\lVert f_{in}\rVert_{L^2(\partial\Omega^{-}_1)}+\lVert f\rVert_{L^2(\partial\Omega^{+}_2)}+\lVert\phi_2\rVert_{L^2(\partial\Omega^{-}_2)}\Big).
		\end{align}
	\end{proposition}
	\begin{proof}
		By \textbf{Lemma 2.1} (\cite{Choulli 1999}), we have
		\begin{align*}
			&\lVert J^\kappa\large(f_{in}\textbf{1}_{\partial\Omega^{-}_1}+(P(f)+\phi_2)\textbf{1}_{\partial\Omega^{-}_2})\rVert^2_{L^2_{0,\gamma/2}}\\
			=&\int_{\Omega}\int_{\mathbb{R}^3}e^{-2\frac{1}{\kappa}\nu(|v|)\tau_{-}(x,v)}
			\Big[\big(f_{in}\textbf{1}_{\partial\Omega^{-}_1}+(P(f)+\phi_2)\textbf{1}_{\partial\Omega^{-}_2}\big)(q(x,v),v)\Big]^2(1+|v|)^{\gamma}dv dx\\
			=&\int_{\Omega^-}\int_{0}^{\tau_{-}(x,-v)}e^{-2\frac{1}{\kappa}\nu(|v|)t}
			\Big[\big(f_{in}\textbf{1}_{\partial\Omega^{-}_1}+(P(f)+\phi_2)\textbf{1}_{\partial\Omega^{-}_2}\big)(z,v)\Big]^2(1+|v|)^{\gamma}|v\cdot n(x)|dt dS dv\\
			=&\kappa\int_{\Omega^-}
			\Big[\big(f_{in}\textbf{1}_{\partial\Omega^{-}_1}+(P(f)+\phi_2)\textbf{1}_{\partial\Omega^{-}_2}\big)(z,v)\Big]^2\frac{(1+|v|)^{\gamma}}{\nu(|v|)}|v\cdot n(x)|dS dv\\\nonumber
			\leq&C\kappa\int_{\Omega^-}
			\big(f^2_{in}\textbf{1}_{\partial\Omega^{-}_1}+P^2(f)\textbf{1}_{\partial\Omega^{-}_2}+\phi^2_2\textbf{1}_{\partial\Omega^{-}_2}\big)(z,v)|v\cdot n(x)|dS dv\\\nonumber
			\leq&C\kappa \Big(\lVert f_{in}\rVert^2_{L^2(\partial\Omega^{-}_1)}+\lVert\phi_2\rVert^2_{L^2(\partial\Omega^{-}_2)}+\lVert P(f)\rVert^2_{L^2(\partial\Omega^{-}_2)}\Big)\\\nonumber
			=&C\kappa\Big(\lVert f_{in}\rVert^2_{L^2(\partial\Omega^{-}_1)}+\lVert\phi_2\rVert^2_{L^2(\partial\Omega^{-}_2)}+\lVert P(f)\rVert^2_{L^2(\partial\Omega^{+}_2)}\Big)\\\nonumber
			\leq& C\kappa\Big(\lVert f_{in}\rVert^2_{L^2(\partial\Omega^{-}_1)}+\lVert\phi_2\rVert^2_{L^2(\partial\Omega^{-}_2)}+\lVert f\rVert^2_{L^2(\partial\Omega^{+}_2)}\Big).\nonumber
		\end{align*}
		\qquad This finish the proof.
	\end{proof}
	
	\subsection{Estimate of $Q_+$ and $Q_-$}
	\begin{lemma}\label{lemma 2.1}
		For any $\beta\geq 1$,
		\begin{align}\label{eq 2.12}
			\int_{\mathbb{S}^2}\langle v'\rangle^{-\beta} \langle u'\rangle^{-\beta} d\sigma \leq C_\beta \left(\frac{1}{\sqrt{1+|v|^2+|u|^2}}\right)^{-(\beta+1)},
		\end{align}
		where $v'=\frac{v+u}{2}+\frac{|v-u|}{2}\sigma$ and $u'=\frac{v+u}{2}-\frac{|v-u|}{2}\sigma$.
	\end{lemma}
	This statement can be found in \cite{Ukai Seiji Solutions of the Boltzmann equation}. We provide a proof below.
	\begin{proof}
		For convenience, let $A=1+\left|\frac{v+u}{2}\right|^2+\left|\frac{v-u}{2}\right|^2$  and $B=2\left|\frac{v+u}{2}\right|\left|\frac{v-u}{2}\right|$. When $\beta=1$,
		\begin{align*}
			\int_{\mathbb{S}^2}\langle v'\rangle^{-1} \langle u'\rangle^{-1} d\sigma=\int_{0}^{2\pi}\int_{0}^{\pi}\frac{\sin\theta}{\sqrt{A^2-B^2\cos^2\theta}}d\theta \phi.
		\end{align*}
		Here we use sphere coordinate with $e_1=\frac{v+u}{|v+u|}$. Let $z=\cos\theta$ and $y=\sin^{-1}\left(Bz/A\right)$. We obtain
		\begin{align*}
			\int_{0}^{2\pi}\int_{0}^{\pi}\frac{\sin\theta}{\sqrt{A^2-B^2\cos^2\theta}}d\theta \phi
			&=4\pi\int_{0}^{1}\frac{1}{\sqrt{A^2-B^2z^2}}dz\\
			&=4\pi\int_{0}^{\sin^{-1}(B/A)}\frac{\cos(y)}{B|\cos(y)|}dy\\
			&\leq8\pi A^{-1}\\
			&\leq \frac{16\pi}{1+|v|^2+|u|^2}.
		\end{align*}
		The penultimate inequality follows from $\sin^{-1}(x)\leq 2x$ for any  $0\leq x\leq 1$. For $\beta>1$, we first consider the case of $\beta\in\mathbb{N}$. We use induction to establish the result. Since $\beta=1$ holds by the above calculation. Assume that $\beta=n$ holds. When $\beta=n+1$, we have
		\begin{align*}
			\int_{\mathbb{S}^2}\langle v'\rangle^{-(n+1)} \langle u'\rangle^{-(n+1)} d\sigma&=4\pi\int_{0}^{B/A}BA^{-n}\frac{1}{\left(\sqrt{1-y^2}\right)^{n+1}}dy\\
			&\leq \frac{4\pi}{A\sqrt{1-(B/A)^2}}\int_{0}^{B/A}BA^{1-n}\frac{1}{\left(\sqrt{1-y^2}\right)^{n}}dy\\
			&\leq\frac{C}{\sqrt{A^2-B^2}}\left(\frac{1}{\sqrt{1+|v|^2+|u|^2}}\right)^{n+1}\\
			&\leq\left(\frac{1}{\sqrt{1+|v|^2+|u|^2}}\right)^{n+2}.
		\end{align*}
		The second-to-last line follows from the induction hypothesis, so \eqref{eq 2.12} holds when $\beta\in\mathbb{N}$. For $\beta\notin\mathbb{N}$, suppose that $\beta=n+d$ for some $n\in\mathbb{N}$ and $d\in(0,1)$. Then, we have
		\begin{align*}
			4\pi\int_{0}^{B/A}BA^{1-\beta}\frac{1}{\left(\sqrt{1-y^2}\right)^{\beta}}dy&\leq
			\frac{4\pi}{\left(\sqrt{A^2-B^2}\right)^d}\int_{0}^{B/A}BA^{1-n}\frac{1}{\left(\sqrt{1-y^2}\right)^{n}}dy\\
			&\leq\frac{C}{\left(\sqrt{A^2-B^2}\right)^d}\left(\frac{1}{\sqrt{1+|v|^2+|u|^2}}\right)^{-n-1}\\
			&=C\left(\frac{1}{\sqrt{1+|v|^2+|u|^2}}\right)^{-(\beta+1)}.
		\end{align*}
	\end{proof}
	Now, we provide estimates of $Q_+, Q_-$, which will be used in \textbf{section 5}. 
	\begin{proposition}\label{proposition 2.9}
		Suppose that \eqref{eq 1.2} holds. Then,
		\begin{align}\label{eq 2.13}
			\left|\int_{\mathbb{R}^3}|v|S^\kappa_\Omega\left(Q_+[f,f]\right)(x,v)dv\right|\leq C_{\gamma,\beta} \lVert f\rVert^2_{L^\infty_{0,\beta}(\Omega\times\mathbb{R}^3)}
		\end{align}
		for any $\beta>5+\max\{0,\gamma\}$ and 
		\begin{align}\label{eq 2.14}
			\left|\int_{\mathbb{R}^3}|v|S^\kappa_\Omega\left(Q_-[f,f]\right)(x,v)dv\right|\leq C_{\gamma,\beta} \lVert f\rVert^2_{L^\infty_{0,\beta}(\Omega\times\mathbb{R}^3)}
		\end{align}
		for any $\beta>3+\max\{0,\gamma\}$.
	\end{proposition}
	\begin{proof}
		The proof of \eqref{eq 2.14} follows directly from 
		\begin{align}
			\lVert S^\kappa(f)\rVert_{L^\infty_{\alpha,\beta}(\Omega\times\mathbb{R}^3)}\leq \frac{C}{|v|}\lVert f\rVert_{L^\infty_{\alpha,\beta}(\Omega\times\mathbb{R}^3)}.
		\end{align} 
		The proof of \eqref{eq 2.13} proceeds as follows:
		\begin{align*}
			\left|\int_{\mathbb{R}^3}|v|S^\kappa_\Omega\left(Q_+[f,f]\right)(x,v)dv\right|&\leq C\left|\int_{\mathbb{R}^3}\lVert Q_+[f,f]\rVert_{L^\infty_{0,\beta}(\Omega\times\mathbb{R}^3)}dv\right|\\
			&\leq \int_{\mathbb{R}^3}\int_{\mathbb{R}^3}\int_{\mathbb{S}^2}|v-u|^\gamma\langle v'\rangle^{-\beta} \langle u'\rangle^{-\beta} d\sigma du dv\lVert f\rVert^2_{L^\infty_{0,\beta}(\Omega\times\mathbb{R}^3)}\\
			&\leq C\int_{\mathbb{R}^3}\int_{\mathbb{R}^3}|v-u|^\gamma\left(\frac{1}{\sqrt{1+|v|^2+|u|^2}}\right)^{-(\beta+1)} du dv\lVert f\rVert^2_{L^\infty_{0,\beta}(\Omega\times\mathbb{R}^3)}\\
			&\leq C_{\gamma,\beta}\lVert f\rVert^2_{L^\infty_{0,\beta}(\Omega\times\mathbb{R}^3)}.
		\end{align*}
		The second-to-last inequality follows from \textbf{Lemma \eqref{lemma 2.1}}.
	\end{proof}

	\section{$L^2$ estimate}\label{section 3}
	Before we begin the $L^2$ estimate, we recall a lemma concerning boundary data. This lemma ensures that the trace of $f$ is well-defined locally for a certain class of functions. The proof can be found in \cite{Esposito R 2013}.
	\begin{lemma}[Ukai Trace Theorem]\label{2.0.1}
		\emph{Define}
		\begin{align}\label{2.1}
			\gamma^{\varepsilon} \equiv \{ (x,v)\in \partial\Omega\times\mathbb{R}^3 : |n(x)\cdot v|\geq \varepsilon, \epsilon\leq\ |v|\leq \tfrac{1}{\varepsilon}\}.
		\end{align}
		\emph{Then}
		\begin{align}
			\| f1_{\gamma^{\varepsilon}} \|_{L^1(\partial\Omega\times\mathbb{R}^3)} \;\leq C_{\varepsilon,\Omega}\Big(\|f\|_{L^1(\Omega\times\mathbb{R}^3)} + \| v\cdot \nabla_x f \|_{L^1(\Omega\times\mathbb{R}^3)}\Big)
			.
		\end{align}
	\end{lemma}
	Also, we recall Green's identity. Let $\gamma$ denote the trace operator.
	\begin{lemma}\label{Greens formula}
		Assume that
		$f(x,v),\, g(x,v) \in L^2(\Omega\times\mathbb{R}^3)$,
		$v\cdot\nabla_x f,\; v\cdot\nabla_x g \in L^2(\Omega\times\mathbb{R}^3)$,
		and
		$\gamma f,\; \gamma g \in L^2(\partial\Omega\times\mathbb{R}^3)$.
		Then
		\begin{equation}
			\iint_{\Omega\times\mathbb{R}^3}
			\bigl\{(v\cdot\nabla_x f)\,g + (v\cdot\nabla_x g)\,f\bigr\}\,dv\,dx
			=
			\int_{\partial\Omega\times\mathbb{R}^3} \gamma f \gamma g \, dS dv.
			\tag{2.3}
		\end{equation}
	\end{lemma}
	The proof can be found in \cite{Cercignani Carlo Reinhard Illner}. We now state the main theorem of this section. The statement and strategy are similar to \cite{Duan R 2019} and \cite{Esposito R 2013}. However, 
	we need to modify the proof due to the difference between the mixed and diffuse boundary conditions. Moreover, we use \textbf{Lemma 1.2 (\cite{Kawagoe Daisuke 2025})} to simplify the argument. We present the details below.
	\begin{theorem}\label{theorem 3.1}
		Let $\Omega\in\mathbb{R}^3$ be a bounded convex domain with $C^2$ boundary. Assume that $-3<\gamma\leq 1$ and that \eqref{eq 1.2} holds. Then, for any $\phi_1\in L^2_{0,\max\{0,-\gamma/2\}}(\Omega\times\mathbb{R}^3)$,   $\phi_2\in L^2_{0,\gamma/2}(\partial\Omega^{-}_1)$, and   $f_{in}\in L^2_{0,\gamma/2}(\partial\Omega^{-}_2)$. There exists a unique solution $f\in L^2_{0,\gamma/2}$ to the following boundary value problem.
		\begin{align}\label{2.6}
			\left\{
			\begin{aligned}
				v\cdot\nabla_x f+\frac{1}{\kappa}L(f)&=\phi_1,&&\quad\text{in}\quad \Omega\times\mathbb{R}^3, \\
				f&=f_{in}, &&\quad\text{on}\quad \partial\Omega^{-}_1,\\
				f&=P(f)+\phi_2(x,v), &&\quad\text{on} \quad\partial\Omega^{-}_2.
			\end{aligned}
			\right.
		\end{align}
		Moreover,
		\begin{align}\label{equation 3.5}
			\lVert f\rVert^2_{L^2_{0,\gamma/2}}
			&+\lVert f\rVert^2_{L^2(\partial\Omega^{+}_1)}
			\leq C\max\{1,\kappa^2\}\Big(
			\lVert\phi_1\rVert^2_{L^2_{0,-\gamma/2}}
			+\lVert f_{in}\rVert^2_{L^2(\partial\Omega^{-}_1)}
			+\lVert\phi_2\rVert^2_{L^2(\partial\Omega^{-}_2)}
			\Big),
		\end{align}
	\end{theorem}
	\begin{proof}
		We consider the following iteration with $f_1=0$,
		\begin{align}\label{2.7}
			\left\{
			\begin{aligned}
				\varpi f_{k+1}+v\cdot\nabla_x f_{k+1}+\frac{1}{\kappa}L(f_{k+1})&=\phi_1,&&\quad\text{in}\quad\in\Omega\times\mathbb{R}^3, \\
				f_{k+1}&=f_{in}, &&\quad\text{on}\quad \partial\Omega^{-}_1,\\
				f_{k+1}&=\eta P(f_{k})+\phi_2(x,v), &&\quad\text{on} \quad\partial\Omega^{-}_2,
			\end{aligned}
			\right.
		\end{align}
		where $0<\varpi, \eta<1$. 
		Note that \eqref{2.7} differs slightly \eqref{2.6}. More precisely, the boundary value problem \eqref{2.7} contains a penalization term $\varpi f_{k+1}$ and the factor $\eta$ inserted in front of $P(f)$.\par
		\qquad We first construct a solution to \eqref{2.7} in $L^2_{0,\gamma/2}$ for arbitrary $\varpi$ and $\eta$ and denote this solution by $f^{\varpi, \eta}$. We then show that, for fixed $\eta$, the family $\{f^{\varpi,\eta}\}$ converges in $L^2_{0,\gamma/2}$. Finally, we remove $\eta$ by proving that $\{f^{\eta}\}$ is a convergent sequence in $L^2_{0,\gamma/2}$,where $f^\eta=\lim_{\varpi\to 0} f^{\varpi,\eta}$.
		We emphasize that the penalization term is introduced in order to justify the application of Green's formula to \eqref{2.7} for all $-3<\gamma\leq 1$. Consequently, when $\gamma<0$, the procedure for  removing $\varpi$ and $\eta$ are different from that in the case of $\gamma\geq 0$.  Since the argument for the negative $\gamma$ case relies on an inequality established in \textbf{section \ref{section 4}}, for the sake of clarity we present only the case $\gamma\geq0$ here. The case  $\gamma<0$ can be treated by applying the same argument as in \cite{Duan R 2019}.\par 
		\qquad To construct the solution of \eqref{2.7}, by Green formula, we have
		\begin{align}\label{2.3}
			&\varpi\lVert f_{k+1}\rVert^2_{L^2}+\frac{1}{2}\lVert f_{k+1}\rVert^2_{L^2(\partial\Omega^{+})}+\frac{1}{\kappa}\langle L(f_{k+1}),f_{k+1} \rangle\\\nonumber
			&=\langle \phi_1,f_{k+1} \rangle+\frac{1}{2}\lVert f_{in}\rVert^2_{L^2(\partial\Omega^{-}_1)}+\frac{1}{2}\lVert \eta P(f_k)+\phi_2\rVert^2_{L^2(\partial\Omega^{-}_2)}.
		\end{align}
		In particular, we have
		\begin{align}
			\lVert f_{k+1}\rVert^2_{L^2(\partial\Omega^{+})}\leq C_{\varpi}\Big(\lVert \phi_1\rVert^2_{L^2_{0,-\gamma/2}}+\lVert f_{in}\rVert^2_{L^2(\partial\Omega^{-}_1)}+\lVert P(f_k)\rVert^2_{L^2(\partial\Omega^{+}_2)}+\lVert\phi_2\rVert^2_{L^2(\partial\Omega^{-}_2)}\Big) .
		\end{align}
		Using \textbf{Proposition \ref{proposition 2.8}} , we deduce that
		\begin{align}
			J^{\kappa}\large(f_{in}\chi_{\partial\Omega^{-}_1}+(P(f_k)+\phi_2)\chi_{\partial\Omega^{-}_2})\in L^2_{0,\gamma/2} \quad\text{if }\, f_k\in L^2_{\Omega^+_2}.
		\end{align}
		We can then rewrite \eqref{2.7} in integral form as:
		\begin{align}
			(I-\frac{1}{\kappa}S^{\kappa}_{\Omega}K)f_{k+1}=J^{\kappa}\large(f_{in}\chi_{\partial\Omega^{-}_1}+(P(f_k)+\phi_2)\chi_{\partial\Omega^{-}_2})+S_{\Omega}\phi_1.
		\end{align}
		Applying \textbf{Lemma 1.2} (\cite{Kawagoe Daisuke 2025}), we establish the existence of $f_{k+1}$. By induction, we can then prove that $f_{k+1}$ exists and belongs to $L^2_{0,\gamma/2}$ for any $k\in\mathbb{N}$. Applying Yong's inequality, we have
		\begin{align}\label{eqq 3.10}
			\lVert \eta P(f_k)+\phi_2\rVert^2_{L^2(\partial\Omega^{-}_2)}&=\eta^2\lVert P(f_k)\rVert^2_{L^2(\partial\Omega^{-}_2)}+\lVert \phi_2\rVert^2_{L^2(\partial\Omega^{-}_2)}+2\eta\langle P(f_{k+1}), \phi_2\rangle_{\partial^-_2}\\\nonumber
			&\leq \eta\lVert P(f_k)\rVert^2_{L^2(\partial\Omega^{-}_2)}+\frac{1}{1-\eta}\lVert \phi_2\rVert^2_{L^2(\partial\Omega^{-}_2)}\\\nonumber
			&\leq \eta\lVert f_k\rVert^2_{L^2(\partial\Omega^{+}_2)}+\frac{1}{1-\eta}\lVert \phi_2\rVert^2_{L^2(\partial\Omega^{-}_2)}
		\end{align}
		Then, we combine the \eqref{2.3} and \eqref{eqq 3.10} to obtain
		\begin{align}
			\lVert f_{k+1}\rVert^2_{L^2}+(2\varpi)^{-1}\lVert f_{k+1}\rVert^2_{L^2(\partial\Omega^{+})}
			&\leq\eta \left(\lVert f_{k}\rVert^2_{L^2}+(2\varpi)^{-1}\lVert f_{k}\rVert^2_{L^2(\partial\Omega^{+})}\right)\\\nonumber
			&\quad+C_{\varpi,\eta}\left(\lVert \phi_1\rVert^2_{L^2_{0,-\gamma/2}}+\lVert f_{in}\rVert^2_{L^2(\partial\Omega^{-}_1)}\right)+C_{\varpi,\eta}\lVert \phi_2\rVert^2_{L^2(\partial\Omega^{-}_2)}.
		\end{align}
		By same process, we have
		\begin{align}
			\lVert f_{n+1}-f_{m+1}\rVert^2_{L^2}+(2\varpi)^{-1}\lVert f_{n+1}-f_{m+1}\rVert^2_{L^2(\partial\Omega^{+})}
			&\leq
			\eta \left(\lVert f_{n}-f_{m}\rVert^2_{L^2}+(2\varpi)^{-1}\lVert f_{n}-f_{m}\rVert^2_{L^2(\partial\Omega^{+})}\right).
		\end{align}
		This implies that $\{f_k\}$ is a Cauchy sequence. For any fixed $\eta$ and $\varpi$, let $\displaystyle\lim_{k\to\infty}f_k\coloneqq f^{\varpi,\eta}$. By applying Green's formula once more, we have
		\begin{align}\label{(2.5)}
			&\varpi\lVert f^{\varpi,\eta}\rVert^2_{L^2}+\frac{1}{2}\lVert f^{\varpi,\eta}\rVert^2_{L^2(\partial\Omega^{+})}+\frac{1}{\kappa}\langle L(f^{\varpi,\eta}),f^{\varpi,\eta} \rangle
			\\\nonumber
			&=\langle \phi_1,f^{\varpi,\eta} \rangle+\frac{1}{2}\lVert f_{in}\rVert^2_{L^2(\partial\Omega^{-}_1)}+\frac{1}{2}\lVert \eta P(f^{\varpi,\eta})+\phi_2\rVert^2_{L^2(\partial\Omega^{-}_2)}.
		\end{align}
		Applying \eqref{eq 2.9} we have
		\begin{align}
			\lVert P_L(f^{\varpi,\eta})\rVert^2_{L^2_{0,\gamma/2}}&\leq C
			\begin{pmatrix*}[l]
				&\varpi\lVert f^{\varpi,\eta}\rVert^2_{L^2}
				+\left[1+\frac{1}{\kappa}\right]\lVert (I-P_L)(f^{\varpi,\eta})\rVert^2_{L^2_{0,\gamma/2}}
				+\lVert \phi_1\rVert^2_{L^2_{0,-\gamma/2}}\\
				&
				+\lVert f^{\varpi,\eta}\rVert^2_{L^2(\partial\Omega^{+}_1)}
				+
				\lVert (I-P)f^{\varpi,\eta}\rVert^2_{L^2(\partial\Omega^{+}_2)}
				+
				\lVert f_{in}\rVert^2_{L^2(\partial\Omega^{-}_1)}
				+\lVert\phi_2\rVert^2_{L^2(\partial\Omega^{-}_2)}.
			\end{pmatrix*}. 
		\end{align}
		Multiplying it with a constant $a>0$ and combining it with equation \eqref{(2.5)}, we derive that
		\begin{align*}\label{process}
			&\varpi(1-aC)\lVert f^{\varpi,\eta}\rVert^2_{L^2}+a\lVert P_L(f^{\varpi,\eta})\rVert^2_{L^2_{0,\gamma/2}}+C_L\frac{1}{\kappa}\lVert (I-P_L)(f^{\varpi,\eta})\rVert^2_{L^2_{0,\gamma/2}}+\frac{1}{2}\lVert f^{\varpi,\eta}\rVert^2_{L^2(\partial\Omega^{+})}\\
			&\leq aC\left(\left[1+\frac{1}{\kappa}\right]\lVert (I-P_L)(f^{\varpi,\eta})\rVert^2_{L^2_{0,\gamma/2}}+\lVert \phi_1\rVert^2_{L^2_{0,-\gamma/2}}+\lVert f^{\varpi,\eta}\rVert^2_{L^2(\partial\Omega^{+}_1)}+\lVert (I-P)f^{\varpi,\eta}\rVert^2_{L^2(\partial\Omega^{+}_2)}\right)\\
			&+aC\left(\lVert f_{in}\rVert^2_{L^2(\partial\Omega^{-}_1)}+\lVert\phi_2\rVert^2_{L^2(\partial\Omega^{-}_2)}\right)\\
			&+\frac{1}{2}\lVert f_{in}\rVert^2_{L^2(\partial\Omega^{-}_1)}+\frac{1}{2}\lVert \eta P(f^{\varpi,\eta})+\phi_2\rVert^2_{L^2(\partial\Omega^{-}_2)}+\langle \phi_1,f^{\varpi,\eta} \rangle.		 	
		\end{align*}
		Finally, by choosing $a$ sufficiently small, we conclude that
		\begin{align}
			&\lVert f^{\varpi,\eta}\rVert^2_{L^2_{0,\gamma/2}}+\lVert f^{\varpi,\eta}\rVert^2_{L^2(\partial\Omega^{+}_1)}+\lVert (I-P)f^{\varpi,\eta}\rVert^2_{L^2(\partial\Omega^{+}_2)}+C_L\lVert (I-P_L)(f^{\varpi,\eta})\rVert^2_{L^2_{0,\gamma/2}}\\\nonumber
			&\leq C_\eta\cdot\kappa\Big(\lVert \phi_1\rVert^2_{L^2_{0,-\gamma/2}}+\lVert f_{in}\rVert^2_{L^2(\partial\Omega^{-}_1)}+\lVert\phi_2\rVert^2_{L^2(\partial\Omega^{-}_2)}\Big).		 	
		\end{align}
		Hence, $f^{\varpi,\eta}$ is a bounded sequence in $L^2_{0,\gamma/2}$ for any fixed $\eta$. Observe that $f^{\eta,\varpi_1}-f^{\eta,\varpi_2}$ satisfies
		\[\left\{
		\begin{aligned}
			v\cdot\nabla_x  \Big(f^{\eta,\varpi_1}-f^{\eta,\varpi_2}\Big)+\frac{1}{\kappa}L\Big(f^{\eta,\varpi_1}-f^{\eta,\varpi_2}\Big)&=\varpi_1 f^{\eta,\varpi_1}-\varpi_2 f^{\eta,\varpi_2},&&\quad\text{in}\quad\Omega\times\mathbb{R}^3, \\
			f^{\eta,\varpi_1}-f^{\eta,\varpi_2}&=0, &&\quad\text{on} \quad\partial\Omega^{-}_1,\\
			f^{\eta,\varpi_1}-f^{\eta,\varpi_2}&=\eta P\Big(f^{\eta,\varpi_1}-f^{\eta,\varpi_2}\Big), &&\quad\text{on}\quad \partial\Omega^{-}_2,
		\end{aligned}\right.
		\]
		for any $0<\varpi_1,\varpi_2<1$. After taking $\varpi_1 f^{\eta,\varpi_1}-\varpi_2 f^{\eta,\varpi_2}$ as $\phi_1$ and setting $f_{in}=\phi_2=0$, by H\"older's inequality, we have
		\begin{align}
			&\lVert f^{\eta,\varpi_1}-f^{\eta,\varpi_2}\rVert^2_{L^2_{0,\gamma/2}}+\lVert f^{\eta,\varpi_1}-f^{\eta,\varpi_2}\rVert^2_{L^2(\partial\Omega^{+}_1)}
			+\lVert (I-P)\Big(f^{\eta,\varpi_1}-f^{\eta,\varpi_2}\Big)\rVert^2_{L^2(\partial\Omega^{+}_2)}\\\nonumber
			&\leq C_\eta\cdot \kappa \left(\varpi_1\lVert f^{\eta,\varpi_1}\rVert^2_{L^2_{0,-\gamma/2}}+ \varpi_2\lVert f^{\eta,\varpi_2}\rVert^2_{L^2_{0,-\gamma/2}}\right)\\\nonumber
			&\leq C_\eta\cdot\kappa\max\{\varpi_1,\varpi_2\}\Big(\lVert \phi_1\rVert^2_{L^2_{0,-\gamma/2}}+\lVert f_{in}\rVert^2_{L^2(\partial\Omega^{-}_1)}+\lVert\phi_2\rVert^2_{L^2(\partial\Omega^{-}_2)}\Big).		 	
		\end{align}
		This implies that $\displaystyle f^\eta=\lim_{\varpi\to 0}f^{\varpi,\eta}$ exists in $L^2_{0,\gamma/2}$ for any $0<\eta<1$. Moreover, applying the same procedure as in  \eqref{process}, we obtain
		\begin{align}\label{2.9}
			\lVert f^{\eta}\rVert^2_{L^2_{0,\gamma/2}}
			&\quad+\lVert f^{\eta}\rVert^2_{L^2(\partial\Omega^{+}_1)}+\lVert (I-P)f^{\eta}\rVert^2_{L^2(\partial\Omega^{+}_2)}+C_L\lVert (I-P_L)(f)\rVert^2_{L^2_{0,\gamma/2}}\\\nonumber
			&\leq C\cdot\kappa\Big(\lVert \phi_1\rVert^2_{L^2_{0,-\gamma/2}}+\lVert f_{in}\rVert^2_{L^2(\partial\Omega^{-}_1)}
			+C_\beta\lVert\phi_2\rVert^2_{L^2(\partial\Omega^{-}_2)}+\beta\lVert P(f^{\eta})\rVert^2_{L^2(\partial\Omega^{-}_2)}\Big),
		\end{align}
		where $\beta$ will be determined later. Furthermore, we have
		\begin{align}\label{2.10}
			v\cdot\nabla_x (f^\eta)^2=-\frac{2}{\kappa}L(f^\eta)f^\eta+2\phi_1 f^\eta.
		\end{align}
		Taking absolute values and integrating over $\Omega\times\mathbb{R}^3$ in (\ref{2.10}), and using (\ref{2.9}), we deduce that
		\begin{align}
			\lVert v\cdot\nabla_x (f^\eta)^2\rVert_{L^1}\leq C\max\{1,\kappa\}\Big(\lVert \phi_1\rVert^2_{L^2_{0,-\gamma/2}}+\lVert f_{in}\rVert^2_{L^2(\partial\Omega^{-}_1)}+C_\beta\lVert\phi_2\rVert^2_{L^2(\partial\Omega^{-}_2)}+\beta\lVert P(f^{\eta})\rVert^2_{L^2(\partial\Omega^{-}_2)}\Big).
		\end{align}
		Using \textbf{Lemma \ref{2.0.1}} , for any $\gamma^\epsilon$ in (\ref{2.1}) away from $\partial\Omega_0$, we obtain 
		\begin{align}\label{2.12}
			\lVert f^\eta \textbf{1}_{\gamma^\epsilon}\rVert^2_{L^2(\partial\Omega\times\mathbb{R}^3)}\leq C_\epsilon\max\{1,\kappa\}\Big(\lVert \phi_1\rVert^2_{L^2_{0,-\gamma/2}}+\lVert f_{in}\rVert^2_{L^2(\partial\Omega^{-}_1)}+C_\beta\lVert\phi_2\rVert^2_{L^2(\partial\Omega^{-}_2)}+\beta\lVert P(f^{\eta})\rVert^2_{L^2(\partial\Omega^{-}_2)}\Big) .
		\end{align}
		By the definition of $P(f)$, we can write $P(f)(x,v)=g(x)m^{1/2}(v)$ for some $g(x)$. When $\epsilon$ is sufficiently small, 
		\begin{align}\label{2.13}
			\lVert P(f^\eta) \textbf{1}_{\gamma^\epsilon}\rVert^2_{L^2(\partial\Omega^{+}_2)}
			&= \int_{\partial\Omega_2} |g(x)|^2 
			\int_{n(x)\cdot v\geq \varepsilon,\, |v|\leq \tfrac{1}{\varepsilon}} 
			m(v)\, |v\cdot n(x)|\, dv\, dx \notag \\
			&\geq \int_{\partial\Omega_2} |g(x)|^2 \, dx 
			\times \tfrac{1}{2} \int_{v\cdot n(x)>0} m(v)\, |v\cdot n(x)|\, dv \notag \\
			&= \tfrac{1}{2} \lVert P(f^\eta) \rVert^2_{L^2(\partial\Omega^{+}_2)}.
		\end{align}
		Here, we have used the fact that
		\begin{align}
			\int_{|n(x)\cdot v|\leq \varepsilon} m(v)\, |n(x)\cdot v|\, dv
			&\leq \int_{-\varepsilon'}^{\varepsilon} e^{-v_{\parallel}^2}
			|v_{\parallel}|\, dv_{\parallel}
			\int_{\mathbb{R}^2} e^{-|v_{\perp}|^2/2}\, dv_{\perp} 
			\;\leq C \varepsilon, \notag \\[6pt]
			\int_{|v|\geq 1/\varepsilon} m(v)\, |n(x)\cdot v|\, dv 
			&\leq C \varepsilon.
		\end{align}
		We observe that 
		\begin{align}
			\lVert P(f^\eta)\rVert^2_{L^2(\partial\Omega^{-}_2)}&=\lVert P(f^\eta)\rVert^2_{L^2(\partial\Omega^{+}_2)}\\
			f^\eta&=P(f^\eta)+(I-P)(f^\eta)
		\end{align}
		which yields
		\begin{align}\label{2.25}
			\lVert P(f^\eta)\textbf{1}_{\gamma^\epsilon}\rVert^2_{L^2(\partial\Omega^{+}_2)}\leq 2 \lVert f^\eta\textbf{1}_{\gamma^\epsilon}\rVert^2_{L^2(\partial\Omega^{+}_2)}+ 2\lVert (I-P)(f^\eta)\textbf{1}_{\gamma^\epsilon}\rVert^2_{L^2(\partial\Omega^{+}_2)}.
		\end{align}
		Combining this estimate with \eqref{2.9}, \eqref{2.12},  \eqref{2.13} and \eqref{2.25}, we conclude that
		\begin{align}\label{2.21}
			\lVert P(f^\eta)\rVert^2_{L^2(\partial\Omega^{+}_2)}\leq C\max\{1,\kappa\}\Big( \lVert \phi_1\rVert^2_{L^2_{0,-\gamma/2}}+\lVert f_{in}\rVert^2_{L^2(\partial\Omega^{-}_1)}+C_\beta\lVert\phi_2\rVert^2_{L^2(\partial\Omega^{-}_2)}+\beta\lVert P(f^{\eta})\rVert^2_{L^2(\partial\Omega^{-}_2)}\Big).  
		\end{align}
		Taking $\beta$ sufficiently small and combining \eqref{2.9} with \eqref{2.21}, we obtain
		\begin{align}\label{eq 3.24}
			&\lVert f^{\eta}\rVert^2_{L^2_{0,\gamma/2}}+\lVert f^{\eta}\rVert^2_{L^2(\partial\Omega^{+}_1)}+\lVert (I-P)f^{\eta}\rVert^2_{L^2(\partial\Omega^{+}_2)}\\\nonumber
			&\leq C\max\{1,\kappa^2\}\Big(\lVert \phi_1\rVert^2_{L^2_{0,-\gamma/2}}+\lVert f_{in}\rVert^2_{L^2(\partial\Omega^{-}_1)}+\lVert\phi_2\rVert^2_{L^2(\partial\Omega^{-}_2)}\Big) .
		\end{align}
		We observe that
		\[\left\{
		\begin{aligned}
			v\cdot\nabla_x  \Big(f^{\eta_1}-f^{\eta_2}\Big)+\frac{1}{\kappa}L\Big(f^{\eta_1}-f^{\eta_2}\Big)&=0,
			&&\quad\text{in}\quad\Omega\times\mathbb{R}^3, \\
			f^{\eta_1}-f^{\eta_2}&=0, 
			&&\quad\text{on}\quad \partial\Omega^{-}_1,\\
			f^{\eta_1}-f^{\eta_2}&=\eta_1 P(f^{\eta_1}-f^{\eta_2})+\Big(\eta_1-\eta_2\Big) P(f^{\eta_2}), 
			&&\quad\text{on}\quad \partial\Omega^{-}_2,
		\end{aligned}
		\right.
		\]  
		for any $0<\eta_1, \eta_2<1$. By regarding $\big(\eta_1-\eta_2\big) P(f^{\eta_2})$ as $\phi_2$, we obtain  
		\begin{align}
			&\lVert f^{\eta_1}-f^{\eta_2}\rVert^2_{L^2_{0,\gamma/2}}+\lVert f^{\eta_1}-f^{\eta_2}\rVert^2_{L^2(\partial\Omega^{+}_1)}
			+\lVert (I-P)\Big(f^{\eta_1}-f^{\eta_2}\Big)\rVert^2_{L^2(\partial\Omega^{+}_2)}\\\nonumber
			&\leq C\max\{1,\kappa^2\}  \Big(\eta_1-\eta_2\Big)^2\lVert P(f^{\eta_2})\rVert^2_{L^2(\partial\Omega^{-}_2)}\\
			&\leq C\max\{1,\kappa^2\}\max\{\kappa^{-2},\kappa^2\} \Big(\eta_1-\eta_2\Big)^2\Big(\lVert \phi_1\rVert^2_{L^2_{0,-\gamma/2}}+\lVert f_{in}\rVert^2_{L^2(\partial\Omega^{-}_1)}+\lVert\phi_2\rVert^2_{L^2(\partial\Omega^{-}_2)}\Big).		 	
		\end{align}
		This means that $\{f^{\eta}\}_{\eta}$ is a Cauchy sequence in $L^2_{0,\gamma/2}$. Therefore $\displaystyle f\coloneqq\lim_{\eta\to 1}f^{\eta}$ is the solution of \eqref{2.6}.
	\end{proof}

\section{$L^{\infty}$ estimate}\label{section 4}
	\qquad In this section, we denote $w(v)=e^{\alpha |v|^2}$. Let $h(x,v)=w(v)f(x,v)$. Then, $h$ satisfies
	\begin{align}\label{3.1}
		\left\{
		\begin{aligned}
			v\cdot\nabla_x h+\frac{1}{\kappa}L_w(h)&=w\phi_1,&&\quad\text{in}\quad\Omega\times\mathbb{R}^3, \\
			h&=wf_{in}, &&\quad\text{on}\quad\partial\Omega^{-}_1,\\
			h&=P_w(h)+w\phi_2(x,v), &&\quad\text{on}\quad \partial\Omega^{-}_2,
		\end{aligned}
		\right.
	\end{align}
	where $P_w(\cdot)\coloneq wP(w^{-1}\cdot)$ and $L_w(\cdot)\coloneq wL(w^{-1}\cdot)$. Furthermore, we denote 
	$$L_w(f)=\nu(v)f-K_\alpha(f)$$
	by defining $K_\alpha(\cdot)\coloneq wK(w^{-1}\cdot)$. 
	Note that $K_\alpha$ and $|K_\alpha|(f)\coloneq w\int_{\mathbb{R}^3}|k(v,u)|w^{-1}f(x,u)du$ have the same properties, where $k(v,u)$ denotes the kernel of $K$.
	In this section, we prove that problem \eqref{3.1} admits a unique solution in $L^\infty$ when $\Omega$ is unsealed with respect to $\partial\Omega_1$, and therefore \textbf{Theorem \ref{theorem 1.1}} in \textbf{subsection \ref{subsection 4.2}}. First, we consider the approximating problem with the parameters $\varpi$ and $\eta$, and we define a recursive sequence $h^{l}$ with $h^1=0$.
	\begin{align}\label{3.2}
		&\left\{
		\begin{aligned}
			\varpi h^{l+1}+v\cdot\nabla_x h^{l+1}+\frac{1}{\kappa}L_w(h^{l+1})&=w\phi_1,&&\quad\text{in}\quad\Omega\times\mathbb{R}^3, \\
			h^{l+1}&=wf_{in}, &&\quad\text{on}\quad\partial\Omega^{-}_1,\\
			h^{l+1}&=\eta P_w(h^{l})+w\phi_2, &&\quad\text{on}\quad \partial\Omega^{-}_2.
		\end{aligned}\right.
	\end{align}
	Second, we establish the following proposition to relate the $L^\infty$ norm of $h^l$ to the weighted $L^2$ norm of $h^l/w$. This relation allows us to apply the results from previous section to construct a solution to \eqref{3.2}  in $L^\infty$ and then remove the parameters $\varpi$ and $\eta$. This approach is known as the $L^2-L^\infty$ estimate, which is an useful tool ( \cite{Duan R 2019}, \cite{Esposito R 2013}) for constructing solutions of Boltzmann equation. Finally we note that \textbf{proposition \ref{property 3.1}} remains valid when $\varpi=0$ and $\eta=1$. The statement of the proposition is
	\begin{proposition}\label{property 3.1}
		Let $\Omega$ be a bounded set in $\mathbb{R}^3$ with $C^2$ boundary, and $\Omega$ is completely unsealed with respect to $\partial\Omega_1$. Suppose that $h^{l+1}$ is the solution of \eqref{3.2} and \eqref{eq 1.2} holds. Then, there exists $n$ depending only on the constant of unsealed set $\zeta$ such that 
		\begin{align}\label{equation 4.6}
			&\lVert h^{l+1}\rVert_{L^\infty(\Omega\times\mathbb{R}^3)}\\\nonumber
			&\leq C_{\zeta}\Big(\lVert f_{in}\rVert_{L^\infty_{\alpha,0}(\partial\Omega^-_1)}+\lVert \phi_2\rVert_{L^\infty_{\alpha,0}(\partial\Omega^-_2)}+\kappa\lVert \phi_1\rVert_{L^\infty_{a,-1}(\Omega\times\mathbb{R}^3)}\Big)\\\nonumber
			&\quad+\frac{1}{8}\max_{0\leq m\leq 2(n-1)} \lVert h^{l-m}\rVert_{L^\infty(\Omega\times\mathbb{R}^3)}+C_\zeta\min\{1,\kappa^{-1}\}\max_{0\leq m\leq 2(n-1)} \lVert h^{l+1-m}w^{-1}\rVert_{L^2_{0,\gamma/2}(\Omega\times\mathbb{R}^3)}.
		\end{align}
		for any $l\geq 2n$.
	\end{proposition}
	To prove this proposition we need the following estimate for $h^l$.  
	\begin{align}
		|h^{l+1}|(x,v)&\leq\frac{1}{\kappa}S^\kappa_\Omega |K_\alpha|(|h^{l+1}|)+S^\kappa_\Omega(|w\phi_1|)+J^\kappa\textbf{1}_{\partial\Omega^-_1}|wf_{in}|+J^\kappa\textbf{1}_{\partial\Omega^-_2}|w\phi_2|+J^\kappa\textbf{1}_{\partial\Omega^-_2}P_w(|h^{l}|).
	\end{align}
	After iterating for n times, we obtain 
	\begin{align}\label{eq 4.4}
		|h^{l+1}|(x,v)&\leq\sum_{k=1}^{4}A_k(x,v)+\sum_{j=1}^{n-1}\hat{P}^j(A^j_1)(x,v)+\sum_{k=2}^{4}\sum_{j=1}^{n-1}\hat{P}^j(A_k)(x,v)+\hat{P}^n(|h^{l+1-n}|)(x,v).
	\end{align}
	Here the definitions of $A_k^j$ and $\hat{P}^j$ are given as follows.
	\begin{enumerate}
		\item $A^j_1(x,v)=\frac{1}{\kappa}S^\kappa_\Omega |K_\alpha|(|h^{l+1-j}|),\quad  A_1(x,v)=A^0_1$.
		\item $A_2(x,v)=S^\kappa_\Omega(|w\phi_1|)$.
		\item $A_3(x,v)=J^\kappa_1|wf_{in}|(x,v)=\textbf{1}_{\partial\Omega^-_1}(q(x,v))e^{-\frac{1}{\kappa}\nu(v)\tau_{-}(x,v)}|w(v)f_{in}|(q(x,v),v)$.
		\item $A_4(x,v)=J^\kappa_2|w\phi_2|(x,v)=\textbf{1}_{\partial\Omega^-_2}(q(x,v))e^{-\frac{1}{\kappa}\nu(v)\tau_{-}(x,v)}|w(v)\phi_2|(q(x,v),v)$.
		\item $\hat{P}(h)(x,v)=J^\kappa_2wP(w^{-1}h)=J^\kappa_2\left(W \displaystyle\int hw^{-1}m^{-1/2} d\sigma_{0,1}\right)(x,v)$.
		\item $\hat{P}^j(h)(x,v)=J^\kappa_2\left(W\displaystyle\int\cdots\int hw^{-1}m^{-1/2}d\sigma_{j-1,j}d\sigma^J_{j-2,j-1}\cdots d\sigma^J_{0,1}\right)(x,v)\quad\forall j\geq 2$.
	\end{enumerate}
	The definitions of the notations appearing above are as follows
	\begin{enumerate}
		\item $W\coloneqq w(v)m^{1/2}(v)$, $x_0=q(x,v)$, $x_j\coloneqq q(x_{j-1},v_j)$ for all $j\in\mathbb{N}$.
		\item $J^\kappa_i \coloneqq {1}_{\partial\Omega_i}J^\kappa$ for $i=1,2$.
		\item $\left(\displaystyle\int hd\sigma_{j-1,j}\right)(x_{j-1},v_{j-1})\coloneqq\displaystyle\int_{v_{j}\cdot n(x_{j-1})>0}h(x_{j-1},v_j) m(v_{j})v_{j}\cdot n(x_{j-1})dv_{j}\quad\forall j\geq 1$.
		\item $\left(\displaystyle\int hd\sigma^{1}_{j-1,j}\right)(x_{j-1},v_{j-1})\coloneqq\displaystyle\int_{v_{j}\cdot n(x_{j-1})>0}h(x_{j},v_j)\textbf{1}_{\partial\Omega\setminus\partial\Omega_1}(x_j) m(v_{j})v_{j}\cdot n(x_{j-1})dv_{j}\, \forall j\geq 1$.
		\item 
		$\left(\displaystyle\int hd\sigma^J_{j-1,j}\right)(x_{j-1},v_{j-1})\coloneqq\di\int_{v_{j}\cdot n(x_{j-1})>0} J^\kappa_2h(x_{j-1},v_j)m(v_{j})v_{j}\cdot n(x_{j-1})dv_{j}\quad\forall j\geq 1$.
	\end{enumerate}
	
	\subsection{The existence of solution to \eqref{3.1}}\label{subsection 4.1}
	\qquad We now introduce the definition of an unsealed set: 
	\begin{definition}\label{definition 4.1}
		Let $\Omega\subset\mathbb{R}^3$ be a bounded domain with $C^2$ boundary, and let $\partial\Omega_1$ be an open subset of $\partial\Omega$. We say that $\Omega$ is \textbf{unsealed} with respect to $\partial\Omega_1$ if there exist $0<\zeta<1$ and $m\in\mathbb{N}$ such that
		\begin{align}\label{1.3}
			\displaystyle\int\cdots\int \textbf{1}d\sigma^1_{m-1,m}\cdots d\sigma^1_{0,1}\leq \zeta\quad\forall x\in \partial\Omega\setminus\partial\Omega_1,
		\end{align}
		In particular, $\Omega$ is said to be \textbf{completely  unsealed} with respect to $\partial\Omega_1$ when  $m=1$. 
	\end{definition}
	\begin{remark}\label{remark 4.1}
		The $m$ can be replaced by $m_T$ in \textbf{Definition \ref{definition 4.1}} . Additional information about unsealed set is given in \textbf{Appendix \ref{appendix B}}.
	\end{remark} 
	Then, we prove \textbf{proposition \ref{property 3.1}} and state our main goal of this section. The statement is as follows:
	\begin{theorem}\label{theorem 4.1}
		Let $\Omega\subset\mathbb{R}^3$ be a bounded convex domain with $C^2$ boundary and suppose $\Omega$ is unsealed with respect to $\partial\Omega_1$. Assume that $-3<\gamma\leq 1$, $0<\alpha < \frac{1}{4}$ and \eqref{eq 1.2} holds. Then, for any $\phi_1\in L^\infty_{\alpha,-1}(\Omega\times\mathbb{R}^3), f_{in}\in L^\infty_{\alpha,0}(\partial\Omega_1),$ and $\phi_2\in L^\infty_{\alpha,0}(\partial\Omega_2)$, there exists a unique solution $h\in L^\infty_{0,0}(\Omega\times\mathbb{R}^3)$ to \eqref{3.1}. Furthermore. the solution $h$ satisfies following estimate
		\begin{align}\label{eq 4.6}
			\lVert h\rVert_{L^\infty(\Omega\times\mathbb{R}^3)}+\lVert h\rVert_{L^\infty(\partial\Omega^+_2)}
			&\leq C_{\zeta}\Big(\lVert f_{in}\rVert_{L^\infty_{\alpha,0}(\partial\Omega^-_1)}+\lVert \phi_2\rVert_{L^\infty_{\alpha,0}(\partial\Omega^-_2)}+\max\{1, \kappa\}\lVert \phi_1\rVert_{L^\infty_{\alpha,-1}(\Omega\times\mathbb{R}^3)}\Big).
		\end{align}
	\end{theorem}
	For clarity, we present the proof for the case of $\Omega$ is completely unsealed with respect to $\partial\Omega_1$. The proof for unsealed case is similar. 
	In order to achieve our goal, we need the following lemmas to establish \textbf{proposition \ref{property 3.1}}. 
	\begin{lemma}\label{lemma 3.1}
		Let $-3<\gamma\leq 1$ and assume that \eqref{eq 1.2} holds. Then
		\begin{align}
			|h^{l+1}|(x,v)&\leq C_{\zeta}\Big(\lVert f_{in}\rVert_{L^\infty_{\alpha,0}(\partial\Omega^-_1)}+\lVert \phi_2\rVert_{L^\infty_{\alpha,0}(\partial\Omega^-_2)}+\kappa\lVert \phi_1\rVert_{L^\infty_{a,-1}(\Omega\times\mathbb{R}^3)}\Big)\\\nonumber
			&+\frac{\zeta^{n-1}}{1-\zeta}\max_{0\leq j\leq n-1}\lVert h^{l+1-j-n}\rVert_{L^\infty(\Omega\times\mathbb{R}^3)}+\kappa^{-2}\sum_{j=0}^{n-1}\sum_{i=0}^{n-1}\hat{P}^jS^\kappa_\Omega |K_\alpha|\hat{P}^i S^\kappa_\Omega |K_\alpha|(|h^{l+1-j-i}|)(x,v)
		\end{align}
		for any $(x,v)\in\overline{\Omega}\times\mathbb{R}^3$.
	\end{lemma}
	\begin{proof}
		For $A_3$ and $A_4$, because $\Omega$ is completely unsealed with respect to $\partial\Omega_1$, we have
		\begin{align}\label{3.11}
			\begin{cases}
				\lVert A_3\rVert_{L^\infty(\Omega\times\mathbb{R}^3)}&\leq C \lVert f_{in}\rVert_{L^\infty_{\alpha,0}(\partial\Omega^-_1)},\\
				\lVert A_4\rVert_{L^\infty(\Omega\times\mathbb{R}^3)}&\leq C \lVert w(v)\phi_2\rVert_{L^\infty(\partial\Omega^-_2)}.
			\end{cases}
		\end{align}
		Also, similar to \eqref{3.11}, we obtain
		\begin{align}\label{3.12}
			\begin{cases}
				\lVert \sum_{j=1}^{n-1}\hat{P}^j(A_3)\rVert_{L^\infty(\Omega\times\mathbb{R}^3)}&\leq C \frac{1}{1-\zeta}\lVert f_{in}\rVert_{L^\infty_{\alpha,0}(\partial\Omega^-_1)},\\
				\lVert \sum_{j=1}^{n-1}\hat{P}^j(A_4)\rVert_{L^\infty(\Omega\times\mathbb{R}^3)}&\leq C \frac{1}{1-\zeta}\lVert f_{in}\rVert_{L^\infty_{\alpha,0}(\partial\Omega^-_1)}.
			\end{cases}
		\end{align}
		Using the \textbf{Proposition \ref{proposition 2.2}}, we obtain
		\begin{align}\label{3.13}
			\lVert A_2\rVert_{L^\infty(\Omega\times\mathbb{R}^3)}&\leq C \kappa\lVert \langle v \rangle^{-1} w(v)\phi_1\rVert_{L^\infty(\partial\Omega^-_1)}
		\end{align}
		and 
		\begin{align}\label{3.14}
			\left\lVert \sum_{j=1}^{n-1}\hat{P}^j(A_2)\right\rVert_{L^\infty(\Omega\times\mathbb{R}^3)}&\leq C\frac{\kappa}{1-\zeta} \lVert \phi_1\rVert_{L^\infty_{a,-1}(\Omega\times\mathbb{R}^3)}.
		\end{align}
		Applying that $\Omega$ is completely unsealed with respect to $\partial\Omega_1$, we again yield
		\begin{align}\label{3.15}
			\lVert \hat{P}^n( h^{l+1-n})\rVert_{L^\infty(\Omega\times\mathbb{R}^3)}&\leq C \zeta^{n-1} \lVert h^{l+1-n}\rVert_{L^\infty(\Omega\times\mathbb{R}^3)}.
		\end{align}
		Combine the above result, we conclude that
		\begin{align}
			|h^{l+1}|&\leq C \frac{1}{1-\zeta}\Big(\lVert f_{in}\rVert_{L^\infty_{\alpha,0}(\partial\Omega^-_1)}+\lVert \phi_2\rVert_{L^\infty_{\alpha,0}(\partial\Omega^-_2)}+\kappa\lVert \phi_1\rVert_{L^\infty_{a,-1}(\Omega\times\mathbb{R}^3)}\Big)\\\nonumber
			&\quad+\frac{1}{\kappa}\sum_{j=0}^{n-1}\hat{P}^jS^\kappa_\Omega |K_\alpha|(|h^{l+1-j}|)+\zeta^{n-1} \lVert h^{l+1-n}\rVert_{L^\infty(\Omega\times\mathbb{R}^3)}.
		\end{align}
		For each $0<j<n-1$, replace $l+1$ to $l+1-j$, we have 
		\begin{align}\label{3.17}
			|h^{l+1-j}|&\leq C \frac{1}{1-\zeta}\Big(\lVert f_{in}\rVert_{L^\infty_{\alpha,0}(\partial\Omega^-_1)}+\lVert \phi_2\rVert_{L^\infty_{\alpha,0}(\partial\Omega^-_2)}+\kappa\lVert \phi_1\rVert_{L^\infty_{\alpha,-1}(\Omega\times\mathbb{R}^3)}\Big)\\\nonumber
			&\quad+\frac{1}{\kappa}\sum_{i=0}^{n-1}\hat{P}^i S^\kappa_\Omega |K_\alpha|(|h^{l+1-j-i}|)+\zeta^{n-1} \lVert h^{l+1-j-n}\rVert_{L^\infty(\Omega\times\mathbb{R}^3)}.
		\end{align}
		Substitute $h^{l+1-j}$ by \ref{3.17}, we obtain
		\begin{align}
			|h^{l+1}|&\leq C_{\zeta}\Big(\lVert f_{in}\rVert_{L^\infty_{\alpha,0}(\partial\Omega^-_1)}+\lVert \phi_2\rVert_{L^\infty_{\alpha,0}(\partial\Omega^-_2)}+\kappa\lVert \phi_1\rVert_{L^\infty_{a,-1}(\Omega\times\mathbb{R}^3)}\Big)\\\nonumber
			&\quad+\frac{1}{\kappa}\sum_{j=0}^{n-1}\zeta^j\left\lVert \displaystyle S^\kappa_\Omega |K_\alpha|\Big(\lVert f_{in}\rVert_{L^\infty_{\alpha,0}(\partial\Omega^-_1)}+\lVert \phi_2\rVert_{L^\infty_{\alpha,0}(\partial\Omega^-_2)}\Big)\right\rVert_{L^\infty(\Omega\times\mathbb{R}^3)}\\\nonumber
			&\quad+\frac{1}{\kappa}\sum_{j=0}^{n-1}\zeta^j\left\lVert \displaystyle S^\kappa_\Omega |K_\alpha|\Big(\kappa\lVert \phi_1\rVert_{L_{\alpha,-1}^\infty(\Omega\times\mathbb{R}^3)}+\zeta^{n-1} \lVert h^{l+1-j-n}\rVert_{L^\infty(\Omega\times\mathbb{R}^3)}\Big)\right\rVert_{L^\infty(\Omega\times\mathbb{R}^3)}\\\nonumber
			&\quad+\kappa^{-2}\sum_{j=0}^{n-1}\sum_{i=0}^{n-1}\hat{P}^jS^\kappa_\Omega |K_\alpha|\hat{P}^i S^\kappa_\Omega |K_\alpha|(|h^{l+1-j-i}|)\\\nonumber
			&\leq C_{\zeta}\Big(\lVert f_{in}\rVert_{L^\infty_{\alpha,0}(\partial\Omega^-_1)}+\lVert \phi_2\rVert_{L^\infty_{\alpha,0}(\partial\Omega^-_2)}+\kappa\lVert \phi_1\rVert_{L^\infty_{a,-1}(\Omega\times\mathbb{R}^3)}\Big)\\
			&\quad+\frac{\zeta^{n-1}}{1-\zeta}\max_{0\leq j\leq n-1}\lVert h^{l+1-j-n}\rVert_{L^\infty(\Omega\times\mathbb{R}^3)}+\kappa^{-2}\sum_{j=0}^{n-1}\sum_{i=0}^{n-1}\hat{P}^jS^\kappa_\Omega |K_\alpha|\hat{P}^i S^\kappa_\Omega |K_\alpha|(|h^{l+1-j-i}|).
		\end{align}
	\end{proof}
	For convenience, we define 
	\begin{align}
		A_{ij}(f)\coloneq\kappa^{-2}\hat{P}^jS^\kappa_\Omega |K_\alpha|\hat{P}^i S^\kappa_\Omega |K_\alpha|(f)
	\end{align} 
	and denote $\mathcal{D}=\text{diam}(\Omega)$. Now, we prove the following lemma
	\begin{lemma}\label{lemma 3.2}
		Suppose $\Omega$ is a bounded set in $\mathbb{R}^3$ with $C^2$ boundary and $\Omega$ is completely  unsealed with respect to $\partial\Omega_1$. Given $\epsilon>0$ and $i,j\in\mathbb{N}$. Suppose $-3<\gamma\leq 1$ and \eqref{eq 1.2} holds. Then 
		\begin{align}\label{eqqu 4.19}
			\lVert A_{ij}(f)\rVert_{L^\infty(\Omega\times\mathbb{R}^3)}\leq C\zeta^{j+i-2}\Big(\epsilon \lVert f\rVert_{L^\infty(\Omega\times\mathbb{R}^3)}+C_\epsilon\min\left\{1,\frac{1}{\kappa}\right\}\lVert f/w\rVert_{L^2_{0,\gamma/2}(\Omega\times\mathbb{R}^3)}\Big)
		\end{align}
	\end{lemma}
	\begin{proof}
		Without loss of generality, we assume $f\geq 0$. For $i=j=0$, we use properties of $|k_\alpha|$ like \ref{eq 2.3}, \text{proposition \ref{proposition 2.7}} , and \text{proposition \ref{proposition 2.5}}. By definition,
		\begin{align}
			A_{00}&=\kappa^{-2}\int_{0}^{\tau_{-}(x,v)}\int_{\mathbb{R}^3}\int_{0}^{\tau_{-}(x-tv,u)}\int_{\mathbb{R}^3}e^{-\frac{1}{\kappa}\nu(v)t}|k_\alpha|(v,u)e^{-\frac{1}{\kappa}\nu(u)s}|k_\alpha|(u,w)f(x-tv-su,w)dwdsdudt.
		\end{align}
		Fixed some $\epsilon_1,\epsilon_2>0$ and $N,M\in\mathbb{N}$, we define
		\begin{align}
			A=\{|v-u|>\epsilon_2, |u|<N\} \text{ and } B=\{|u-w|>\epsilon_1\}.
		\end{align}
		Moreover, we assume that $f\geq 0$. Then, we have
		\begin{align}
			&A_{00}\\
			&\nonumber\leq C \epsilon^{3-|\gamma|}_1 \lVert f\rVert_{L^\infty(\Omega\times\mathbb{R}^3)}\\\nonumber
			&+ \kappa^{-2}\int_{0}^{\tau_{-}(x,v)}\int_{\mathbb{R}^3}\int_{0}^{\tau_{-}(x-tv,u)}\int_{B}e^{-\frac{1}{\kappa}\nu(v)t}|k_\alpha|(v,u)e^{-\frac{1}{\kappa}\nu(u)s}|k_\alpha|(u,w)f(x-tv-su,w)dwdsdudt\\\nonumber
			&\leq C \Big(\epsilon^{3-|\gamma|}_1+\epsilon^{3-|\gamma|}_2\Big) \lVert f\rVert_{L^\infty(\Omega\times\mathbb{R}^3)}\\\nonumber
			&+
			\kappa^{-2}\int_{0}^{\tau_{-}(x,v)}\int_{\{|v-u|>\epsilon_2\}}\int_{0}^{\tau_{-}(x-tv,u)}\int_{B}e^{-\frac{\nu(v)t}{\kappa}}|k_\alpha|(v,u)e^{-\frac{\nu(u)s}{\kappa}}|k_\alpha|(u,w)f(x-tv-su,w)dwdsdudt\\\nonumber
			&\leq C \Big(\epsilon^{3-|\gamma|}_1+\epsilon^{3-|\gamma|}_2+\frac{1}{1+N}\Big) \lVert f\rVert_{L^\infty(\Omega\times\mathbb{R}^3)}\\\nonumber
			&+
			\kappa^{-2}\int_{0}^{\tau_{-}(x,v)}\int_{A}\int_{0}^{\tau_{-}(x-tv,u)}\int_{B}e^{-\frac{1}{\kappa}\nu(v)t}|k_\alpha|(v,u)e^{-\frac{1}{\kappa}\nu(u)s}|k_\alpha|(u,w)f(x-tv-su,w)dwdsdudt\\\nonumber
			&\leq C \Big(\epsilon^{3-|\gamma|}_1+\epsilon^{3-|\gamma|}_2+\frac{1}{1+N}+\frac{1}{N}\Big) \lVert f\rVert_{L^\infty(\Omega\times\mathbb{R}^3)}\\
			&+\kappa^{-2}
			\int_{0}^{\tau_{-}(x,v)}\int_{A}\int_{1/N}^{\tau_{-}(x-tv,u)}\int_{B}e^{-\frac{1}{\kappa}\nu(v)t}|k_\alpha|(v,u)e^{-\frac{1}{\kappa}\nu(u)s}|k_\alpha|(u,w)f(x-tv-su,w)dwdsdudt.
		\end{align}
		Using H\"older inequality and noticing that $k(v,u),k(u,w)$ are bounded on $A,B$ respectively. Let $y=x-tv-su$, we obtain
		\begin{align}
			A_{00}&\nonumber\leq C \Big(\epsilon^{3-|\gamma|}_1+\epsilon^{3-|\gamma|}_2+\frac{1}{1+N}+\frac{1}{N}+\frac{1}{1+M}\Big) \lVert f\rVert_{L^\infty(\Omega\times\mathbb{R}^3)}\\\nonumber
			&+C\textbf{1}_{|v|<M}\kappa^{-2}\int_{\frac{\mathcal{D}}{5M}}^{\frac{\mathcal{D}}{|v|}}\int_{A}\int_{1/N}^{\infty}\int_{B}e^{-\frac{1}{\kappa}\nu(v)t}|k_\alpha|(v,u)e^{-\frac{1}{\kappa}\nu(u)s}|k_\alpha|(u,w)f(x-tv-su,w)dwdsdudt\\\nonumber
			&\leq C \Big(\epsilon^{3-|\gamma|}_1+\epsilon^{3-|\gamma|}_2+\frac{1}{1+N}+\frac{1}{N}+\frac{1}{M}\Big) \lVert f\rVert_{L^\infty(\Omega\times\mathbb{R}^3)}\\\nonumber
			&+C\kappa^{-2}\textbf{1}_{|v|<M}\Big(\int_{\frac{\mathcal{D}}{5M}}^{\frac{\mathcal{D}}{|v|}}\int_{\mathbb{R}^3}\int_{1/N}^{\infty}\int_{\mathbb{R}^3}e^{-\frac{C_M}{\kappa}t}\frac{1}{s^3}\nu(w)f^2(y,w)w^{-2}(w)dwdsdydt\Big)^{1/2}\\
			&\leq C \Big(\epsilon^{3-|\gamma|}_1+\epsilon^{3-|\gamma|}_2+\frac{1}{1+N}+\frac{1}{N}+\frac{1}{M}\Big) \lVert f\rVert_{L^\infty(\Omega\times\mathbb{R}^3)}+C\min\left\{1,\frac{1}{\kappa}\right\}\lVert f/w\rVert_{L^2_{0,\gamma/2}(\Omega\times\mathbb{R}^3)}.
		\end{align}
		Hence, if $\epsilon_1,\epsilon_2$ are sufficiently small and $N,M$ are chosen large enough, the proof for $A_{00}$ is complete. For $i$ or $j\neq 0$, since $wm^{1/2}\leq C$ and $\Omega$ is completely unsealed with respect to $\partial\Omega_1$, it suffice to estimate $S^\kappa_\Omega |K_\alpha| \hat{P}^i S^\kappa_\Omega |K_\alpha|(f)$. Furthermore, we can observe that the character of $|K_\alpha|$ can be replace by $\hat{P}^i$. Therefore, the estimate $A_{ij}$ is similar to $A_{00}$ and it is sufficient to estimate $\frac{1}{\kappa}\hat{P}^i S^\kappa_\Omega |K_\alpha|(f)$. For instance, let $i=1$ and denote $q=q(x,v)$,
		\begin{align}
			&\hat{P} S^\kappa_\Omega |K_\alpha|(f)(x,v)\\\nonumber
			&=\frac{1}{\kappa}\di\int W(v)J^-_2(q(x,v),v)S^\kappa_\Omega |K_\alpha|(f)(q(x,v), v_1)w^{-1}(v_1)d\sigma_{0,1}\\\nonumber
			&\leq \frac{1}{\kappa}\di\int_{v_1\cdot n(q)>0}\int_{0}^{\tau_{-}(q,v_1)}\int_{\mathbb{R}^3} e^{-\frac{1}{\kappa}\nu(v_1)t}|k_\alpha|(v_1,u)f(q-tv_1,u)w^{-1}(v_1)m^{1/2}(v_1)v_1\cdot n(q)dudtdv_1\\\nonumber
			&=\frac{1}{\kappa}\di\int_{v_1\cdot n(q)>0}\int_{0}^{\tau_{-}(q,v_1)}\int_{\mathbb{R}^3} e^{-\frac{1}{\kappa}\nu(v_1)t}|k|(v_1,u)f(q-tv_1,u)w^{-1}(u)m^{1/2}(v_1)v_1\cdot n(q)dudtdv_1.
		\end{align}
		Similar to $A_{00}$, fix some $\epsilon_1>0$ and $N\in\mathbb{N}$, we define 
		\begin{align}
			A=\{ |v_1|<N\},\quad B=\{|v_1-u|>\epsilon_1\}.
		\end{align} 
		Then,
		\begin{align}
			&\hat{P} S^\kappa_\Omega |K_\alpha|(f)(x,v)\\\nonumber
			&\leq C \epsilon^{3-|\gamma|}_1 \lVert f\rVert_{L^\infty(\Omega\times\mathbb{R}^3)}\\\nonumber
			&\quad+\frac{1}{\kappa}\di\int_{v_1\cdot n(q)>0}\int_{0}^{\frac{\mathcal{D}}{|v_1|}}\int_{B} e^{-\frac{1}{\kappa}\nu(v_1)t}|k|(v_1,u)f(q-tv_1,u)w^{-1}(u)m^{1/2}(v_1)v_1\cdot n(q)dudtdv_1\\\nonumber
			&\leq C \Big(\epsilon^{3-|\gamma|}_1+\frac{1}{1+N}\Big) \lVert f\rVert_{L^\infty(\Omega\times\mathbb{R}^3)}\\
			&\quad+\frac{1}{\kappa}\di\int_{\{v_1\cdot n(q)>0\}\cap A}\int_{0}^{\frac{\mathcal{D}}{|v_1|}}\int_{B} e^{-\frac{1}{\kappa}\nu(v_1)t}|k|(v_1,u)f(q-tv_1,u)w^{-1}(u)m^{1/2}(v_1)v_1\cdot n(q)dudtdv_1.
		\end{align}
		Using H\"older inequality and notice that $k(v_1,u)$ is bounded on $B$. Let $y=q(x,v)-tv_1$. We obtain
		\begin{align}\label{eq 4.32}
			&\leq C \Big(\epsilon^{3-|\gamma|}_1+\frac{1}{1+N}\Big)\lVert f\rVert_{L^\infty(\Omega\times\mathbb{R}^3)}\\\nonumber
			&\quad+\frac{1}{\kappa}\di\int_{\{v_1\cdot n(q)>0\}\cap A}\int_{\frac{\mathcal{D}}{N}}^{\mathcal{D}/|v_1|}\int_{B}e^{-\frac{1}{\kappa}\nu(v_1)t}|k|(v_1,u)f(q-tv_1,u)w^{-1}(u)m^{1/2}(v_1)v_1\cdot n(q)dudtdv_1\\\nonumber
			&\leq C \Big(\epsilon^{3-|\gamma|}_1+\frac{1}{1+N}\Big)\lVert f\rVert_{L^\infty(\Omega\times\mathbb{R}^3)}\\
			&\quad+\frac{C}{\kappa}\Big(\di\int_{\{v_1\cdot n(q)>0\}\cap A}\int_{\frac{\mathcal{D}}{N}}^{\frac{\mathcal{D}}{|v_1|}}\int_{B}e^{-\frac{C_N}{\kappa}t}|k|(v_1,u)f^2(q-tv_1,u)w^{-2}(u)dudtdv_1\Big)^{1/2}\\\nonumber
			&\leq C \Big(\epsilon^{3-|\gamma|}_1+\frac{1}{1+N}\Big)\lVert f\rVert_{L^\infty(\Omega\times\mathbb{R}^3)}\\\nonumber
			&\quad+C\min\left\{1,\frac{1}{\kappa}\right\}\Big(\di\int_{\Omega}\int_{\frac{\mathcal{D}}{N}}^{\infty}\int_{\mathbb{R}^3}\frac{1}{t^3}\nu(u)f^2(y,u)w^{-2}(u)dudtdy\Big)^{1/2}\\
			&\leq C\Big(\epsilon^{3-|\gamma|}_1+\frac{1}{1+N}\Big)\lVert f\rVert_{L^\infty(\Omega\times\mathbb{R}^3)}+C\min\left\{1,\frac{1}{\kappa}\right\}\lVert f/w\rVert_{L^2_{0,\gamma/2}(\Omega\times\mathbb{R}^3)}.
		\end{align}
		Hence, when $\epsilon_1$ is sufficiently small and $N$ is sufficiently large, the proof for $A_{ij}$ is complete.
	\end{proof}
	\begin{proof}\textbf{The proof of property \ref{property 3.1}}\par
		\qquad We use \textbf{Lemma \ref{lemma 3.1}} and \textbf{Lemma \ref{lemma 3.2}} to obtain 
		\begin{align}
			&|h^{l+1}|(x,v)\\\nonumber
			&\leq C_{\zeta}\Big(\lVert f_{in}\rVert_{L^\infty_{\alpha,0}(\partial\Omega^-_1)}+\lVert \phi_2\rVert_{L^\infty_{\alpha,0}(\partial\Omega^-_2)}+\kappa\lVert \phi_1\rVert_{L^\infty_{a,-1}(\Omega\times\mathbb{R}^3)}\Big)\\\nonumber
			&\quad+\frac{\zeta^{n-1}}{1-\zeta}\max_{0\leq j\leq n-1}\lVert h^{l+1-j-n}\rVert_{L^\infty(\Omega\times\mathbb{R}^3)}\\\nonumber
			&\quad+\frac{C}{\zeta^2(1-\zeta)^2}\Big(\epsilon \max_{0\leq i,j\leq n-1}\lVert h^{l+1-i-j}\rVert_{L^\infty(\Omega\times\mathbb{R}^3)}+C_\epsilon\min\left\{1,\frac{1}{\kappa}\right\}\max_{0\leq i,j\leq n-1}\lVert h^{l+1-i-j}/w\rVert_{L^2_{0,\gamma/2}(\Omega\times\mathbb{R}^3)}\Big).
		\end{align}
		Taking $n$ is enough large and $\epsilon$ is sufficiently small, we have
		\begin{align}\label{3.38}
			&\lVert h^{l+1}\rVert_{L^\infty(\Omega\times\mathbb{R}^3)}\\\nonumber
			&\leq C_{\zeta}\Big(\lVert f_{in}\rVert_{L^\infty_{\alpha,0}(\partial\Omega^-_1)}+\lVert \phi_2\rVert_{L^\infty_{\alpha,0}(\partial\Omega^-_2)}+\kappa\lVert \phi_1\rVert_{L^\infty_{a,-1}(\Omega\times\mathbb{R}^3)}\Big)\\\nonumber
			&\quad+\frac{1}{8}\max_{0\leq m\leq 2(n-1)}\lVert h^{l-m}\rVert_{L^\infty(\Omega\times\mathbb{R}^3)}+C_\zeta\min\left\{1,\frac{1}{\kappa}\right\}\max_{0\leq m\leq 2(n-1)}\lVert h^{l+1-m}/w\rVert_{L^2_{0,\gamma/2}(\Omega\times\mathbb{R}^3)}.
		\end{align}
		The proof is complete.
	\end{proof}
	\begin{proof}\textbf{The proof of \textbf{Theorem \ref{theorem 4.1}} }\par
		\qquad
		Now, we construct the solution of \eqref{3.1} by starting from \eqref{3.2}. It is clear that \textbf{Property \ref{property 3.1}}  still holds when $\varpi=0$ or $\eta=1$. We rewrite \eqref{3.38} as 
		\begin{align}
			&\lVert h^{i+2(n-1)+1}\rVert_{L^\infty(\Omega\times\mathbb{R}^3)}\\\nonumber
			&\leq C_{\zeta}\Big(\lVert f_{in}\rVert_{L^\infty_{\alpha,0}(\partial\Omega^-_1)}+\lVert \phi_2\rVert_{L^\infty_{\alpha,0}(\partial\Omega^-_2)}+\kappa\lVert \phi_1\rVert_{L^\infty_{a,-1}(\Omega\times\mathbb{R}^3)}\Big)\\\nonumber
			&\quad+\frac{1}{8}\max_{0\leq m\leq 2(n-1)}\lVert h^{i+m}\rVert_{L^\infty(\Omega\times\mathbb{R}^3)}+C_\zeta\min\left\{1,\frac{1}{\kappa}\right\}\max_{0\leq m\leq 2(n-1)}\lVert h^{i+1+m}/w\rVert_{L^2_{0,\gamma/2}(\Omega\times\mathbb{R}^3)},
		\end{align}
		for any $i=0,1,\cdots$. By \textbf{Lemma \ref{lemma appendix B 6.1}}  with $a_i=\lVert h^{i}\rVert_{L^\infty(\Omega\times\mathbb{R}^3)}$ and $k=2(n-1)$ and notice that
		\begin{align}
			\lVert h^{i}/w\rVert_{L^2_{0,\gamma/2}(\Omega\times\mathbb{R}^3)}&=\lVert f^{i}\rVert_{L^2_{0,\gamma/2}(\Omega\times\mathbb{R}^3)}\\\nonumber
			&\leq C_\eta\max\{1,\kappa\}\Big(\lVert f_{in}\rVert_{L^2(\partial\Omega^-_1)}+\lVert \phi_2\rVert_{L^2(\partial\Omega^-_1)}+\lVert \phi_1\rVert_{L^2_{0,\gamma/2}(\Omega\times\mathbb{R}^3)}\Big),
		\end{align}
		for any $i\in\mathbb{N}$. Then 
		\begin{align}\label{3.42}
			\max_{0\leq m\leq 2(n-1)}\lVert h^{i+m}\rVert_{L^\infty(\Omega\times\mathbb{R}^3)}
			&\leq \left(\frac{1}{8}\right)^{[\frac{i}{2n-1}]}\max_{0\leq m\leq 4(n-1)}\lVert h^{m}\rVert_{L^\infty(\Omega\times\mathbb{R}^3)}\\\nonumber
			&\quad+C_{\zeta}\Big(\lVert f_{in}\rVert_{L^\infty_{\alpha,0}(\partial\Omega^-_1)}+\lVert \phi_2\rVert_{L^\infty_{\alpha,0}(\partial\Omega^-_2)}+\kappa\lVert \phi_1\rVert_{L^\infty_{a,-1}(\Omega\times\mathbb{R}^3)}\Big)\\\nonumber
			&\quad+C_{\varpi,\eta}\Big(\lVert f_{in}\rVert_{L^2(\partial\Omega^-_1)}+\lVert \phi_2\rVert_{L^2(\partial\Omega^-_1)}+\lVert \phi_1\rVert_{L^2_{0,\gamma/2}(\Omega\times\mathbb{R}^3)}\Big).
		\end{align}
		The statement of \textbf{Lemma \ref{lemma appendix B 6.1}} is
		\begin{lemma}[\cite{Duan R 2019}(lemma 6.1)]\label{lemma appendix B 6.1}
			Let $\{a_i\}_{i=0}\in\mathbb{R}_+$. For any fixed $k\in\mathbb{N}$, we denote $$A^k_i=\max\{a_i,a_{i+1},\cdots,a_{i+k}\}$$.
			\begin{enumerate}
				\item Assume $D\geq 0$. If $a_{i+1+k}\leq \frac{1}{8}A_i^k+D$ for $i=0,1,\cdots,$ then it holds that
				\begin{align}
					A_i^k\leq\left(\frac{1}{8}\right)^{[\frac{i}{k+1}]}\max\{A_0^k,A_1^k,\cdots,A_k^k\}+\frac{8+k}{7}D,\quad \text{for}\, i\leq k+1
				\end{align}
				\item Let $0\leq \eta <1$ with $\eta^{k+1}\geq \frac{1}{4}$. If $a_{i+1+k}\leq \frac{1}{8}A_i^k+C_k\eta^{i+k+1}$ for $i=0,1,\cdots$, then it holds that 
				\begin{align}
					A_i^k\leq\left(\frac{1}{8}\right)^{[\frac{i}{k+1}]}\max\{A_0^k,A_1^k,\cdots,A_k^k\}+2C_k\frac{8+k}{7}D,\quad \text{for}\, i\leq k+1
				\end{align}
			\end{enumerate}
		\end{lemma}
		The proof of \textbf{Lemma \ref{lemma appendix B 6.1}} is omitted as it appears in the appendix of \cite{Duan R 2019}. Now, we should control $\di\max_{0\leq m\leq 4(n-1)}\lVert h^{m}\rVert_{L^\infty(\Omega\times\mathbb{R}^3)}$. We can use same techniques in the proof of \textbf{Property \ref{property 3.1}} and induction to show that
		\begin{align}\label{3.45}
			\lVert h^{m}\rVert_{L^\infty(\Omega\times\mathbb{R}^3)}&\leq C_{\zeta}\Big(\lVert f_{in}\rVert_{L^\infty_{\alpha,0}(\partial\Omega^-_1)}+\lVert \phi_2\rVert_{L^\infty_{\alpha,0}(\partial\Omega^-_2)}+\kappa\lVert \phi_1\rVert_{L^\infty_{a,-1}(\Omega\times\mathbb{R}^3)}\Big)\\\nonumber
			&\quad+C_{\zeta,\varpi,\eta}\Big(\lVert f_{in}\rVert_{L^2(\partial\Omega^-_1)}+\lVert \phi_2\rVert_{L^2(\partial\Omega^-_1)}+\lVert\phi_1\rVert_{L^2_{0,\gamma/2}(\Omega\times\mathbb{R}^3)}\Big),
		\end{align}
		for $0\leq m\leq 4(n-1)$. Here, $C_{n,\eta},C_{n,\zeta}$ are increasing with respect to $n$. In particular, combining \eqref{3.42} and \eqref{3.45}, we have
		\begin{align}\label{hl is bounded}
			\lVert h^{l}\rVert_{L^\infty(\Omega\times\mathbb{R}^3)}
			&\leq \left(\frac{1}{8}\right)^{[\frac{l}{2n-1}]}C_{\zeta}\Big(\lVert f_{in}\rVert_{L^\infty_{\alpha,0}(\partial\Omega^-_1)}+\lVert \phi_2\rVert_{L^\infty_{\alpha,0}(\partial\Omega^-_2)}+\kappa\lVert \phi_1\rVert_{L^\infty_{a,-1}(\Omega\times\mathbb{R}^3)}\Big)\\\nonumber
			&\qquad+C_{\zeta,\varpi,\eta}\Big(\lVert f_{in}\rVert_{L^\infty_{\alpha,0}(\partial\Omega^-_1)}+\lVert \phi_2\rVert_{L^\infty_{\alpha,0}(\partial\Omega^-_2)}+\lVert \phi_1\rVert_{L^\infty_{a,-1}(\Omega\times\mathbb{R}^3)}\Big),
		\end{align}
		
		since $L^2_{0,\gamma/2}(\Omega\times\mathbb{R}^3)\subset L^\infty_{\alpha,0}(\Omega\times\mathbb{R}^3)$.\par
		\qquad Now, we show that $\{h^{l}\}$ is a Cauchy sequence in $L^\infty(\Omega\times\mathbb{R}^3)$. Consider $2n\ll m_1,m_2$, we notice that $h^{m_1+1}-h^{m_2+1}$ satisfies
		\[\left\{
		\begin{aligned}
			&\varpi\nu(v)\left(h^{m_1+1}-h^{m_2+1}\right)+v\cdot\nabla_x \left(h^{m_1+1}-h^{m_2+1}\right)+\frac{1}{\kappa}L_w\left(h^{m_1+1}-h^{m_2+1}\right)=0,
			&&\quad\text{in}\quad\Omega\times\mathbb{R}^3, \\
			&\left(h^{m_1+1}-h^{m_2+1}\right)=0, &&\quad\text{on}\quad\partial\Omega^{-}_1,\\
			&\left(h^{m_1+1}-h^{m_2+1}\right)=\eta P_w\left(h^{m_1}-h^{m_2}\right), 
			&&\quad\text{on}\quad \partial\Omega^{-}_2.
		\end{aligned}\right.
		\]
		By the same procedure as in \textbf{Property \ref{property 3.1}}, we obtain
		\begin{align}
			\lVert h^{m_1+1}-h^{m_2+1}\rVert_{L^\infty(\Omega\times\mathbb{R}^3)}&\leq\frac{1}{8}\max_{0\leq m\leq 2(n-1)} \lVert h^{m_1-m}-h^{m_2-m}\rVert_{L^\infty(\Omega\times\mathbb{R}^3)}\\\nonumber
			&+C_\zeta\min\{1,\kappa^{-1}\}\max_{0\leq m\leq 2(n-1)} \lVert f^{m_1+1-m}-f^{m_2+1-m}\rVert_{L^2_{0,\gamma/2}(\Omega\times\mathbb{R}^3)}.
		\end{align}
		Note that $\{f_m\}$ converges in $L^2_{0,\gamma/2}$ (see the proof of \textbf{Theorem \ref{theorem 3.1}}). Given $\epsilon>0$, by definition, there exists $M$ such that for any $m_1,m_2> M$, we have 
		\begin{align}
			C_\zeta\max_{0\leq m\leq 2(n-1)} \lVert f^{m_1+1-m}-f^{m_2+1-m}\rVert_{L^2_{0,\gamma/2}(\Omega\times\mathbb{R}^3)}<\frac{\epsilon}{2(2n+7)/7}.
		\end{align}
		we then obtain
		\begin{align}
			\lVert h^{m_1+1}-h^{m_2+1}\rVert_{L^\infty(\Omega\times\mathbb{R}^3)}&\leq\frac{1}{8}\max_{0\leq m\leq 2(n-1)} \lVert h^{m_1-m}-h^{m_2-m}\rVert_{L^\infty(\Omega\times\mathbb{R}^3)}+\epsilon.
		\end{align}
		By \eqref{hl is bounded}, we find that
		\begin{align}
			\lVert h^{m_1}-h^{m_2}\rVert_{L^\infty(\Omega\times\mathbb{R}^3)}&\leq C_{f_{in},\phi_1,\phi_2}C_{n,\eta,\zeta}\left(\frac{1}{8}\right)^{[\frac{\min\{m_1,m_2\}}{2n-1}]}+\epsilon/2.
		\end{align}
		Hence, if $m_1,m_2$ satisfy 
		\begin{align*}
			C_{f_{in},\phi_1,\phi_2}C_{n,\eta,\zeta}\left(\frac{1}{8}\right)^{[\frac{\min\{m_1,m_2\}}{2n-1}]}<\epsilon/2,
		\end{align*}
		then it follows that $\lVert h^{m_1}-h^{m_2}\rVert_{L^\infty(\Omega\times\mathbb{R}^3)}<\epsilon$. From the discussion above, $\{h^{l}\}$ is a Cauchy sequence. Therefore, the limit $\di\lim_{l\to\infty}h^{l}=h^{\varpi,\eta}$ exists for any $\alpha,\eta$. Moreover, $h^{\varpi,\eta}$ satisfies
		\[\left\{
		\begin{aligned}
			\varpi\nu(v)h^{\varpi,\eta}+v\cdot\nabla_x h^{\varpi,\eta}+\frac{1}{\kappa}L_w(h^{\varpi,\eta})&=w\phi_1,&&\quad\text{in}\quad\Omega\times\mathbb{R}^3, \\
			h^{\varpi,\eta}&=wf_{in}, &&\quad\text{on}\quad\partial\Omega^{-}_1,\\
			h^{\varpi,\eta}&=\eta P_w(h^{\varpi,\eta})+w\phi_2, &&\quad\text{on}\quad\partial\Omega^{-}_2.
		\end{aligned}\right.
		\]
		Moreover, we observe that the argument used in \textbf{Property \ref{property 3.1}} also applies to $h^{\varpi,\eta}$ for any fixed $0<\eta<1$. Then we obtain
		\begin{align}\label{eq 4.43}
			\lVert h^{\varpi,\eta}\rVert_{L^\infty(\Omega\times\mathbb{R}^3)}
			&\leq C_{\zeta}\Big(\lVert f_{in}\rVert_{L^\infty_{\alpha,0}(\partial\Omega^-_1)}+\lVert \phi_2\rVert_{L^\infty_{\alpha,0}(\partial\Omega^-_2)}+\kappa\lVert \phi_1\rVert_{L^\infty_{a,-1}(\Omega\times\mathbb{R}^3)}\Big)\\\nonumber
			&\qquad+\frac{C_\zeta}{\kappa}\lVert h^{\varpi,\eta}/w\rVert_{L^2_{0,\gamma/2}(\Omega\times\mathbb{R}^3)}\\\nonumber
			&= C_{\zeta}\Big(\lVert f_{in}\rVert_{L^\infty_{\alpha,0}(\partial\Omega^-_1)}+\lVert \phi_2\rVert_{L^\infty_{\alpha,0}(\partial\Omega^-_2)}+\kappa\lVert \phi_1\rVert_{L^\infty_{a,-1}(\Omega\times\mathbb{R}^3)}\Big)\\\nonumber
			&\qquad+C_\zeta\min\{1,\kappa^{-1}\}\lVert f^{\varpi,\eta}\rVert_{L^2_{0,\gamma/2}(\Omega\times\mathbb{R}^3)}\\\nonumber
			&\leq C_{\zeta}\Big(\lVert f_{in}\rVert_{L^\infty_{\alpha,0}(\partial\Omega^-_1)}+\lVert \phi_2\rVert_{L^\infty_{\alpha,0}(\partial\Omega^-_2)}+\max\{1,\kappa\}\lVert \phi_1\rVert_{L^\infty_{a,-1}(\Omega\times\mathbb{R}^3)}\Big)\\\nonumber
			&\qquad+C_{\zeta,\varpi,\eta}\Big(\lVert f_{in}\rVert_{L^\infty_{\alpha,0}(\partial\Omega^-_1)}+\lVert \phi_2\rVert_{L^\infty_{\alpha,0}(\partial\Omega^-_2)}+\lVert \phi_1\rVert_{L^\infty_{a,-1}(\Omega\times\mathbb{R}^3)}\Big)
		\end{align}
		Let $0<\varpi_1,\varpi_2<1$. We recall that 
		\[\left\{
		\begin{aligned}
			v\cdot\nabla_x \left(h^{\varpi_1,\eta}-h^{\varpi_2,\eta}\right)+\frac{1}{\kappa}L_w\left(h^{\varpi_1,\eta}-h^{\varpi_2,\eta}\right)&=-\varpi_1h^{\varpi_1,\eta}+\varpi_2h^{\varpi_2,\eta},
			&&\quad\text{in}\quad\Omega\times\mathbb{R}^3, \\
			\left(h^{\varpi_1,\eta}-h^{\varpi_2,\eta}\right)&=0, &&\quad\text{on}\quad\partial\Omega^{-}_1,\\
			\left(h^{\varpi_1,\eta}-h^{\varpi_2,\eta}\right)&=\eta P_w\left(h^{\varpi_1,\eta}-h^{\varpi_2,\eta}\right), &&\quad\text{on}\quad\partial\Omega^{-}_2.
		\end{aligned}\right.
		\]
		Regard  $-\varpi_1h^{\varpi_1,\eta}+\varpi_2h^{\varpi_2,\eta}$ as $\phi_1$. By repeating the same process as above, we obtain
		\begin{align}
			&\lVert h^{\varpi_1,\eta}-h^{\varpi_2,\eta}\rVert_{L^\infty(\Omega\times\mathbb{R}^3)}\\[2mm]\nonumber
			&\leq C_{\zeta}\cdot\kappa\cdot\max\{\varpi_1,\varpi_2\}\left(\lVert \langle v \rangle^{-1} h^{\varpi_1,\eta}\rVert_{L^\infty(\Omega\times\mathbb{R}^3)}+\langle v \rangle^{-1} h^{\varpi_2,\eta}\rVert_{L^\infty(\Omega\times\mathbb{R}^3)}\right)\\\nonumber
			&\quad+C_\zeta\min\{1,\kappa^{-1}\}\lVert f^{\varpi_1,\eta}-f^{\varpi_2,\eta}\rVert_{L^2_{0,\gamma/2}(\Omega\times\mathbb{R}^3)}\\\nonumber
			&\leq C_{\zeta}\cdot\kappa\cdot\max\{\varpi_1,\varpi_2\}\left(\lVert f_{in}\rVert_{L^\infty_{\alpha,0}(\partial\Omega^-_1)}+\lVert \phi_2\rVert_{L^\infty_{\alpha,0}(\partial\Omega^-_2)}+\kappa\lVert \phi_1\rVert_{L^\infty_{a,-1}(\Omega\times\mathbb{R}^3)}\right)\\\nonumber
			&\quad+C_\zeta\min\{1,\kappa^{-1}\}\lVert f^{\varpi_1,\eta}-f^{\varpi_2,\eta}\rVert_{L^2_{0,\gamma/2}(\Omega\times\mathbb{R}^3)}.
		\end{align}
		This implies that $\{h^{\varpi,\eta}\}$ is a Cauchy sequence in $L^\infty(\Omega\times\mathbb{R}^3)$. After we denote $h^{\eta}=\di\lim_{\alpha\to 0}h^{\varpi,\eta}$, note that $h^{\eta_1}-h^{\eta_2}$ satisfies 
		\[\left\{
		\begin{aligned}
			&v\cdot\nabla_x \left(h^{\eta_1}-h^{\eta_2}\right)+\frac{1}{\kappa}L_w\left(h^{\eta_1}-h^{\eta_2}\right)=0,
			&&\quad\text{in}\quad\Omega\times\mathbb{R}^3, \\
			&\left(h^{\eta_1}-h^{\eta_2}\right)=0, &&\quad\text{on}\quad\partial\Omega^{-}_1,\\
			&\left(h^{\eta_1}-h^{\eta_2}\right)=\eta_1P_w\left(h^{\eta_1}-h^{\eta_2}\right)+\left(\eta_1-\eta_2\right) P_w\left(h^{\eta_2}\right), &&\quad\text{on}\quad\partial\Omega^{-}_2.
		\end{aligned}\right.
		\]
		We can regard $\left(\eta_1-\eta_2\right) P_w\left(h^{\eta_2}\right)$ as $\phi_2$ and apply the same argument again. We obtain   
		\begin{align}
			\lVert h^{\eta_1}-h^{\eta_2}\rVert_{L^\infty(\Omega\times\mathbb{R}^3)}
			&\leq C_{\zeta}|\eta_1-\eta_2|\cdot\lVert P_w(h^{\eta_2})\rVert_{L^\infty(\partial\Omega^-_2)}+C_\zeta\min\{1,\kappa^{-1}\}\lVert f^{\eta_1}-f^{\eta_2}\rVert_{L^2_{0,\gamma/2}(\Omega\times\mathbb{R}^3)}\\\nonumber
			&\leq C_{\zeta}|\eta_1-\eta_2|\cdot\lVert h^{\eta_2}\rVert_{L^\infty(\partial\Omega^+)}+C_\zeta\min\{1,\kappa^{-1}\}\lVert f^{\eta_1}-f^{\eta_2}\rVert_{L^2_{0,\gamma/2}(\Omega\times\mathbb{R}^3)}\\\nonumber
			&\leq C_{\zeta}|\eta_1-\eta_2|\cdot\left(\lVert f_{in}\rVert_{L^\infty_{\alpha,0}(\partial\Omega^-_1)}+\lVert \phi_2\rVert_{L^\infty_{\alpha,0}(\partial\Omega^-_2)}+\kappa\lVert \phi_1\rVert_{L^\infty_{a,-1}(\Omega\times\mathbb{R}^3)}\right)\\\nonumber
			&+C_\zeta\min\{1,\kappa^{-1}\}\lVert f^{\eta_1}-f^{\eta_2}\rVert_{L^2_{0,\gamma/2}(\Omega\times\mathbb{R}^3)}.
		\end{align}
		The last inequality follows from a similar argument showing that $\lVert h^{\varpi,\eta}\rVert_{L^\infty}$ is bounded with respect to $\varpi,\eta$. Thus $\{h^{\eta}\}_\eta$ is a Cauchy sequence in $L^\infty$ and $h\coloneq\lim_{\eta\to 1}h^{\eta}$ is a mild solution of \eqref{3.1}. The proof of the estimate \eqref{eq 4.6} is similar to the proof of \eqref{eq 4.43}.
	\end{proof}
	\begin{remark}
		the proof for the unseal set is similar to the case of the complete unsealed set. We just take $[\frac{n}{m}]$ sufficiently large in the proof of \textbf{Proposition} \ref{property 3.1} to make $\zeta^{[\frac{n}{m}]}$ is small enough.
	\end{remark}
	\subsection{the proof of \textbf{Theorem \ref{theorem 1.1}}}
	\label{subsection 4.2}
	Now, we conclude this section by proving \textbf{Theorem \ref{theorem 1.1}}.
	\begin{proof}\textbf{the proof of theorem \ref{theorem 1.1}}
		We consider the following iteration with $f_1=0$.
		\begin{align}\label{eq 4.48}
			\left\{
			\begin{aligned}
				v\cdot\nabla_x f_{k+1}+\frac{1}{\kappa}L(f_{k+1})=&\frac{1}{\kappa}\Gamma(f_k,f_k),&&\quad\text{in}\quad \Omega\times\mathbb{R}^3, \\
				f_{k+1}=&f_{in}, &&\quad\text{on}\quad \partial\Omega^{-}_1,\\
				f_{k+1}=&P(f_{k+1})+m^{-1}(m_{T(x)}-m)P(f_k)\\
				&\quad+\frac{1}{\sqrt{2\pi}}m^{-1/2}(m_{T(x)}-m), &&\quad\text{on} \quad\partial\Omega^{-}_2.
			\end{aligned}
			\right.
		\end{align}
		By \textbf{Property \ref{property 3.1}}, we obtain 
		\begin{align}\label{eq 4.47}
			\lVert f_{k+1}\rVert_{L^\infty_{\alpha,0}(\Omega\times\mathbb{R}^3)}+\lVert f_{k+1}\rVert_{L^\infty_{\alpha,0}(\partial\Omega^+_2)}&\leq C_\zeta
			\begin{pmatrix*}[l]
				&\lVert f_{in}\rVert_{L^\infty_{\alpha,0}(\partial\Omega^-_1)}\\
				+&\lVert m^{-1}(m_{T(x)}-m)P(f_k)\rVert_{L^\infty_{\alpha,0}(\partial\Omega^-_2)}\\
				+&\lVert \frac{1}{\sqrt{2\pi}}m^{-1/2}(m_{T(x)}-m)\rVert_{L^\infty_{\alpha,0}(\partial\Omega^-_2)}\\\nonumber
				+&\max\{1,\kappa^{-1}\}\lVert \Gamma(f_k,f_k)\rVert_{L^\infty_{a,-1}(\Omega\times\mathbb{R}^3)}
			\end{pmatrix*}\\
			&\leq C_\zeta
			\begin{pmatrix*}[l]
				&\lVert f_{in}\rVert_{L^\infty_{\alpha,0}(\partial\Omega^-_1)}\\
				+&C\lVert m^{-1/2}(m_{T(x)}-m)\rVert_{L^\infty_{\alpha,0}(\partial\Omega^-_2)}\lVert f_k\rVert_{L^\infty_{\alpha,0}(\partial\Omega^+_2)}\\
				+&\lVert \frac{1}{\sqrt{2\pi}}m^{-1/2}(m_{T(x)}-m)\rVert_{L^\infty_{\alpha,0}(\partial\Omega^-_2)}\\
				+&\max\{1,\kappa^{-1}\}\lVert \Gamma(f_k,f_k)\rVert_{L^\infty_{a,-1}(\Omega\times\mathbb{R}^3)}
			\end{pmatrix*}.
		\end{align} 
		Let $\mathcal{A}=\lVert f_{in}\rVert_{L^\infty_{\alpha,0}(\partial\Omega^-_1)}$ and $\mathcal{B}=\lVert \frac{1}{\sqrt{2\pi}}m^{-1/2}(m_{T(x)}-m)\rVert_{L^\infty_{\alpha,0}(\partial\Omega^-_2)}$. It is well-known that 
		\begin{align}
			\lVert\Gamma(f,f)\rVert_{L^\infty_{\alpha,-1}(\Omega\times\mathbb{R}^3)}\leq C \lVert f\rVert^2_{L^\infty_{a,0}(\Omega\times\mathbb{R}^3)}. 
		\end{align}
		and if $0<\alpha<1/4$ and $T(x)<2$, then 
		\begin{align}
			\lVert \frac{1}{\sqrt{2\pi}}m^{-1/2}(m_{T(x)}-m)\rVert_{L^\infty_{\alpha,0}(\partial\Omega^-_2)}\leq C_{\alpha}\lVert T(x)-1\rVert_{L^\infty(\partial\Omega^-_2)}.
		\end{align}
		We show that there exists $\delta$ such that, if $\mathcal{A}+\mathcal{B}\leq \delta$, then 
		\begin{align}
			\lVert f_{k+1}\rVert_{L^\infty_{\alpha,0}(\Omega\times\mathbb{R}^3)}+\lVert f_{k+1}\rVert_{L^\infty_{\alpha,0}(\partial\Omega^+_2)}\leq 2C_\zeta(\mathcal{A}+\mathcal{B})
		\end{align}
		for all $k\in\mathbb{N}$. The constant $C_\zeta$ is defined in \textbf{Theorem \ref{theorem 4.1}}.\par
		For $k=1$,  
		\begin{align*}
			\lVert f_{k+1}\rVert_{L^\infty_{\alpha,0}(\Omega\times\mathbb{R}^3)}+\lVert f_{k+1}\rVert_{L^\infty_{\alpha,0}(\partial\Omega^+_2)}\leq C_{\zeta}(\mathcal{A}+\mathcal{B})
		\end{align*}
		by \eqref{eq 4.47}. Suppose the statement holds for $k=n$ . For $k=n+1$, 
		\begin{align}\label{eqqq 4.50}
			\lVert f_{k+1}\rVert_{L^\infty_{\alpha,0}(\Omega\times\mathbb{R}^3)}+\lVert f_{k+1}\rVert_{L^\infty_{\alpha,0}(\partial\Omega^+_2)}
			&\leq C_{\zeta}\Big[(\mathcal{A}+\mathcal{B})
			+C\mathcal{B}(\mathcal{A}+\mathcal{B})+4CC_\zeta\max\{1,\kappa^{-1}\}(\mathcal{A}+\mathcal{B})^2\Big]\\\nonumber
			&\leq C_\zeta\Big[1+C\mathcal{B}+4CC_\zeta\max\{1,\kappa^{-1}\}(\mathcal{A}+\mathcal{B})\Big](\mathcal{A}+\mathcal{B}).
		\end{align}
		We choose 
		$$\delta<
		\min\left\{\frac{1}{4C},\frac{1}{16CC_\zeta\max\{1,\kappa^{-1}\}}\right\}.
		$$ 
		Then, the proof is complete. Now, consider $f_{k+1}-f_k$ for $k\in\mathbb{N}$. Following a similar procedure, we obtain
		\begin{align*}
			\lVert f_{k+1}-f_k\rVert_{L^\infty_{\alpha,0}(\Omega\times\mathbb{R}^3)}+\lVert f_{k+1}-f_k\rVert_{L^\infty_{\alpha,0}(\partial\Omega^+_2)}\leq \frac{1}{2}\Big(\lVert f_k-f_{k-1}\rVert_{L^\infty_{\alpha,0}(\Omega\times\mathbb{R}^3)}+\lVert f_k-f_{k-1}\rVert_{L^\infty_{\alpha,0}(\partial\Omega^+_2)}\Big).
		\end{align*}
		This implies that $\{f_k\}$ is a Cauchy sequence in $L^\infty_{\alpha,0}(\Omega\times\mathbb{R}^3)$. Hence, we obtain the solution by taking the limit $f=\lim_{k\to\infty}f_k$.  Uniqueness follows by same argument. Further the upper bound of $f$, by \eqref{eqqq 4.50}, we have
		\begin{align}
			\lVert f\rVert_{L^\infty_{\alpha,0}(\Omega\times\mathbb{R}^3)}+\lVert f\rVert_{L^\infty_{\alpha,0}(\partial\Omega^+_2)}
			\leq C(\mathcal{A}+\mathcal{B}).
		\end{align}
	\end{proof}

	\section{Application to Thermal transpiration and thermomolecular pressure difference}\label{section 5}
	\subsection{the model of the transpiration : the proof of \textbf{Theorem \ref{theorem 1.2}} and \textbf{Theorem \ref{corollary 1.1}}}
	\label{subsection 5.1}
	In this section, we prove the \textbf{Theorem} \ref{theorem 1.2}. Without loss of generality, we assume that $a=0$ and $b=L$. To simplify the calculation and to obtain a more explicit estimate, we prove \textbf{Theorem \ref{theorem 1.2}} when $\Omega$ satisfies
	\begin{align*}
		\Omega\cap\{0\leq x_3\leq L\}&=\{(x_1,x_2,x_3)\in\mathbb{R}^3 |\, x_1^2+x_2^2\leq 1, 0\leq x_3\leq L\}.
	\end{align*} 
	The boundary condition is imposed as follows.
	\begin{align}
		\begin{aligned}
			F(x,v)&=\frac{1}{(2\pi)^{3/2}}e^{-\frac{v^2}{2}},&&\, \text{on}\quad \partial \Omega^-_{1,a}=\partial\Omega^-\cap\{x_3<0\},\\
			F(x,v)&=\frac{1}{(2\pi)^{3/2}T^{5/2}_2}e^{-\frac{v^2}{2T_2}},&&\, \text{on}\quad \partial \Omega^-_{1,b}=\partial\Omega^-\cap\{x_3>L\},\\
			F(x,v)&=m_{T(x)}\int_{v\cdot n(x)>0}F(x,v)v\cdot n(x)dv,&&\, \text{on}\quad \partial \Omega^-_2=\partial\Omega^-\setminus\left(\partial \Omega^-_{1,a}\cup  \partial \Omega^-_{1,b}\right).
		\end{aligned}
	\end{align}
	Without loss of generality,  we assume that $T_2>1$. 
	We rewrite
	\begin{align}\label{eq 5.2}
		F(x,v)=m_{T_q}(x,v)+f(x,v)=\frac{1}{(2\pi)^{3/2}T^2(q(x,v))}e^{-\frac{|v|^2}{2T(q(x,v))}}+f(x,v).
	\end{align}
	Recall that $m_{T_q}(x,v)$ is a solution of 
	\begin{align}
		\left\{
		\begin{aligned}
			v\cdot\nabla_x F&=0,&&\text{in}\quad \Omega\times\mathbb{R}^3, \\
			F(x,v)&=\frac{1}{(2\pi)^{3/2}}e^{-\frac{|v|^2}{2}},&& \text{on}\quad \partial \Omega^-_{1,a}=\partial\Omega^-_1\cap\{x_3<a\},\\
			F(x,v)&=\frac{1}{(2\pi)^{3/2}T^2_2}e^{-\frac{|v|^2}{2T_2}},&&\text{on}\quad \partial \Omega^-_{1,b}=\partial\Omega^-_1\cap\{x_3>b\},\\
			F(x,v)&=m_T(x,v)\int_{v\cdot n(x)>0}F(x,v)v\cdot n(x)dv,&&\text{on}\quad \partial \Omega^-_2=\partial \Omega^-\setminus\left(\partial \Omega^-_{1,a}\cup \partial \Omega^-_{1,b}\right).
		\end{aligned}
		\right.
	\end{align}
	The proof appears in the \cite{Sone Y}. Now, the equation \eqref{eq 1.1} becomes
	\begin{align}
		\left\{
		\begin{aligned}\label{eq 5.5}
			v\cdot\nabla_x f&=\frac{2}{\kappa}Q(m_{T_q},f)+\frac{1}{\kappa}Q(f,f)+\frac{1}{\kappa}Q(m_{T_q},m_{T_q}),&&\quad\text{in}\quad \Omega\times\mathbb{R}^3, \\
			f&=0, &&\quad\text{on}\quad\partial \Omega^-_{1,a},\\
			f&=\frac{1}{\sqrt{2\pi}}\left(\frac{1}{\sqrt{T_2}}-1\right)\frac{1}{2\pi T_2^2}e^{-\frac{|v|^2}{2T_2}}\coloneq -C_{T_2}m_{T_2}, &&\quad\text{on} \quad\partial \Omega^-_{1,b},\\
			f&=\frac{1}{(2\pi)T^2(x)}e^{-\frac{|v|^2}{2T(x)}}\int_{v\cdot n(x)>0}f(x,v)v\cdot n(x)dv,&&\quad\text{on}\quad\partial \Omega^-_2.
		\end{aligned}
		\right.
	\end{align}
	Note that $C_{T_2}>0$. Recall that the local flux $u(x)$ and the total flux $U(z)$ in the $x_3$-direction are defined as
	\begin{align}
		u(x)\coloneq\int_{\mathbb{R}^3}v_3F(x,v)dv,\quad U(z)\coloneq \int_{\{x_1^2+x_2^2\leq 1, x_3=z\}}u(x) dS.
	\end{align}
	Now, we establish a lemma concerning $\Omega$ and $\partial \Omega_{1,a} \cup \partial \Omega_{1,b}$.
	\begin{lemma}\label{lemma 5.1}
		$\Omega$ is completely unsealed with respect to $\partial \Omega_{1,a} \cup \partial \Omega_{1,b}$.
	\end{lemma}
	\begin{proof}
		For any fixed $x\in\overline{\Omega\cap\{0\leq x_3\leq L\}}$, we classify $v\in\mathbb{R}^3$ as follows.
		\begin{align}
			A_{1}&=\left\{(v_1,v_2,v_3)\in\mathbb{R}^3\, \Big | \,(q(x,v),v)\in \partial\Omega_{1,a}^-\right\},\\\nonumber
			A_{2}&=\left\{(v_1,v_2,v_3)\in\mathbb{R}^3\, \Big| \, (q(x,v),v)\in \partial\Omega_{1,b}^-\right\},\\\nonumber
			A_{3}&=\left\{(v_1,v_2,v_3)\in\mathbb{R}^3\, \Big| \, (q(x,v),v)\in \partial\Omega_{2}^-\right\},
		\end{align}
		Let $v=(v_1,v_2,v_3)=(\rho\cos\phi\sin\theta,\rho\sin\phi\sin\theta,\rho\cos\theta)$ and $\hat{v}=(v_1,v_2),\, \hat{x}=(x_1,x_2)$. Then the sets $A_i$ can be written more explicitly as follow.
		\begin{align}\label{eq 5.7 5.8 5.9}
			A_{1}&=\left\{(v_1,v_2,v_3)\in\mathbb{R}^3\, \Big | \, \frac{x_3}{\sqrt{x_3^2+r^2}}<\cos\theta\leq 1\right\},\\\nonumber
			A_{2}&=\left\{(v_1,v_2,v_3)\in\mathbb{R}^3\, \Big| \, -1\leq\cos\theta< \frac{x_3-L}{\sqrt{(x_3-L)^2+r^2}}\right\},\\\nonumber
			A_{3}&=\left\{(v_1,v_2,v_3)\in\mathbb{R}^3\, \Big| \, \frac{x_3-L}{\sqrt{(x_3-L)^2+r^2}}\leq \cos\theta\leq \frac{x_3}{\sqrt{x_3^2+r^2}}\right\},
		\end{align}
		where 
		\begin{align*}
			r\coloneq h+\sqrt{h^2+k^2},\quad
			h\coloneq \hat{x}\cdot\frac{\hat{v}}{|\hat{v}|},\quad
			k\coloneq \sqrt{1-|\hat{x}|^2}.
		\end{align*}
		Recalling \textbf{Remark} \ref{remark 4.1}, we now consider the case $x\in \partial\Omega_{2}^-$. Without loss of generality, we assume that $(x_1,x_2)=(1,0)$. Then,
		\begin{align*}
			\int_{v\cdot n(x)>0}\textbf{1}_{\partial \Omega_2}\circ q(x,v)m_T(x,v)v\cdot n(x)dv&=2\int_{-\pi/2}^{\pi/2}\int_{\cos^{-1}\left(\frac{x_3}{\sqrt{x_3^2+4\sin^2\phi}}\right)}^{\cos^{-1}\left(\frac{x_3-L}{\sqrt{(x_3-L)^2+4\sin^2\phi}}\right)}\cos\phi\sin^2\theta d\theta d\phi\\
			&\coloneq \Theta(x_3).
		\end{align*}
		Observe that $0\leq \Theta(x_3)\leq 1$ and $\Theta(x_3)$ is a continuous  on $[0,L]$. Therefore, there exists $z_0\in[0,L]$ such that $\di\max_{[0,L]}\Theta(x_3)=\Theta(z_0)$. However, $\Theta(z_0)$ is equal to $1$ if and only if
		\begin{align}
			\cos^{-1}\left(\frac{z_0-L}{\sqrt{(z_0-L)^2+4\sin^2\phi}}\right)=\pi\quad \text{and}\quad \cos^{-1}\left(\frac{z_0}{\sqrt{z_0^2+4\sin^2\phi}}\right)=0
		\end{align}  
		for all $\phi$. This is impossible. Hence, 
		\begin{align*}
			\int_{v\cdot n(x)>0}\textbf{1}_{\partial \Omega_2}\circ q(x,v)m_T(x,v)v\cdot n(x)dv&\leq\Theta(z_0)<1\quad \text{for all}\, x\in\partial \Omega_2.
		\end{align*}	
	\end{proof}
	Note that if $T_2$ is sufficiently close to $1$ and $\Omega$ has $C^2$ boundary, then \textbf{Theorem \ref{theorem 1.1}} guarantees the existence of $F$ in $L^\infty_{0,\beta}(\Omega\times\mathbb{R}^3)$.
	\begin{proof}
		\textbf{The proof of theorem \ref{theorem 1.2} for the pipe case.}\\
		Note that the $f$ in \ref{eq 5.5} has the form
		\begin{align}\label{transpiration of solution}
			f(x,v)
			&=-\textbf{1}_{\partial \Omega^-_{1,b}}(q(x,v))C_{T_2}m_{T_2}\\\nonumber
			&\quad+\textbf{1}_{\partial \Omega^-_2}( q(x,v))m_T(q(x,v),v)\int_{v\cdot n(x)>0}f(q(x,v),v)v\cdot n(x)dv\\\nonumber
			&\quad+\frac{1}{\kappa}\left(\int_{0}^{\tau_{-}(x,v)}\left[2Q(m_{T_q},f)+Q(f,f)+Q(m_{T_q},m_{T_q})\right](x-sv,v)ds\right)\\\nonumber
			&\coloneq\textbf{1}_{\partial \Omega^-_{1,b}}g(x,v)+  J^T_{\partial \Omega^-_2}\overline{P}(f)(x,v)+\frac{1}{\kappa}h[f](x,v).
		\end{align}
		Observe that
		\begin{align*}
			\int_{\mathbb{R}^3}v_3m_{T_q}dv=\int_{0}^{\infty}\int_{0}^{2\pi}\int_{0}^{\pi}\frac{e^{-\frac{\rho^2}{2T(\theta,\phi)}}}{2\pi T^2(\theta,\phi)}\rho^3\cos\theta\sin\theta d\theta d\phi d\rho=0.
		\end{align*}
		We now use a procedure similar to the one in \textbf{section \ref{section 4}} . When tracing back $n+1$ time for $n\geq 1$, we obtain
		\begin{align}\label{eq 5.12}
			\int_{\mathbb{R}^3}v_3F(x,v)dv
			&=\int_{\mathbb{R}^3}v_3f(x,v)dv\\\nonumber
			&=\int_{\mathbb{R}^3}v_3\textbf{1}_{\partial \Omega^-_{1,b}}gdv+\frac{1}{\kappa}\int_{\mathbb{R}^3}v_3h[f]dv+\int_{\mathbb{R}^3}v_3  J^T_{\partial \Omega^-_2}\overline{P}(f)dv\\\nonumber
			&=\int_{\mathbb{R}^3}v_3\textbf{1}_{\partial \Omega^-_{1,b}}g dv+\int_{\mathbb{R}^3}v_3  J^T_{\partial \Omega^-_2}\overline{P}(\textbf{1}_{\partial \Omega^-_{1,b}}g)dv\\\nonumber
			&\quad+\sum_{j=1}^{n}\int_{\mathbb{R}^3}v_3  J^T_{\partial \Omega^-_2}\hat{P}^j\circ\overline{P}(\textbf{1}_{\partial \Omega^-_{1,b}}g)+\frac{1}{\kappa}\int_{\mathbb{R}^3}v_3h[f]dv\\\nonumber
			&\quad+\frac{1}{\kappa}\int_{\mathbb{R}^3} J^T_{\partial \Omega^-_2}\overline{P}(h[f])dv+\frac{1}{\kappa}\sum_{j=1}^{n}\int_{\mathbb{R}^3} J^T_{\partial \Omega^-_2}\hat{P}^j\circ\overline{P}(h[f])dv\\\nonumber
			&\quad+\int_{\mathbb{R}^3}v_3 J^T_{\partial \Omega^-_2}\hat{P}^{n+1}\circ\overline{P}(f)dv\\\nonumber
			&\coloneq A_1+A_2+\sum_{j=1}^{n}B_j+\frac{1}{\kappa}\Big(C_1+C_2\Big)+D.
		\end{align} 
		The definitions of $\overline{P}(f)$, $\hat{P}^j(f)$ and $ J^T_{\partial A}$ are as follows.
		\begin{align*}
			J^T_{A}(f)(x,v)
			&\coloneq \textbf{1}_{A}( q(x,v))m_T(q(x,v),v)f(q(x,v)),\\
			\overline{P}(f)(x,v)
			&\coloneq\Big(\int_{v_0\cdot n(x)>0}f(x,v_0)v_0\cdot n(x)dv_0\Big)(x,v),\\
			\hat{P}^j(f)(x,v)
			&\coloneq\Big(\int\cdots\int fd\sigma^T_{j-1,j}d\sigma^T_{j-2,j-1}\cdots d\sigma^T_{0,1}\Big)(x,v),\\
			\Big(\displaystyle\int fd\sigma^{T}_{j-1,j}\Big)(x_{j-1}, v_{j-1})
			&\coloneqq\displaystyle\int_{v_{j}\cdot n(x_{j-1})>0} J^T_{\partial \Omega^-_2}(f)(x_{j-1},v_j)v_{j}\cdot n(x_{j-1})dv_{j}.
		\end{align*}
		Here $x_0=q(x,v)$, $v_0=v$, and $x_{i}=q(x_{i-1},v_i)$  for all $i\in\mathbb{N}$.\par 
		Now, we estimate each term in \eqref{eq 5.12}.
		The definition of $A_1$ to $D$ appearing in the last line of \eqref{eq 5.12} are:
		\begin{align}
			A_1&=\int_{\mathbb{R}^3}v_3\textbf{1}_{\partial \Omega^-_{1,b}}gdv,\\
			A_2&=\int_{\mathbb{R}^3}v_3  J^T_{\partial \Omega^-_2}\overline{P}(\textbf{1}_{\partial \Omega^-_{1,b}}g)dv,\\
			B_j&=\int_{\mathbb{R}^3}v_3  J^T_{\partial \Omega^-_2}\hat{P}^j\circ\overline{P}(\textbf{1}_{\partial \Omega^-_{1,b}}g),\\
			C_1&=\int_{\mathbb{R}^3}v_3h[f]dv,\\
			C_2&=\int_{\mathbb{R}^3} v_3J^T_{\partial \Omega^-_2}\overline{P}(h[f])dv+\sum_{j=1}^{n}\int_{\mathbb{R}^3} v_3J^T_{\partial \Omega^-_2}\hat{P}^j\circ\overline{P}(h[f])dv,\\
			D&=\int_{\mathbb{R}^3}v_3 J^T_{\partial \Omega^-_2}\hat{P}^{n+1}\circ\overline{P}(f)dv.
		\end{align} 
		For $A_1$, using \eqref{eq 5.7 5.8 5.9}. we obtain 
		\begin{align*}
			A_1&=-\int_{0}^{\infty}\int_{0}^{2\pi}\int_{\cos^{-1}\left(\frac{x_3-L}{\sqrt{(x_3-L)^2+r^2}}\right)}^{\pi}C_{T_2}\frac{e^{-\frac{\rho^2}{2T_2}}}{2\pi T_2}\rho^3\cos\theta\sin\theta d\theta d\phi d\rho\\
			&=C\left(\int_{0}^{2\pi}\frac{r^2}{(x_3-L)^2+r^2}d\phi\right)\left(1-\frac{1}{\sqrt{T_2}}\right)\\
			&>0.
		\end{align*}
		For $A_2$, we observe that
		\begin{align*}
			\overline{P}(\textbf{1}_{\partial \Omega^-_{1,b}}g)(x)&=\int_{v_0\cdot n(x)>0}\textbf{1}_{\partial \Omega^-_{1,b}}gv_0\cdot n(x)dv_0,\\
			&=2C\int_{-\pi/2}^{\pi/2}\int^{\pi}_{\cos^{-1}\left(\frac{x_3-L}{\sqrt{(x_3-L)^2+4\sin^2\phi}}\right)}\cos\phi\sin^2\theta d\theta d\phi\\
			&\coloneq-\Theta(x_3)\left(1-\frac{1}{\sqrt{T_2}}\right).
		\end{align*}
		Let $v=(\rho\cos\phi\sin\theta,\rho\sin\phi\sin\theta,\rho\cos\theta)$. Then, we have
		\begin{align*}
			A_2&=\frac{1}{T_2^2}\left(1-\frac{1}{\sqrt{T_2}}\right)\int_{0}^{\infty}\int_{0}^{2\pi}\int_{\cos^{-1}\left(\frac{x_3-L}{\sqrt{(x_3-L)^2+r^2}}\right)}^{\cos^{-1}\left(\frac{x_3}{\sqrt{x_3^2+r^2}}\right)}m_T\Theta(q(x,v))\rho^3\cos\theta\sin\theta d\theta d\phi d\rho\\
			&\coloneq \overline{\Theta}_1(x)\left(1-\frac{1}{\sqrt{T_2}}\right).
		\end{align*}
		Similarly, there exists $\overline{\Theta}_{2,n}(x)$ such that
		\begin{align*}
			\sum_{j=1}^{n}B_j=\overline{\Theta}_{2,n}(x)\left(1-\frac{1}{\sqrt{T_2}}\right).
		\end{align*}
		The relation between the explicit forms of $\Theta(x)$, $\overline{\Theta}_1(x)$ and $\overline{\Theta}_2(x)$ is given by 
		\begin{align}
			\Theta(x)&=2\int_{-\pi/2}^{\pi/2}\int^{\pi}_{\cos^{-1}\left(\frac{x_3-L}{\sqrt{(x_3-L)^2+4\sin^2\phi}}\right)}\cos\phi\sin^2\theta d\theta d\phi,\\
			\overline{\Theta}_1(x)&=\int_{\mathbb{R}^3} v_3J^T_{\partial \Omega^-_2}\Theta dv\\\nonumber
			&=\int_{0}^{\infty}\int_{0}^{2\pi}\int_{\cos^{-1}\left(\frac{x_3-L}{\sqrt{(x_3-L)^2+r^2}}\right)}^{\cos^{-1}\left(\frac{x_3}{\sqrt{x_3^2+r^2}}\right)}m_T\Theta(q(x,v))\rho^3\cos\theta\sin\theta d\theta d\phi d\rho,\\
			\overline{\Theta}_2(x)&=\sum_{j=1}^{n}\int_{\mathbb{R}^3}v_3  J^T_{\partial \Omega^-_2}\hat{P}^j(\Theta(x)).\\\nonumber
			&=\sum_{j=1}^{n}\int_{0}^{\infty}\int_{0}^{2\pi}\int_{\cos^{-1}\left(\frac{x_3-L}{\sqrt{(x_3-L)^2+r^2}}\right)}^{\cos^{-1}\left(\frac{x_3}{\sqrt{x_3^2+r^2}}\right)}m_T\hat{P}^j(\Theta)(q(x,v))\rho^3\cos\theta\sin\theta d\theta d\phi d\rho,
		\end{align}
		where $v=(\rho\cos\phi\sin\theta,\rho\sin\phi\sin\theta,\rho\cos\theta)$. Since $\Omega$ is completely unsealed with respect to $\partial \Omega_{1,a} \cup \partial \Omega_{1,b}$ and applying \textbf{Proposition \ref{proposition 2.9}} , we obtain
		\begin{align*}
			\frac{1}{\kappa}\Big(|C_1|+|C_2|\Big)\leq \frac{C_\zeta}{\kappa}\lVert f\rVert^2_{L^\infty_{0,\beta}(\Omega\times\mathbb{R}^3)},\quad |D|\leq C\zeta^{n+1}\lVert f\rVert_{L^\infty_{0,\beta}(\Omega\times\mathbb{R}^3)}.
		\end{align*}
		Observe that $\overline{\Theta}_1(x)$ and $\overline{\Theta}_{2,n}(x)$ are positive for any $x\in\{x_3=0\}$. After combining all estimates and taking $n$ and $\kappa$ sufficiently large, there exists a constant $C$ such that
		\begin{align}
			U(0)=\int_{\{x_1^2+x_2^2\leq 1, x_3=0\}}u(x) dS\geq C\left(1-\frac{1}{\sqrt{T_2}}\right).
		\end{align} 
		Then, we finish the proof by invoking the conservation law of mass.
	\end{proof}
	For the general case, the terms $A_1, A_2$ and $B_j$ are very difficult to write down in explicit form, since \eqref{eq 5.7 5.8 5.9} depends on $\partial\Omega \cap \{x_3=0\}$ and $\partial\Omega \cap \{x_3=L\}$. However, we can still prove that $\Omega$ is a completely unsealed set with respect to $\partial \Omega_{1,a} \cup \partial \Omega_{1,b}$, and the proof of the general case proceed in the same way as before.
	\begin{lemma}\label{lemma 5.2}
		$\Omega$ is completely unsealed set with respect to 
		$\partial \Omega_{1,a} \cup \partial \Omega_{1,b}$	
	\end{lemma}
	\begin{proof}
		Since $\Omega$ is an open set, there exist $\rho>0$ and $a_1\in \{x_3=0\}, a_2\in\{x_3=L\}$ such that 
		\begin{align}
			B(a_1,\rho)\cap \{x_3=0\}\subset\Omega\cap\{x_3=0\}\quad \text{and}\quad B(a_2,\rho)\cap \{x_3=L\}\subset\Omega\cap\{x_3=L\}.
		\end{align}
		Then we define
		\begin{align}
			\Pi\coloneq B(a_1,\rho)\cap \{x_3=0\}\cup B(a_2,\rho)\cap \{x_3=L\}.
		\end{align}
		We show that $\Omega\cap\{0\leq x_3 \leq L\}$ is completely unsealed with respect to $\Pi$. We can then deduce the lemma from this conclusion.
		For any fixed $x\in\Xi\coloneq\partial(\Omega\cap\{0< x_3 < L\})$, without loss of generality, we may assume that $a_1=(0,0,0)$,  $x=(x_1,x_2,x_3)$, and $a_2=(a_{2,1},a_{2,2},L)$. Then we classify $v\in\mathbb{R}^3$ as follows.
		\begin{align*}
			A_{1}&=\left\{(v_1,v_2,v_3)\in\mathbb{R}^3\, \Big | \,q(x,v)\in B(0,\rho)\cap \{x_3=0\}\right\},\\
			A_{2}&=\left\{(v_1,v_2,v_3)\in\mathbb{R}^3\, \Big| \, q(x,v)\in B(a_2,\rho)\cap \{x_3=L\}\right\},\\
			A_3&=\mathbb{R}^3\setminus(A_1\cup A_2).
		\end{align*}
		Let $v=(\hat{\rho}\sin\theta\cos\phi,\hat{\rho}\sin\theta\sin\phi,\hat{\rho}\cos\theta)$. Then the sets $A_1$ and $A_2$ can be written more explicitly as follow.
		\begin{align}\label{eq 5.25}
			A_1=
			&\left\{(\hat{\rho},\phi,\theta)\in \mathbb{R}_+\times[0,2\pi]\times[0,\frac{\pi}{2}]\,\Big|0\leq\alpha_1^2+\gamma_1,\quad \frac{\alpha_1-\sqrt{\alpha_1^2+\gamma_1}}{x_3}< \tan\theta < \frac{\alpha_1+\sqrt{\alpha_1^2+\gamma_1}}{x_3}\right\}\\
			A_2=
			&\left\{(\hat{\rho},\phi,\theta)\in \mathbb{R}_+\times[0,2\pi]\times[\frac{\pi}{2},\pi]\,\Big|0\leq\alpha_2^2+\gamma,\quad \frac{-\alpha_2-\sqrt{\alpha_2^2+\gamma_2}}{L-x_3}< \tan\theta < \frac{-\alpha_2+\sqrt{\alpha_2^2+\gamma_2}}{L-x_3}\right\}
		\end{align}
		where 
		\begin{align*}
			\begin{aligned}
				&\alpha_1=x_1\cos\phi+x_2\sin\phi,\,&&\gamma_1=\rho^2-\left(x_1^2+x_2^2\right)^2\\
				&\alpha_2=[x_1-a_{2,1}]\cos\phi+[x_2-a_{2,2}]\sin\phi,\,&&\gamma_2=\rho^2-[(x_1-a_{2,1})^2+(x_2-a_{2,2})^2].
			\end{aligned}
		\end{align*}
		Let
		\begin{align*}
			W_1(x)\coloneq=1-\int_{u\cdot n(x)>0}1_{A_1}(u)m_{T(q(x,u))}u\cdot n(x) du
		\end{align*}
		and 
		\begin{align*}
			W_2(x)\coloneq 1-\int_{u\cdot n(x)>0}1_{A_2}(u)m_{T(q(x,u))}u\cdot n(x) du.
		\end{align*}
		Since $\Omega\cap\{0\leq x_3 \leq L\}$ is a convex set and $\partial\Omega\cap\{0\leq x_3 \leq L\}$ is a $C^1$ boundary, $W_1$ is continuous on $\partial\Omega\cap\{\frac{L}{2}\leq x_3 \leq L\}$ and $W_2$ is continuous on $\partial\Omega\cap\{0\leq x_3 \leq\frac{L}{2}\}$. Moreover, $W_1$ and $W_2$ are continuous on $\overline{\Omega\cap\{x_3=L\}\setminus\Pi}$ and $\overline{\Omega\cap\{x_3=0\}\setminus\Pi}$ respectively. Hence, we have
		\begin{align*}
			&\sup_{x\in\Xi\setminus\Pi}\int_{u\cdot n(x)>0}1_{\Xi\setminus\Pi}(q(x,u))m_{T(q(x,u))}u\cdot n(x) du\\
			&\leq\max\left\{\max_{x\in\partial\Omega\cap\{\frac{L}{2}\leq x_3 \leq L\}}W_1(x),\max_{x\in\overline{\Omega\cap\{x_3=L\}\setminus\Pi}}W_1(x),\max_{x\in\partial\Omega\cap\{\frac{L}{2}\leq x_3 \leq L\}}W_2(x),\max_{x\in\overline{\Omega\cap\{x_3=0\}\setminus\Pi}}W_2(x)\right\}\\
			&\coloneq \zeta.
		\end{align*}
		Similar to the argument of \text{lemma} \ref{lemma 5.1}, we deduce that $\zeta$ is less than one. Thus, we conclude that $\Omega$ is completely unsealed set with respect to 
		$\partial \Omega_{1,a} \cup \partial \Omega_{1,b}$.
	\end{proof}
	The proof of \textbf{Theorem \ref{corollary 1.1}} is similar to the proof of \textbf{Theorem \ref{theorem 1.2}} because, after rewriting \eqref{eq 1.4} by equation \eqref{eq 5.2}, \eqref{eq 1.4} becomes
	\begin{align}
		\left\{
		\begin{aligned}\label{eq 5.6}
			v\cdot\nabla_x f&=\frac{2}{\kappa}Q(m_{T_q},f)+\frac{1}{\kappa}Q(f,f)+\frac{1}{\kappa}Q(m_{T_q},m_{T_q}),&&\quad\text{in}\quad \Omega\times\mathbb{R}^3, \\
			f&=0, &&\quad\text{on}\quad \partial \Omega^-_{1,a},\\
			f&=0, &&\quad\text{on} \quad\partial \Omega^-_{1,b},\\
			f&=\frac{1}{2\pi T^2(x)}e^{-\frac{v^2}{2T(x)}}\int_{v\cdot n(x)>0}f(x,v)v\cdot n(x)dv,&&\quad\text{on}\quad\partial \Omega^-_2.
		\end{aligned}
		\right.
	\end{align}
	Then, equation \eqref{eq 5.12} becomes
	\begin{align}
		\int_{\mathbb{R}^3}v_3F(x,v)dv
		&=\int_{\mathbb{R}^3}v_3f(x,v)dv\\\nonumber
		&=\frac{1}{\kappa}\int_{\mathbb{R}^3}v_3hdv+\int_{\mathbb{R}^3}v_3  J^T_{\partial \Omega^-_2}\overline{P}(f)dv\\\nonumber
		&=\frac{1}{\kappa}\int_{\mathbb{R}^3}v_3h[f]dv+\frac{1}{\kappa}\int_{\mathbb{R}^3} J^T_{\partial \Omega^-_2}\overline{P}(h[f])dv+\frac{1}{\kappa}\sum_{j=1}^{n}\int_{\mathbb{R}^3} J^T_{\partial \Omega^-_2}\hat{P}^j\circ\overline{P}(h[f])dv\\\nonumber
		&\quad+\int_{\mathbb{R}^3}v_3 J^T_{\partial \Omega^-_2}\hat{P}^{n+1}\circ\overline{P}(f)dv\\\nonumber
		&\coloneq\frac{1}{\kappa}\Big(C_1+C_2\Big)+D.
	\end{align}
	Since the upper bound of $|C_1|+|C_2|$ can be estimated independently of $n$ and $D$ can be made arbitrary small when $n$ is sufficiently large, the proof is complete.

	\section{The $H^2$ estimate of Laplace equation with Robin boundary condition}\label{section 6}
	In this section, we prove the inequality \eqref{H2 estimate}.
	That is, we discuss about the following elliptic equation
	\begin{align}\label{Robin boundary condition}
		\left\{
		\begin{aligned}
			-\Delta U &= F, && \text{in }\Omega, \\
			\chi(x)U+\partial_\nu U&= 0, && \text{on }\partial\Omega.
		\end{aligned}
		\right.
	\end{align}
	Here $\chi(x)$ is a non-negative $C^2$ function and $\chi(x)=0$ on $\partial\Omega_{2}$. We require that $\chi> 0$ on a subset of $\partial\Omega$ of positive measure. Before we start our argument, we recall several classical trace theorems and elliptic estimate. Here $\gamma$ is the trace operator.
	\begin{theorem}
		Let $u,v \in W^1(\Omega)$ and $\Delta u, \Delta v\in L^2(\Omega)$ . Then:
		\begin{enumerate}[(i)]
			\item
			\[
			\int_\Omega \Delta u \, dx
			=
			\int_{\partial \Omega} \gamma\partial_\nu u\, dS.
			\]
			
			\item
			\[
			\int_\Omega \nabla v \cdot \nabla u \, dx
			=
			-\int_\Omega u \Delta v \, dx
			+
			\int_{\partial \Omega} u\, \gamma\partial_\nu v\, dS.
			\]
			
			\item
			\[
			\int_\Omega \bigl(u\,\Delta v - v\,\Delta u\bigr)\, dx
			=
			\int_{\partial \Omega}
			\left(
			u\,\gamma\partial_\nu v
			-
			v\,\gamma\partial_\nu u
			\right)
			\, dS.
			\]
		\end{enumerate}
	\end{theorem}
	\begin{theorem}[\cite{Grisvard Pierre.} Theorem 1.5.1.3]\label{thm:trace 1}
		Let $\Omega$ be a bounded open subset of $\mathbb{R}^n$ with a Lipschitz boundary $\Gamma$.
		Then the mapping
		\[
		u \mapsto \gamma u
		\]
		which is defined for $u \in C^{0,1}(\overline{\Omega})$, has a unique continuous extension
		as an operator from $W^{1}_{p}(\Omega)$ onto $W^{1-\frac{1}{p}}_{p}(\Gamma)$.
		This operator has a right continuous inverse independent of $p$.
	\end{theorem}
	The following theorem is a general version of the above theorem
	\begin{theorem}[\cite{Grisvard Pierre.} Theorem 1.5.1.2]\label{thm:trace 2}
		Let $\Omega$ be a bounded open subset of $\mathbb{R}^n$ with boundary
		$\Gamma=\partial\Omega$ of class $C^{k,1}$.
		Assume that
		\[
		s-\frac1p \notin \mathbb{N}, \qquad s \le k+1, \qquad
		s-\frac1p = l+\sigma,
		\]
		where $l\in\mathbb{N}_0$ and $0<\sigma<1$.
		Then the mapping
		\[
		u \longmapsto
		\bigl(
		\gamma u,\,
		\gamma \partial_\nu u,\,
		\ldots,\,
		\gamma \partial_\nu^{\,l} u
		\bigr),
		\]
		initially defined for $u\in C^{k,1}(\overline{\Omega})$, admits a unique
		continuous extension as an operator
		\[
		W_p^{\,s}(\Omega)
		\;\longrightarrow\;
		\prod_{j=0}^{l} W_p^{\,s-j-\frac1p}(\Gamma),
		\]
		which is onto.
		Moreover, this operator possesses a continuous right inverse which does
		not depend on $p$.
	\end{theorem}
	
	\begin{theorem}[\cite{Grisvard Pierre.} Boundary $H^2$-regularity]\label{Neumann boundary condition}
		Assume
		\begin{equation*}
			f \in L^{2}(\Omega).
		\end{equation*}
		Suppose that $u \in H^{1}(U)$ is a weak solution of the elliptic boundary-value problem
		\begin{equation*}
			\begin{cases}
				\Delta u = f & \text{in } \Omega, \\
				\gamma \partial_\nu u = 0 & \text{on } \partial \Omega .
			\end{cases}
		\end{equation*}
		Assume finally
		\begin{equation*}
			\partial \Omega \text{ is } C^{2} \text{ and } \int_{\Omega}f =0.
		\end{equation*}
		Then
		\[
		u \in H^{2}(\Omega),
		\]
		and we have the estimate
		\begin{equation*}
			\|u\|_{H^{2}(\Omega)} \leq C \|f\|_{L^{2}(\Omega)},
		\end{equation*}
		the constant $C$ depending only on $\Omega$.
	\end{theorem}
	Now, we prove the following estimate
	\begin{lemma}
		Let $\Omega$ be a bounded open subset of $\mathbb{R}^n$ with $C^2$ boundary. Assume that $F\in L^2(\Omega)$. Suppose that $U \in H^{1}(\Omega)$ is a weak solution of \eqref{Robin boundary condition}. 
		Then
		\[
		U \in H^{2}(\Omega),
		\]
		and we have the estimate
		\begin{equation*}
			\|u\|_{H^{2}(\Omega)} \leq C \left(\|F\|_{L^{2}(\Omega)}+\|u\|_{H^{1}(\Omega)}\right),
		\end{equation*}
		the constant $C$ depending only on $\Omega$.
	\end{lemma}
	\begin{proof}
		By Green's identity, we have
		\begin{align}\label{compatibility condition}
			\int_{\Omega}F dx=\int_{\partial\Omega}\chi(x)U dS.
		\end{align}
		We consider the following elliptic problem
		\begin{align}
			\left\{
			\begin{aligned}
				-\Delta U &= F, && \text{in }\Omega, \\
				\partial_nU&=g\coloneq-\chi(x)U, && \text{on }\partial\Omega.
			\end{aligned}
			\right.
		\end{align}
		After applying theorem \ref{thm:trace 2}, there exists $G\in W^2(\Omega)$ such that $\partial_\nu G=-\chi(x)U$. Then, $U-G$ satisfies
		\begin{align}
			\left\{
			\begin{aligned}
				-\Delta (U-G) &= F+\Delta G, && \text{in }\Omega, \\
				\partial_n(U-G)&=0, && \text{on }\partial\Omega.
			\end{aligned}
			\right.
		\end{align}
		By \eqref{compatibility condition} and Green's identity, we have
		\begin{align*}
			\int_{\Omega}F+\Delta G dx=\int_{\Omega}F dx+\int_{\Omega}\Delta G dx=\int_{\Omega}F dx+\int_{\partial\Omega}\partial_\nu G dS=\int_{\Omega}F dx-\int_{\partial\Omega}\chi(x)U dS=0.
		\end{align*}
		Applying theorem \ref{Neumann boundary condition}, we yield
		\begin{equation*}
			\|U-G\|_{H^{2}(\Omega)} \leq C \|F+\Delta G\|_{L^{2}(\Omega)}.
		\end{equation*}
		Using theorem \ref{thm:trace 2} and theorem \ref{thm:trace 1}, we obtain
		\begin{align*}
			\|U\|_{H^{2}(\Omega)} 
			&
			\leq C\left( \|F+\Delta G\|_{L^{2}(\Omega)}+\|G\|_{H^{2}(\Omega)}\right)\\
			&
			\leq C\left( \|F\|_{L^{2}(\Omega)}+\|G\|_{H^{2}(\Omega)}\right)\\
			&
			\leq C\left( \|F\|_{L^{2}(\Omega)}+\|\chi(x)U\|_{H^{1/2}(\partial\Omega)}\right)\\
			&
			\leq C\left( \|F\|_{L^{2}(\Omega)}+\|U\|_{H^{1/2}(\partial\Omega)}\right)\\
			&
			\leq C\left( \|F\|_{L^{2}(\Omega)}+\|U\|_{H^{1}(\Omega)}\right)\\
		\end{align*}
		The proof is complete.
	\end{proof}
	\begin{lemma}
		Let $\Omega$ be a bounded connected open subset of $\mathbb{R}^n$ with $C^2$ boundary. Assume that $F\in L^2(\Omega)$. 
		Then, \eqref{Robin boundary condition} admits a unique weak solution $U$. Furthermore,
		\begin{equation}\label{estimate H^1}
			\|U\|_{H^1(\Omega)} \leq C \|F\|_{L^{2}(\Omega)},
		\end{equation}
		the constant $C$ depending only on $\Omega$.
	\end{lemma}
	\begin{proof}
		Note that the weak solution of \eqref{Robin boundary condition} means
		\begin{align}\label{weak solution}
			\int_\Omega \nabla U \cdot \nabla v \, dx
			+\int_{\partial \Omega} \chi Uv\, dS
			=\int_\Omega Fv \, dx.
		\end{align}
		Let 
		$$a(u,v)\coloneq\int_\Omega \nabla u \cdot \nabla v \, dx
		+\int_{\partial \Omega} \chi uv\, dS
		.$$
		Since $\chi\in L^\infty(\partial\Omega)$ and $u,v\in H^1(\Omega)$, we have
		$$
		|a(u,v)|\leq C\|u\|_{H^1(\Omega)}\|v\|_{H^1(\Omega)}
		.$$
		Assume that $a$ is not coercive, There exists a sequence in $H^1(\Omega)$ such that 
		\begin{align}\label{estimate of u_n}
			\|u_n\|^2_{H^1(\Omega)}>n\left(\int_\Omega |\nabla u_n|^2 \, dx
			+\int_{\partial \Omega} \chi u_n^2\, dS\right).	
		\end{align}
		Without loss of generality, we assume that $\|u_n\|_{L^2(\Omega)}=1$ for any $n\in\mathbb{N}$. By \eqref{estimate of u_n} and $\|u_n\|_{L^2(\Omega)}=1$, we discover that $\{u_n\}$ is a bounded sequence in $H^1(\Omega)$. Hence $\{u_n\}$ is weakly compact in $H^1(\Omega)$ and is compact in $L^2(\Omega)$. Let $\{u_{n_j}\}$ be a convergent subsequence of $\{u_n\}$ and $u$ be its limit. That is, 
		\begin{align}
			\begin{aligned}
				&u_{n_j}\rightharpoonup u &&\quad\text{ in } H^1(\Omega),\\
				&u_{n_j}\to u &&\quad\text{ in } L^2(\Omega).
			\end{aligned}
		\end{align}

		We observe that 
		$$
		0=\lim_{n_j\to\infty}\frac{1}{n_j}\|\nabla u\|_{L^2(\Omega)}\geq \lim_{n_j\to\infty}\|\nabla u_{n_j}\|_{L^2(\Omega)}\|\nabla u\|_{L^2(\Omega)}\geq \lim_{n_j\to\infty}\int_\Omega \nabla u_{n_j}\cdot\nabla u \, dx=\int_\Omega |\nabla u|^2 \, dx.
		$$
		It implies that $u_{n_j}$ converges strongly to $u$ in $H^1(\Omega)$; hence, 
		$$\lim_{n_j\to\infty}\int_{\partial \Omega} \chi u_n^2\, dS=\int_{\partial \Omega} \chi u^2\, dS=0.
		$$
		Thus, we deduce that $u=0$, which yields a contradiction since $\|u_n\|_{L^2(\Omega)}=1$. By Lax-Milgram theorem, the proof of the existence is complete. Assume that \eqref{estimate H^1} does not hold, there exists a sequence $\{U_n\}$ such that 
		\begin{equation}\label{111111}
			\|U_n\|_{H^1(\Omega)} \geq n\|F_n\|_{L^{2}(\Omega)}.
		\end{equation}
		Note that 
		\begin{align}\label{222222}
			\int_\Omega |\nabla U_n|^2 \, dx
			+\int_{\partial \Omega} \chi U_n^2\, dS
			=\int_\Omega F_nU_n \, dx \leq \|F_n\|_{L^2(\Omega)}\|U_n\|_{L^2(\Omega)}.
		\end{align}
		Without loss of generality, we assume that $\|U_n\|_{L^2(\Omega)}=1$ for any $n\in\mathbb{N}$. By \eqref{111111} and \eqref{222222}, we can immediately conclude that $\{U_n\}$ is bounded in $H^1(\Omega)$ and there exists a subsequence of $U_n$ which strongly converge to $0$ in $H^1(\Omega)$ by similar argument. It is a contradiction since $\|U_n\|_{L^2(\Omega)}=1$ for any $n\in\mathbb{N}$.
		The proof is complete.
	\end{proof}
	\section{The sufficient conditions for unsealed sets}
	\label{section 7}
	Let $\Omega$ be a subset in $\mathbb{R}^3$ and $\partial B$ be a subset of $\partial\Omega$. Determining whether $\Omega$ is unsealed with respect to $\partial B$ is a difficult problem. In this section, we provide some sufficient conditions concerning this problem.
	\begin{lemma}\label{lemma 5.3}
		Let $\Omega\in\mathbb{R}^3$ be a bounded convex domain with $C^1$ boundary and $\partial \Omega_{1,a}, \partial \Omega_{1,b}$ be subsets of $\partial\Omega$. If  $\partial \Omega_{1,a}\subseteq\partial \Omega_{1,b}$ and $\Omega$ is unsealed with respect to $\partial \Omega_{1,a}$, then $\Omega$ is  unsealed with respect to $\partial \Omega_{1,b}$. 
	\end{lemma}
	\begin{proof}
		Since $\partial \Omega_{1,a}\subseteq\partial \Omega_{1,b}$, $\big(\partial\Omega\setminus\partial \Omega_{1,b}\big)\subseteq\big(\partial\Omega\setminus\partial \Omega_{1,a}\big)$. Then, by \textbf{definition \ref{definition 4.1}}, the proof is complete.
	\end{proof}
	\begin{proposition}\label{proposition 5.1}
		Let $\Omega\subset\mathbb{R}^3$ be a bounded domain with $C^1$ boundary ,and let $\partial \Omega_1$ be an open subset of $\partial\Omega$. If there exist $\partial C_1$, $\partial C_2 \subset\partial\Omega\setminus\partial \Omega_1$ and $0<\zeta<1$ such that 
		\begin{enumerate}
			\item $\partial C_1\cup\partial C_2 =\partial\Omega\setminus\partial \Omega_1$.
			\item $\di\int_{v\cdot n(x)>0}\textbf{1}_{\partial\Omega\setminus\partial \Omega_1}(q(x,v))m(v)v\cdot n(x)dv\leq \zeta \quad\forall x\in\partial C_1$.
			\item $\di\int_{v\cdot n(x)>0}\textbf{1}_{\partial C_2}(q(x,v))m(v)v\cdot n(x)dv=0 \quad\forall x\in\partial C_2$.
		\end{enumerate} Then, $\Omega$ is unsealed with respect to $\partial \Omega_1$. 
	\end{proposition}
	\begin{proof}
		For convenience, we use same symbol in \textbf{section \ref{section 4}}
		\begin{align*}
			&\displaystyle\int \textbf{1}_{\partial\Omega\setminus\partial \Omega_1}(x_2)d\sigma_{2,1}\textbf{1}_{\partial\Omega\setminus\partial \Omega_1}(x_1)d\sigma_{0,1}\\
			&=\displaystyle\int \textbf{1}_{\partial\Omega\setminus\partial \Omega_1}(x_2)d\sigma_{2,1}\textbf{1}_{\partial C_1}(x_1)d\sigma_{0,1}+\displaystyle\int \textbf{1}_{\partial C_1}(x_2)d\sigma_{2,1}\textbf{1}_{\partial C_2}(x_1)d\sigma_{0,1}+\displaystyle\int \textbf{1}_{\partial C_2}(x_2)d\sigma_{2,1}\textbf{1}_{\partial C_2}(x_1)d\sigma_{0,1}\\
			&\leq\displaystyle\int \textbf{1}_{\partial\Omega\setminus\partial \Omega_1}(x_2)d\sigma_{2,1}\textbf{1}_{\partial C_1}(x_1)d\sigma_{0,1}+\displaystyle\int \textbf{1}_{\partial C_2}(x_1)d\sigma_{0,1}.
		\end{align*}
		If $x\in\partial C_2$, then 
		\begin{align*}
			\displaystyle\int \textbf{1}_{\partial\Omega\setminus\partial \Omega_1}(x_2)d\sigma_{2,1}\textbf{1}_{\partial C_1}(x_1)d\sigma_{0,1}\leq \zeta.
		\end{align*}
		Otherwise, 
		\begin{align*}
			\displaystyle\int \textbf{1}_{\partial\Omega\setminus\partial \Omega_1}(x_2)d\sigma_{2,1}\textbf{1}_{\partial\Omega\setminus\partial \Omega_1}(x_1)d\sigma_{0,1}&\leq\displaystyle\int \textbf{1}_{\partial\Omega\setminus\partial \Omega_1}(x_2)d\sigma_{2,1}\textbf{1}_{\partial C_1}(x_1)d\sigma_{0,1}+\displaystyle\int \textbf{1}_{\partial C_2}(x_1)d\sigma_{0,1}\\
			&\leq\int\displaystyle\textbf{1}_{\partial C_1}(x_1)d\sigma_{0,1}+\displaystyle\int \textbf{1}_{\partial C_2}(x_1)d\sigma_{0,1}\\
			&=\int_{v\cdot n(x,v)>0}\textbf{1}_{\partial\Omega\setminus\partial \Omega_1}(q(x,v))m(v)v\cdot n(x)dv\\
			&\leq \zeta.
		\end{align*}
		So
		\begin{align}
			\displaystyle\int \textbf{1}_{\partial\Omega\setminus\partial \Omega_1}(x_2)d\sigma_{2,1}\textbf{1}_{\partial\Omega\setminus\partial \Omega_1}(x_1)d\sigma_{0,1}\leq\zeta \quad\forall x\in\partial\Omega\setminus\partial \Omega_1.
		\end{align}
	\end{proof}

	\bibliographystyle{plain}

\begin{thebibliography}{99}
		
		\bibitem{A rarefied gas flow induced by a temperature field}
		Kazuo Aoki, Yoshio Sone, and Yorifumi Waniguchi. "A rarefied gas flow induced by a temperature field: Numerical analysis of the flow between two coaxial elliptic cylinders with different uniform temperatures." Computers \& Mathematics with Applications 35.1-2 (1998): 15-28.
		
		
		\bibitem{Bennett M J Tompkins F C 1957}
		M, J Bennett,, F. C Tompkins. (1957). Thermal transpiration: application of Liang's equation. Transactions of the Faraday Society, 53, 185-192.
		
		\bibitem{Bird G A 1994}
		Graeme Austin Bird. (1994). Molecular gas dynamics and the direct simulation of gas flows. Oxford university press.
		
		\bibitem{Caflisch ; Russel E.}
		Russel Caflisch. The Boltzmann equation with a soft potential. I. Linear, spatiallyhomogeneous. Comm. Math. Phys. 74 (1980), no. 1, 71–95.
		
		\bibitem{Rojas Cardenas Marcos}
		Rojas Cardenas, Marcos, et al. "Thermal transpiration flow: A circular cross-section microtube submitted to a temperature gradient." Physics of Fluids 23.3 (2011).
		
		\bibitem{Cercignani Carlo Reinhard Illner}
		Cercignani, Carlo, Reinhard Illner, and Mario Pulvirenti. The mathematical theory of dilute gases. Vol. 106. Springer Science Business Media, 2013.
		
		
		\bibitem{Chen C C Chen I K Liu T P Sone Y 2007}
		Chiun-Chuan Chen, I-Kun Chen, Tai-Ping Liu, Yoshio Sone. (2007). Thermal transpiration for the linearized Boltzmann equation. Communications on Pure and Applied Mathematics: A Journal Issued by the Courant Institute of Mathematical Sciences, 60(2), 147-163.
		
		\bibitem{Choulli 1999}
		Mourad Choulli, Plamen Stefanov. An inverse boundary value problem for the stationary transport equation. \emph{Osaka J. Math}. \text{36} (1999), no.1, 87-104.
		
		\bibitem{Chen I-Kun 2024}
		I-Kun Chen, Chun-Hsiung Hsia, Kawagoe, Daisuke, Jhe-Kuan Su. Geometric effects on $W^{1,p}$ regularity of the stationary linearized Boltzmann equation, Indiana University mathematical Journal 74(6) 1749-1793.
		
		\bibitem{Chen I K Liu T P Takata S 2014}
		I-Kun Chen, Tai-Ping Liu,  Shigeru Takata. (2014). Boundary singularity for thermal transpiration problem of the linearized Boltzmann equation. Archive for Rational Mechanics and Analysis, 212(2), 575-595.
		
		\bibitem{Chen I 2025}
		I-Kun Chen, Chun-Hsiung Hsia, Kawagoe, Daisuke. (2025). On the existence and regularity of weakly nonlinear stationary Boltzmann equations: a Fredholm alternative approach. preprint arXiv:2501.02419.
		
		\bibitem{Crookes W 1874}
		William Crookes. (1874). XV. On attraction and repulsion resulting from radiation. Philosophical transactions of the Royal society of London, (164), 501-527.
		
		
		\bibitem{Duan R 2017}
		Renjun Duan, Feimin Huang, Yong Wang, Tong Yang. Global Well-Posedness of the Boltzmann Equation with Large Amplitude Initial Data. Arch Rational Mech Anal 225, 375–424 (2017).
		
		\bibitem{Duan R 2019}
		Renjun Duan, Feimin Huang, Yong Wang,  Zhu Zhang. Effects of Soft Interaction and Non-isothermal Boundary Upon Long-Time Dynamics of Rarefied Gas. \emph{Arch Rational Mech Anal}. \textbf{234} (2019), 925–1006.
		
		\bibitem{Esposito R 2013}
		Raffaele Esposito, Yan Guo, Chanwoo Kim, Rossana Marra. Non-Isothermal Boundary in the Boltzmann Theory and Fourier Law. Commun. Math. Phys. 323, 177–239 (2013).
		
		\bibitem{Evans Lawrence C 2022}
		Lawrence Craig Evans. Partial differential equations. Vol. 19. American mathematical society, 2022.
		
		\bibitem{elliptic problem}
		Gazzola, Filippo, Hans-Christoph Grunau, and Guido Sweers. Polyharmonic boundary value problems: positivity preserving and nonlinear higher order elliptic equations in bounded domains. Springer, 2010.
		
		
		\bibitem{Glassey Robert T 1996}
		Robert Theodore Glassey. The Cauchy problem in kinetic theory. Society for Industrial and Applied Mathematics, 1996.
		
		
		\bibitem{Grisvard Pierre.}
		Pierre Grisvard. Elliptic problems in nonsmooth domains. Society for Industrial and Applied Mathematics, 2011
		
		\bibitem{Guo X Singh D Murthy J  Alexeenko A A 2009}
		Xiaohui Guo, Dhruv Singh, Jayathi Murthy, Alina A. Alexeenko1. (2009). Numerical simulation of gas-phonon coupling in thermal transpiration flows. Physical Review E—Statistical, Nonlinear, and Soft Matter Physics, 80(4), 046310.
		
		\bibitem{Guo Y 2003}
		Yan Guo. Classical Solutions to the Boltzmann Equation for Molecules with an Angular Cutoff. Arch. Rational Mech. Anal. 169, 305–353 (2003).
		
		
		
		\bibitem{Han Y L Phillip Muntz E Alexeenko A Young M 2007}
		Yen-Lin Han, Eric Phillip Muntz, Alina A. Alexeenko, Marcus Young. (2007). Experimental and computational studies of temperature gradient–driven molecular transport in gas flows through nano/microscale channels. Nanoscale and Microscale Thermophysical Engineering, 11(1-2), 151-175.
		
		\bibitem{Mouhot Clément 2007}
		Clément Mouhot and Robert Michael Strain. "Spectral gap and coercivity estimates for linearized Boltzmann collision operators without angular cutoff." Journal de mathématiques pures et appliquées 87.5 (2007): 515-535.
		
		\bibitem{Kandlikar S 2006}
		Dongqing Li, Michael R. King, Satish G. Kandlikar, Stéphane Colin, Srinivas Garimella. (2006). Heat transfer and fluid flow in minichannels and microchannels. elsevier.
		
		\bibitem{Kawagoe Daisuke 2025}
		Kawagoe, Daisuke. $ H^ s_x $ regularity of solutions to the stationary Boltzmann equation with the incoming boundary condition,
		\emph{} preprint, arXiv:2507.18211.
		
		\bibitem{Knudsen M 1909}
		Martin Knudsen. (1909). Eine revision der gleichgewichtsbedingung der gase. Thermische molekularströmung. Annalen der Physik, 336(1), 205-229.
		
		\bibitem{Knudsen M 1910}
		Martin Knudsen. (1910). Thermischer molekulardruck der gase in röhren. Annalen der Physik, 338(16), 1435-1448.
		
		\bibitem{Liang S C 1953}
		Shih Ching Liang. (1953). On the calculation of thermal transpiration. The Journal of Physical Chemistry, 57(9), 910-911.
		
		
		\bibitem{Loyalka S K 1971}
		Shankar K. Loyalka. (1971). Kinetic theory of thermal transpiration and mechanocaloric effect. I. The Journal of Chemical Physics, 55(9), 4497-4503.
		
		\bibitem{Maxwell J C 1879}
		James Clerk Maxwell. (1879). VII. On stresses in rarified gases arising from inequalities of temperature. Philosophical Transactions of the royal society of London, (170), 231-256.
		
		\bibitem{Niimi H 1971}
		Hideyuki Niimi. (1971). Thermal creep flow of rarefied gas between two parallel plates. Journal of the Physical Society of Japan, 30(2), 572-574.
		
		
		
		\bibitem{Ohwada T Sone Y Aoki K 1989}
		Taku Ohwada, Yoshio Sone,  Kazuo Aoki. (1989). Numerical analysis of the Poiseuille and thermal transpiration flows between two parallel plates on the basis of the Boltzmann equation for hard‐sphere molecules. Physics of Fluids A: Fluid Dynamics, 1(12), 2042-2049.
		
		\bibitem{Setting}
		Sarantis Pantazisa, Steryios Narisa, Christos Tantosa, Dimitris Valougeorgisa, Julien Andréb, Francois Milletb, Jean Paul Perin (2013). Nonlinear vacuum gas flow through a short tube due to pressure and temperature gradients. Fusion Engineering and Design, 88(9-10), 2384-2387.
		
		\bibitem{Passian A Warmack R J Ferrell T L  Thundat T 2003}
		Ali Passian, Robert J. Warmack, Thomas L. Ferrell, Thomas Thundat. (2003). Thermal transpiration at the microscale: a Crookes cantilever. Physical review letters, 90(12), 124503.
		
		\bibitem{Radtke G A Hadjiconstantinou N G  Wagner W 2011}
		Gregg A. Radtke, Nicolas G. Hadjiconstantinou, Wolfgang Wagner. (2011). Low-noise Monte Carlo simulation of the variable hard sphere gas. Physics of fluids, 23(3).
		
		\bibitem{Renardy}
		Michael Renardy, Robert C. Rogers. An introduction to partial differential equations. New York, NY: Springer New York, 2004.
		
		\bibitem{Reynolds O 1879}
		Osborne Reynolds. (1879). Xviii. on certain dimensional properties of matter in the gaseous state.-Part I. Experimental researches on thermal transpiration of gases through porous plates and on the laws of transpiration and impulsion, including an experimental proof that gas is not a continuous plenum.-part ii. on an extension of the dynamical theory of gas, which includes the stresses, tangential and normal, caused by a varying condition of gas, and affords an explanation of the phenomena of transpiration and impulsion. Philosophical Transactions of the Royal Society of London, (170), 727-845.
		
		
		\bibitem{Sone Y}
		Yoshio Sone, (Ed.). (2007). Molecular gas dynamics: theory, techniques, and applications. Boston, MA: Birkhäuser Boston.
		
		\bibitem{Strongrich A Pikus A Sebastiao I B  Alexeenko A}
		Andrew Strongrich, Alexander Pikus, Inderjit B. Sebastiao, Alina Alexeenko. (2017). Microscale in-plane knudsen radiometric actuator: Design, characterization, and performance modeling. Journal of Microelectromechanical Systems, 26(3), 528-538.
		
		\bibitem{Su W Zhu L Wang P Zhang Y W L}
		Wei Su, Zhu, Liang Zhu, Peng Wang, Yonghao Zhang, Lei Wu. (2020). Can we find steady-state solutions to multiscale rarefied gas flows within dozens of iterations?. Journal of Computational Physics, 407, 109245.
		
		\bibitem{Tang G H Zhang Y H Gu X J Barber R W  Emerson D R 2009}
		Gui Hua Tang, Yonghao Zhang, Xiao Jun Gu, Robert W. Barber, David R.  (2009). Lattice Boltzmann model for thermal transpiration. Physical Review E—Statistical, Nonlinear, and Soft Matter Physics, 79(2), 027701.
		
		\bibitem{Takaishi T Sensui Y 1963}
		Tetsuo Takaishi, Yoshihiko Sensui. (1963). Thermal transpiration effect of hydrogen, rare gases and methane. Transactions of the Faraday Society, 59, 2503-2514.
		
		\bibitem{Takata S Funagane H 2011}
		Shigeru Takata, Hitoshi Funagane. (2011). Poiseuille and thermal transpiration flows of a highly rarefied gas: over-concentration in the velocity distribution function. Journal of fluid mechanics, 669, 242-259.
		
		
		\bibitem{Tcheremissine F 2005 May}
		Felix G. Tcheremissine. (2005, May). Direct numerical solution of the Boltzmann equation. In AIP Conference Proceedings (Vol. 762, No. 1, pp. 677-685). American Institute of Physics.
		
		
		\bibitem{Ukai Seiji Solutions of the Boltzmann equation}
		Seiji Ukai. "Solutions of the Boltzmann equation." Studies in Mathematics and its Applications. Vol. 18. Elsevier, 1986. 37-96.
		
		
		\bibitem{Vargo S E Muntz E P Shiflett G R Tang W C 1999}
		Stephen E. Vargo, E. Phillip Muntz, Gerald R. Shiflett, William C. Tang. (1999). Knudsen compressor as a micro-and macroscale vacuum pump without moving parts or fluids. Journal of Vacuum Science Technology A: Vacuum, Surfaces, and Films, 17(4), 2308-2313.
		
		\bibitem{Kuan-Hsiang Wang 2024}
		Kung-Chien Wu, Kuan-Hsiang Wang. H\" older regularity of solutions of the steady Boltzmann equation with soft potentials,
		\emph{} preprint arXiv:2409.12513.
		
		
		\bibitem{Wu L Reese J M Zhang Y 2014}
		Lei Wu, Reese, Jason M. Reese, Zhang, Yonghao Zhang. (2014). Solving the Boltzmann equation deterministically by the fast spectral method: application to gas microflows. Journal of Fluid Mechanics, 746, 53-84.
		
		
		
	\end{thebibliography}
	
\end{document}